\documentclass[reqno,10pt]{amsart}
\usepackage{etex}
\usepackage{float}
\usepackage{graphicx}
\usepackage{amsmath,amssymb}
\usepackage{hyperref}
\usepackage{amsaddr}
\usepackage{concmath}
\usepackage{color}

\usepackage{amsfonts}
\usepackage{epsfig}
\usepackage{graphics}
\usepackage{subfig}
\usepackage{dsfont}
\usepackage{amsthm}

\usepackage{geometry}
\DeclareMathOperator\supp{supp}

\usepackage{bold-extra}
\makeatletter
\renewcommand\section{\@startsection{section}{1}
  \z@{-0.8\linespacing\@plus-0.7\linespacing}{0.7\linespacing}
 {\large\scshape\center}}
\makeatother

\newtheoremstyle{thm}
  {2mm}   % ABOVESPACE
  {2mm}   % BELOWSPACE
  {\itshape}  % BODYFONT
  {0pt}       % INDENT (empty value is the same as 0pt)
  {\ttfamily\bfseries}  % HEADFONT
  {.}         % HEADPUNCT
  {5pt plus 1pt minus 1pt} % HEADSPACE
  {\thmname{#1}\thmnumber{ #2}\thmnote{ (#3)}} % CUSTOM-HEAD-SPEC
 \theoremstyle{thm}

\newtheorem{thm}{\texttt{\textbf{Theorem}}}[section]

\newtheorem{lem}[thm]{Lemma}

\theoremstyle{definition}

\theoremstyle{remark}
\newtheorem{rem}[thm]{Remark}
\numberwithin{equation}{section}
\newcommand{\eps}{\varepsilon}

\allowdisplaybreaks %for breaking long equations over two pages
\usepackage{tikz}
\usetikzlibrary{arrows.meta}
\usetikzlibrary{decorations.markings}
\tikzset{->-/.style={decoration={
  markings,
  mark=at position #1 with {\arrow{>}}},postaction={decorate}}}
  \tikzset{middlearrow/.style={
        decoration={markings,
            mark= at position 0.55 with {\arrow{#1}} ,
        },
        postaction={decorate}
    }
}
\usetikzlibrary{shapes.geometric}

\makeatletter
\@namedef{subjclassname@2020}{%
  \textup{2020} Mathematics Subject Classification}
\makeatother

\begin{document}
	
	\title[Higher-order nonlinear Schrödinger equation]{{\Large\ttfamily\bfseries L\MakeLowercase{ocal well-posedness of the higher-order nonlinear} S\MakeLowercase{chrödinger equation on the half-line:  single boundary condition case}}}%
	\author{Aykut Alkın$^{\MakeLowercase{a}}$, Dionyssios  Mantzavinos$^{\MakeLowercase{b}}$ and Türker Özsarı$^{\MakeLowercase{c}}$}%
	\address{\normalfont $^a$Department of Mathematics, Izmir Institute of Technology\\ Urla, Izmir 35430, Turkey}
	\address{\normalfont $^b$Department of Mathematics, University of Kansas\\ Lawrence, KS 66045, USA}
	\email{mantzavinos@ku.edu}
	\address{\normalfont $^c$Department of Mathematics, Bilkent University\\ Çankaya, Ankara 06800, Turkey
	\\[7mm]
{\itshape Dedicated to Professor Athanassios S. Fokas}
}
	\email{turker.ozsari@bilkent.edu.tr}

\thanks{\textit{Acknowledgements.} The authors would like to thank the Isaac Newton Institute for Mathematical Sciences, Cambridge, U.K. for support and hospitality during the program ``Dispersive Hydrodynamics'', when work on this paper was undertaken (EPSRC Grant Number EP/R014604/1). 
The second author gratefully acknowledges partial support from the U.S. National Science Foundation (NSF-DMS 2206270).}
\date{May 18, 2023}

% ----------------------------------------------------------------
\begin{abstract}
We establish local well-posedness in the sense of Hadamard  for a certain third-order nonlinear Schrödinger equation with a multi-term linear part and a general power nonlinearity, known as higher-order nonlinear Schr\"odinger equation, formulated on the half-line $\{x>0\}$. 
We consider the scenario of associated coefficients such that  only one boundary condition is required and hence assume a general nonhomogeneous boundary datum of Dirichlet type at $x=0$. 
Our functional framework centers around fractional Sobolev spaces $H_x^s(\mathbb{R}_+)$ with respect to the spatial variable. We treat both high regularity ($s>\frac{1}{2}$) and low regularity ($s<\frac{1}{2}$) solutions: in the former setting, the relevant nonlinearity can be handled via the Banach algebra property; in the latter setting, however, this is no longer the case and, instead, delicate Strichartz estimates must be established. This task is especially challenging in the framework of nonhomogeneous initial-boundary value problems, as it involves proving boundary-type Strichartz estimates that are not common in the study of Cauchy (initial value) problems.  

The linear analysis, which forms the core of this work, crucially relies on a weak solution formulation defined through the novel solution formulae obtained via the Fokas method (also known as the unified transform) for the associated forced linear problem.  
In this connection, we note that the higher-order Schr\"odinger equation comes with an increased level of difficulty due to the presence of more than one spatial derivatives in the linear part of the equation. This feature manifests itself via several complications throughout the analysis, including  
(i) analyticity issues related to complex square roots, which require careful treatment of branch cuts and deformations of integration contours; 
(ii) singularities that emerge upon changes of variables in the Fourier analysis arguments; 
(iii) complicated oscillatory kernels in the weak solution formula for the linear initial-boundary value problem, which require a subtle analysis of the dispersion in terms of the regularity of the boundary data. 
The present work provides a first, complete treatment via the Fokas method of a nonhomogeneous initial-boundary value problem for a partial differential equation associated with a multi-term linear differential operator.
\end{abstract}

\vspace*{-1cm}
\maketitle
%\tableofcontents

% ----------------------------------------------------------------
\section{Introduction and main results}

\subsection{\ttfamily\bfseries Mathematical model}
We consider the nonhomogeneous initial-boundary value problem for the higher-order nonlinear Schrödinger (HNLS) equation  on the half line
\begin{equation}\label{nonlinear1}
\begin{aligned}
&iu_t+i\beta u_{xxx}+\alpha u_{xx}+i\delta u_x = f(u),  \quad (x,t)\in\mathbb{R}_{+} \times (0,T),
\\
&u(x,0) = u_0(x), \quad x\in \mathbb{R}_{+}, 
\\
&u(0,t) = g(t), \quad t\in (0,T),
\end{aligned}
\end{equation}
where $\alpha,\delta\in\mathbb{R}$, $\beta>0$, $f(z)=\kappa|z|^pz$ with $z\in\mathbb{C}$, $\kappa\in \mathbb{C}$, $p>0$, and $T>0$. 
The reason why we only need to supplement one boundary condition at $x=0$ is the assumption $\beta>0$. On the other hand, if $\beta<0$ then two boundary conditions are required at $x=0$; this scenario will be considered in a future work.
Furthermore, here we consider the case of a Dirichlet boundary datum; the case of a Neumann datum can be handled via entirely analogous ideas and techniques.

We establish local well-posedness of the nonlinear, nonhomogeneous initial-boundary value problem ~\eqref{nonlinear1} in the sense of Hadamard, namely, we prove existence of a unique local-in-time solution that depends continuously on the initial and boundary data (see Theorems \ref{HighRegThm} and \ref{LowRegThm} below). 
It will be shown (see Theorem \ref{cauchylemma} below) that the evolution operator associated with the free higher-order Schrödinger operator enjoys the following regularity property: 
\begin{equation*}
	\left\|e^{(-\beta\partial_x^3+i\alpha\partial_x^2-\delta\partial_x)t}\right\|_{H^s_x(\mathbb{R})\rightarrow L^\infty_x(\mathbb{R};H_t^{\frac{s+1}{3}}(-T,T))}\le c<\infty.
\end{equation*}
Note that, in the case of the classical second-order Schr\"odinger operator (i.e. for $\alpha=1$ and $\beta=\delta=0$), the time regularity of the solution is described by the Sobolev exponent $\frac{2s+1}{4}$ (see \cite{fhm2017}). Hence, for $s<\frac 12$ (which implies $\frac{2s+1}{4}<\frac{s+1}{3}$) the above result for the higher-order Schr\"odinger operator can be regarded as a kind of smoothing. 
Due to this smoothing, one anticipates that the local well-posedness for the initial-boundary value problem \eqref{nonlinear1}  should be studied with initial data $u_0\in H^s_x(\mathbb{R}_+)$ and boundary data $g\in H_{t,\textnormal{loc}}^{\frac{s+1}{3}}(\mathbb{R}_+)$.  In addition, for $s$ large enough and, more precisely, for $s>\frac 12$, the relevant traces make sense in the aforementioned spaces and one has to also impose compatibility conditions between the initial and the boundary data in order to obtain solutions that are continuous at $t=0$ (see \ref{SecRedibvp} for more details).  

Our treatment of the nonlinear problem \eqref{nonlinear1} is crucially based on a contraction mapping argument applied to a weak solution formula  for the associated forced linear initial-boundary value problem. This novel solution formula is derived in Section \ref{linear-s} via the Fokas method (also known as the unified transform). Showing that the map obtained by replacing the forcing with the power nonlinearity of \eqref{nonlinear1} in the Fokas method solution is a contraction requires further assumptions on the smoothness and growth of the nonlinearity in relation to the Sobolev regularity exponent $s$; see \eqref{sprelation1}  and \eqref{assmponsandp} for the precise assumptions used in this work.

\subsection{\ttfamily\bfseries Physical significance and motivation}
For $\alpha=1$, $\beta=\delta=0$ and $p=2$, the HNLS equation in \eqref{nonlinear1} reduces to the celebrated cubic 
nonlinear Schrödinger equation (NLS) equation. 
Cubic NLS is a ubiquitous model in mathematical physics, with a broad spectrum of applications ranging from optics to water waves to  plasmas to Bose-Einstein condensates. 
However, for pulses in the femtosecond regime, the classical NLS model is not precise enough and a correction involving a higher-order dispersive term is necessary (see \cite{agrawal} for a detailed discussion of the higher-order effects upon the propagation of an optical pulse). 
This need for a more accurate model led to the introduction of the HNLS equation, originally in the form
\begin{equation} \label{hnls_or}
	i u_t + \frac{1}{2} u_{xx} + |u|^2u + i \epsilon  \left(\beta_1 u_{xxx} + \beta_2 (|u|^2u)_x + \beta_3 u |u|^2_x\right) = 0, 
\end{equation}
for modeling the femtosecond pulse propagation in nonlinear optical fibers \cite{koda85,koda87}. 
Note that in \eqref{nonlinear1} we have a power nonlinearity of general order $p>0$, while \eqref{hnls_or}  involves cubic nonlinearities which, nevertheless, contain derivatives.

Furthermore, beyond physical considerations, it should be noted that the cubic NLS equation is a prototypical example of a completely integrable system \cite{zs1972}. As such, in addition to analysis techniques, it has been studied extensively via the inverse scattering transform and related methods. 
This is not possible, however, for general nonlinearities like the one of the HNLS equation in \eqref{nonlinear1}, as the corresponding models are not completely integrable. In those general cases, which nevertheless remain very relevant when it comes to applications, rigorous results can be established via harmonic analysis techniques.
In particular, the well-posedness of the Cauchy (initial value) problem for HNLS has been treated in a number of articles  \cite{Car2003,Carvajal04,Carvajal06,staffilani2,Laurey97,staffilani1,Taka00, Fam22}.  In addition, numerical solutions were obtained in \cite{Caval19}. Moreover, there are some results concerning the controllability properties of this equation; see \cite{Ceballos05} for exact boundary controllability, \cite{Bis07,chen2018} for internal feedback stabilization, and \cite{batal2020,OY2022} for boundary feedback stabilization.

The goal of this work is to establish the local well-posedness theory for the nonlinear initial-boundary value problem \eqref{nonlinear1} at the level of $H^s_x(\mathbb{R}_+)$ spatial regularity for the initial data.  We are interested in both the high regularity ($s>\frac 12$) and the low regularity ($s<\frac 12$) settings.  The main distinction between the two is that, in the low regularity setting, the well-known Banach algebra property of $H^s_x(\mathbb{R}_+)$ is no longer available. Instead, handling the nonlinearity $|u|^p u$ when $s<\frac 12$ requires use of more advanced tools that revolve around the celebrated Strichartz estimates.  Estimates of this type measure the size and temporal decay of solutions in space-time Lebesgue norms and have played a crucial role in the treatment of the Cauchy problem of nonlinear dispersive equations since their introduction in 1977 \cite{s1977}.  On the other hand, the use of Strichartz estimates in the analysis of initial-boundary value problems is a more recent advancement.  For the Cauchy problem, Strichartz estimates involve certain norms of initial and/or interior data, while for initial-boundary value problems these estimates additionally depend on  information related to boundary data, for which temporal regularity also plays a key role.  Such boundary-type Strichartz estimates have been recently established for some initial-boundary value problems associated with dispersive equations, see for instance \cite{h2005,Bona18,KO22} for the one-dimensional NLS, \cite{hm2020,Ran18} for the two-dimensional NLS, \cite{Aud19,Hay21D,Hay21N} for NLS in $n$ dimensions,  \cite{OY19,Cap20} for the one-dimensional biharmonic NLS, and \cite{OAK22} for a fourth-order NLS in one dimension.

\subsection{\ttfamily\bfseries Challenges, methodology, and main results}
The first contribution of the present paper is the development of a sharp linear theory through the analysis of the solutions of the relevant forced linear initial-boundary value problem (see problem \eqref{linear1} below). This is accomplished by decomposing this linear problem into three simpler component problems: (i) a homogeneous Cauchy problem  associated with (an appropriate extension of) the initial datum; (ii) a nonhomogeneous Cauchy problem associated with (an appropriate extension of) the forcing; (iii) a reduced initial-boundary value problem involving the original boundary datum and the spatial traces of the two aforementioned Cauchy problems.

The homogeneous Cauchy problem of item (i) is studied in Section \ref{SecHomCauchy} via classical Fourier analysis.  However, the multi-term nature of the spatial differential operator introduces certain difficulties in the proofs of the temporal estimates, due to the changes of variables performed in order to extract the desired Sobolev norms.  These difficulties  are overcome by introducing a proper cut-off function that depends on the polynomial structure of the spatial differential operator.

The nonhomogeneous Cauchy problem of item (ii) is analyzed in Section \ref{SecNonHomCauchy} by expressing its solution in Duhamel form and then estimating it via the nonlocal (i.e. in the physical space) definition of fractional Sobolev spaces.   We note that, when the equation involves a multi-term spatial differential operator, the analysis of the corresponding nonhomogeneous Cauchy problem becomes quite involved when carried out via other approaches such as the Riemann-Liouville fractional integral method that was successfully applied to the Korteweg-de Vries equation (without the first-order derivative) and the NLS equation \cite{ck2002,h2005,h2006}.  In contrast, the physical space definition of the fractional Sobolev spaces offers a more robust and direct approach in this framework; furthermore, this approach does not require interpolation arguments.

A major emphasis in this work is placed on the regularity analysis of the solution to the reduced initial-boundary value problem of item (iii) above.  This is done in Section \ref{SecRedibvp}. Weak solutions of this reduced initial-boundary value problem are defined via a boundary integral operator whose explicit form is obtained through the Fokas method  \cite{Fokas97,Fokasbook}.  Importantly, in the multi-term framework considered in this paper, certain analyticity issues arise in the application of the Fokas method. This is because the method relies on the construction of analytic maps that respect certain spectral invariance properties of the linear dispersion relation.  However, for multi-term spatial differential operators, such a construction requires use of complex square root functions which, in many cases, cause the invariance maps to be non-analytic on some parts of the complex spectral plane. We handle this complication via suitable contour deformations around the branch cuts associated with these maps. It is worth noting that this phenomenon also appears in the context of higher-dimensional initial-boundary value problems, see for instance \cite{BOF2020}, as well as in equations that involve higher-order time derivatives, e.g. the ``good'' Boussinesq equation analyzed in \cite{hm2015l,hm2015}.  After constructing a suitable boundary integral operator for the reduced initial-boundary value problem, we analyze it by using the oscillatory integral theory which, in particular, requires us to establish dispersive estimates of the same type like the ones satisfied by solutions of the associated Cauchy problem.

The solution of the fully nonlinear problem \eqref{nonlinear1} will be constructed as a fixed point of the solution operator formed by reunifying the respective solution formulae for the three linear problems of  items (i)-(iii) above.  In the high regularity setting of $s>\frac 12$, the spatiotemporal estimates established in the linear theory of Section \ref{linear-s} lead to a contraction mapping argument in the Hadamard-type space $C([0,T];H^s_x(\mathbb{R}_+))$.  The uniqueness in this space utilizes the Sobolev embedding $H^s_x(\mathbb{R}_+)\hookrightarrow L^\infty_x(\mathbb{R}_+)$ (which is valid for $s>\frac12$).  In the low regularity setting of $s<\frac 12$,  the algebra property in $H^s_x(\mathbb{R}_+)$ and the embedding $H^s_x(\mathbb{R}_+)\hookrightarrow L^\infty_x(\mathbb{R}_+)$ are no longer valid and Strichartz estimates assume the key role instead.  In that case, the solution space is refined to $C([0,T];H^s_x(\mathbb{R}_+))\cap L_t^{\mu}((0, T); H_x^{s,r}(\mathbb{R}_+))$ with $(\mu,r)$ obeying the admissibility criterion \eqref{admissiblepair} associated with the underlying evolution operator.  However, this only leads to a conditional uniqueness result in the aforementioned space.

The main results of this work, which emanate from the analysis described above, establish the local well-posedness of the HNLS initial-boundary value problem \eqref{nonlinear1} in the high and low regularity settings and read as follows:
\begin{thm}[High regularity well-posedness]\label{HighRegThm}
Let $\frac 12 < s \leq 2$ and $p>0$. In addition, if $p \notin 2\mathbb Z$, suppose that
\begin{equation}\label{sprelation1}
\begin{aligned}
&\text{if } s\in\mathbb{Z}_+, \text{then } p\geq s \text{ if }p\in\mathbb{Z}_+ \text{ and odd}; \lfloor p\rfloor\geq s-1 \text{ if }p\notin \mathbb{Z}_+,
\\
&\text{if } s\notin\mathbb{Z}_+, \text{then } p>s \text{ if }p \in\mathbb{Z}_+ \text{ and odd};\lfloor p\rfloor\geq \lfloor s\rfloor \text{ if }p\notin \mathbb{Z}_+.
\end{aligned}
\end{equation}
Then, for initial data $u_0\in H_x^s(\mathbb{R}_+)$ and boundary data $g\in H_{t,\textnormal{loc}}^{\frac{s+1}{3}}(\mathbb{R}_+)$ satisfying the compatibility condition \eqref{comp-cond}, there is $T=T(u_0,g)>0$ such that the initial-boundary value problem \eqref{nonlinear1} for the HNLS equation on the half-line has a unique solution $u\in C([0,T];H_x^s(\mathbb{R}_+))$. Furthermore, this solution depends continuously on the initial and boundary data.
\end{thm}	
\begin{thm}[Low regularity well-posedness]
\label{LowRegThm}
Suppose 
\begin{equation}\label{assmponsandp}
0 \leq s < \frac12, \quad 1\le p\leq \frac{6}{1-2s}, 
\quad
\mu=\frac{6(p+1)}{p(1-2s)}, \quad r=\frac{2(p+1)}{1+2sp}.
\end{equation}
Then, for initial data $u_0\in H_x^s(\mathbb{R}_+)$ and boundary data $g\in H_{t,\textnormal{loc}}^{\frac{s+1}{3}}(\mathbb{R}_+)$, with the additional assumption that if $p=\frac{6}{1-2s}$ (critical case) then $\|u_0\|_{H^s_x(\mathbb{R}_+)}$ is sufficiently small, there is $T=T(u_0,g)>0$ such that the initial-boundary value problem \eqref{nonlinear1} for the HNLS equation on the half-line has a unique solution $u\in C([0,T];H^s_x(\mathbb{R}_+))\cap L^\mu_t((0, T); H_x^{s,r}(\mathbb{R}_+))$. Furthermore, this solution depends continuously on the initial and boundary data.
\end{thm}
The linear estimates that form the core of the proofs of the above two theorems are established in Section \ref{linear-s}, and the contraction mapping arguments that complete those proofs are provided in Section~\ref{nonlinear-s}.

\subsection{$\texttt{The Fokas method for the rigorous treatment of initial-boundary value problems}$}
While the Cauchy problem for nonlinear dispersive equations has been  broadly explored through a variety of techniques, progress towards the rigorous study of initial-boundary value problems for these equations is more limited. In fact, problems of this latter kind can present significant challenges even at the linear level. For example, while on the whole line linear evolution equations can be easily solved via Fourier transform in the space variable, on domains with a boundary like the half-line no classical spatial transform exists for linear equations of spatial order three or higher. Another important obstacle arises in  the case of  boundary conditions that are non-separable. Moreover, even when a linear initial-boundary value problem can be solved via classical techniques, the resulting solution formula is not always useful, especially in regard to setting up an effective iteration scheme for proving the well-posedness of associated nonlinear problems. 

At the linear level, the Fokas method bridges the gap between the Cauchy problem and initial-boundary value problems by providing the direct analogue of the Fourier transform in domains with a boundary. Indeed, the method provides a fundamentally novel, algorithmic way of solving any linear evolution equation formulated on a variety of domains in one or higher dimensions and supplemented with any kind of admissible boundary conditions. 
An alternative perspective that further establishes the Fokas method as the natural counterpart of the Fourier transform in the context of linear initial-boundary value problems stems from the nonlinear component of the method, which was developed for completely integrable nonlinear equations and corresponds to the analogue of the inverse scattering transform in domains with a boundary. Then, noting that the linear limit of the inverse scattering transform is nothing but the Fourier transform, it is only reasonable that the linear limit of the nonlinear component of the Fokas method, namely the linear Fokas method, should provide the equivalent  of the Fourier transform for linear initial-boundary value problems.

The analogy between the Fokas method and the Fourier transform has been solidified by a new approach introduced in recent years by Himonas and one of the authors for the well-posedness of nonlinear initial-boundary value problems. This approach is based on treating the nonlinear problem as a perturbation of its forced linear counterpart, which is of course a classical idea coming from the Cauchy problem. However, the linear formulae produced via the Fourier transform in the case of the Cauchy problem are now replaced by the Fokas method solution formulae (recall that Fourier transform is no longer available). As these novel formulae involve complex contours of integration, new tools and techniques are required in order to obtain the various linear estimates needed for the contraction mapping argument. It should be noted that several of these estimates are specific to initial-boundary value problems and do not typically arise in the study of the Cauchy problem; they are results of particular importance, as they capture the effect of the boundary conditions on the regularity of the solution of both linear and nonlinear problems.

The Fokas method based approach for the rigorous study of initial-boundary value problems has already been implemented in several works in the literature: NLS on the half-line and the half-plane \cite{fhm2017,hmy2019-nls,hm2021,hm2020,hm2022,KO22}, KdV on the half-line and the finite interval \cite{fhm2016,hmy2019-kdv,hmy2021,hy2022a}, ``good'' Boussinesq on the half-line \cite{hm2015}, biharmonic NLS on the half-line \cite{OY19}, fourth order Schrödinger equation on the half-line \cite{OAK22}, a higher-dispersion KdV on the half-line \cite{hy2022b}, and even non-dispersive models  \cite{hmy2019-rd,cm2022}.  It should be noted that, in those cases where a problem has been previously considered in the literature, the results via the new method are consistent with the existing ones, typically obtained via the Colliander-Kenig-Holmer or the Bona-Sun-Zhang approaches, e.g. see  \cite{bsz2002,ck2002,h2005,h2006,et2016,BO2016,ct2017,Bona18,Ran18,Aud19,OY2022}.
Finally, we remark that a certain third-order model with cubic nonlinearity which is related to equations \eqref{nonlinear1} and \eqref{hnls_or} and is known as the Hirota equation has also been considered in the literature in the context of initial-boundary value problems, see \cite{h2020,gw2021,gw2023}. However, it is important to emphasize that in the present work we treat the case of a general power nonlinearity and, furthermore, we study the associated boundary integral operator at the low regularity level of Strichartz estimates.

We conclude by noting that rigorous treatment of initial-boundary value problems through the Fokas formulae has not only played an important role in establishing well-posedness results; it has also given insight towards solving problems that stem from other related fields such as systems theory and control, since there is a close connection between regularity theory and controllability. It is well-known that initial-boundary value problems with nonhomogeneous boundary conditions can be used to model physical evolutions in which the boundary input acts as control. Such boundary control models are particularly important for governing dynamics of physical processes in which access to the interior of a medium is blocked or not feasible while manipulations through the boundary remain an efficient choice.  See, for instance, \cite{KalOz2020, KalOz2023,KalOz2023-2} for some recent applications of the Fokas method to boundary control problems related to the heat and Schrödinger equations.
	
\section{Linear theory}
\label{linear-s}
% ----------------------------------------------------------------
In this section, we study the forced linear initial-boundary value problem
\begin{equation}\label{linear1}
\begin{aligned}
&iu_t+i\beta u_{xxx}+\alpha u_{xx}+i\delta u_x = f,  \quad (x,t)\in\mathbb{R}_{+} \times (0,T),
\\
&u(x,0) = u_0(x), \quad x\in \mathbb{R}_{+}, 
\\
&u(0,t) = g(t), \quad t\in (0,T),
\end{aligned}
\end{equation}
where $\alpha,\delta\in\mathbb{R}$ and $\beta>0$.  The analysis of the linear problem \eqref{linear1} will be carried out via a \emph{decomposition-reunification} approach.  This decomposition allows us to split the problem into three simpler components, two of which are Cauchy problems on the real line with data associated with $u_0$ and $f$, respectively, and one of which is a (reduced) initial-boundary value problem with data associated with $g$ as well as the traces of the solutions of the aforementioned Cauchy problems at $x=0$.	

\subsection{\ttfamily\bfseries Homogeneous linear Cauchy problem}\label{SecHomCauchy}
Consider the problem
\begin{equation}\label{cauchy}
\begin{aligned}
&iy_t+i\beta y_{xxx}+\alpha y_{xx}+i\delta y_x = 0,  \quad (x,t)\in\mathbb{R} \times \mathbb{R},
\\
&y(x,0) = y_0(x), \quad x\in \mathbb{R},
\end{aligned}
\end{equation}
where $y_0=E_0u_0$ denotes an extension of $u_0$ with respect to a fixed bounded extension operator $E_0: H_x^{s}(\mathbb{R}_+)\rightarrow H_x^s(\mathbb{R})$, namely we have
\begin{equation}\label{e0-def}
\|E_0u_0\|_{H_x^s(\mathbb{R})}\lesssim \|u_0\|_{H_x^s(\mathbb{R}_+)}.
\end{equation}
\begin{thm}\label{cauchylemma}
Let $s\in\mathbb R$. The unique solution of the Cauchy problem \eqref{cauchy}, denoted by $y=S[y_0;0]$, belongs to $C(\mathbb{R}_t;H_x^s(\mathbb{R}))$ and satisfies the conservation law
		\begin{equation}\label{cuachyspaceest}
			\left\| y(\cdot ,t) \right\|_{H_x^s(\mathbb{R})}= \left\| y_0 \right\|_{H_x^s(\mathbb{R})},\quad t\in\mathbb{R}.
		\end{equation}
		Moreover, if  $\alpha^2+3\beta\delta\ge 0$, then $y\in C(\mathbb{R}_x;H_t^{\frac{s+1}{3}}(-T,T))$ for $T>0$ and there exists a constant $c=c(s,\alpha,\beta,\delta)\geq 0$ such that
		\begin{equation}\label{cauchyextra3}
			\sup_{x\in\mathbb{R}}\left\|y(x,\cdot)\right\|_{H_t^{\frac{s+1}{3}}(-T,T)}\leq c(1+T^{\frac12}) \left\|y_0\right\|_{H_x^s(\mathbb{R})},
		\end{equation}
while if $\alpha^2+3\beta\delta<0,$ then $y\in C(\mathbb{R}_x;H_t^{\frac{s+1}{3}}(\mathbb{R}))$ and there is a constant $c=c(s,\alpha,\beta,\delta)\geq 0$ such that
		\begin{equation}\label{cauchyextra3*}
			\sup_{x\in\mathbb{R}}\left\|y(x,\cdot)\right\|_{H_t^{\frac{s+1}{3}}(\mathbb{R})}\leq c \left\|y_0\right\|_{H_x^s(\mathbb{R})}.
		\end{equation}
	\end{thm}
	\begin{proof}
		Taking the Fourier transform of \eqref{cauchy} with respect to $x$,  we find $\widehat{y}(k,t)=e^{-\omega(k)t} \, \widehat{y}_0(k)$ where
		\begin{equation}\label{wkdef}
			\omega(k):=-i\beta k^3+i \alpha k^2+i\delta k.
		\end{equation}
		 For $k\in\mathbb R$, $\omega(k)$ is purely imaginary thus $|\widehat{y}(k,t)|=|\widehat{y}_0(k)|$ and the conservation law \eqref{cuachyspaceest} readily follows via Plancherel's theorem.  The continuity of the map $t\mapsto y(t)$ from $[0,T]$ into $H_x^s(\mathbb{R})$ follows from the dominated convergence theorem and the fact that $y_0\in H_x^s(\mathbb{R})$. 
		 
In order to prove the temporal estimates \eqref{cauchyextra3} and \eqref{cauchyextra3*}, we start from the Fourier transform solution representation 
		\begin{equation}\label{c1}
			y(x,t)=S[y_0;0](x,t)=\frac{1}{2\pi}\int_{\mathbb{R}}e^{ikx-\omega(k)t}\, \widehat{y}_0(k) dk.
		\end{equation}
		Consider the real-valued map $\tau=i\omega(k)$. Notice that if $\alpha^2+3\beta\delta\le 0$ then $\tau$ is monotone increasing and so $k=(i\omega)^{-1}(\tau)$ is well-defined. In the case of strict inequality $\alpha^2+3\beta\delta<0$, we observe that $i\omega'(k)>-(\alpha^2+3\beta \delta)/(3\beta)>0$ and so by the inverse function theorem we can change variable from $k$ to $\tau$ to rewrite \eqref{c1} as
		\begin{equation}\label{c1new}
			y(x,t)=\frac{1}{2\pi}\int_{\mathbb{R}}e^{i(i\omega)^{-1}(\tau)x+i\tau t} \, \widehat{y}_0((i\omega)^{-1}(\tau))\;\frac{d\tau}{i\omega'((i\omega)^{-1}(\tau))}.
		\end{equation}
		In addition, we have $i\omega(k)=\mathcal{O}(k^3)$ and $\frac{1}{i\omega'(k)}=\mathcal{O}(k^{-2})$ as $|k|\rightarrow \infty$. Using the Fourier transform characterization of the Sobolev norm, for each $x\in \mathbb{R}$ we find 
\begin{equation*}
\begin{aligned}
\left\|y(x,\cdot)\right\|_{H_t^{\frac{s+1}{3}}(\mathbb{R})}^2
&= 
\int_{\mathbb{R}}(1+\tau^2)^{\frac{s+1}{3}}|\widehat{y}_0((i\omega)^{-1}(\tau))|^2\;\frac{d\tau}{|i\omega'((i\omega)^{-1}(\tau))|^2}
\\
&\lesssim \int_{\mathbb{R}}\left(1+k^2\right)^{s}|\widehat{y}_0(k)|^2 dk = \left\| y_0 \right\|_{H_x^s(\mathbb R)}^2
\end{aligned}
\end{equation*}
which amounts to estimate \eqref{cauchyextra3*}.  

Next, consider the case $\alpha^2+3\beta\delta\ge 0$.  Let $\theta\in C_c^\infty(\mathbb{R})$ be a function whose additional properties will be specified below. Then, we can write $y=y_1+y_2$, where
\begin{equation}
\begin{aligned}
y_1(x,t)&:=\frac{1}{2\pi}\int_{\mathbb{R}}e^{ikx-\omega(k)t}\theta(k)\widehat{y}_0(k) dk,
\\
y_2(x,t)&:=\frac{1}{2\pi}\int_{\mathbb{R}}e^{ikx-\omega(k)t}(1-\theta(k))\widehat{y}_0(k) dk.
\end{aligned}
\end{equation}
		Taking $j$-th order time derivative of $y_1$ and using Cauchy-Schwarz inequality, we deduce  
\begin{align*}
|\partial_t^j y_1(x,t)|&\leq\frac{1}{2\pi}\int_{\supp(\theta)}|\omega(k)|^{j}|\theta(k)\|\widehat{y}_0(k)|dk
\\
&\lesssim\bigg( \int_{\supp(\theta)}\left(1+k^2\right)^{-s}|\omega(k)|^{2j}dk \bigg)^{\frac 12} \left\|y_0\right\|_{H_x^s(\mathbb{R})}
=
c(s,j,\theta)\left\|y_0\right\|_{H_x^s(\mathbb{R})}.
\end{align*}
We note that this inequality holds for any $s\in\mathbb R$.
Thus, by the physical space characterization of the Sobolev norm, namely
\begin{equation}
\left\| f \right\|_{H_t^{\mu}(-T,T)} = \sum_{j=0}^\mu \big\| \partial_t^j f \big\|_{L_t^{2}(-T,T)}, \quad \mu\in\mathbb N_0,
\end{equation}
 we obtain
\begin{equation}\label{cauchyextra10}
\left\| y_1(x,\cdot) \right\|_{H_t^\mu(-T,T)}\leq c(s,\mu,\theta) T^{\frac12}\left\|y_0\right\|_{H_x^s(\mathbb{R})}
\end{equation}
for any $\mu\in\mathbb N_0$ and any $x, s\in\mathbb R$. 
Then, since given any $m \in \mathbb R$ we can always find $\mu \in \mathbb N \cup\{0\}$ such that $m\leq \mu$,  estimate \eqref{cauchyextra10} readily implies
\begin{equation}\label{cauchyextra1}
\left\| y_1(x,\cdot) \right\|_{H_t^m(-T,T)}\leq c(s,m,\theta) T^{\frac12}\left\|y_0\right\|_{H_x^s(\mathbb{R})}, \quad m,s,x \in \mathbb R.
\end{equation}

In order to handle $y_2$, we note that given $\alpha, \delta\in\mathbb{R}$ and $\beta>0$ satisfying $\alpha^2+3\beta\delta\ge 0$ one can find $k_j=k_j(\alpha,\delta,\beta)\in \mathbb{R}$, $j=1,2$, such that
(i) the roots $\frac{\alpha \pm \sqrt{\alpha^2+3\beta\delta}}{3\beta}$ of $\omega'(k)=0$ lie in $(k_1, k_2)$ and (ii) the mapping $\tau=i\omega(k)$ is monotone increasing on $\mathbb{R}\setminus (k_1, k_2)$.
Now, let $k_3<k_1$ and $k_4>k_2$ be any two numbers and fix $\theta$ so that it further satisfies the condition
$$
\theta(k)=\left\{
\begin{array}{ll}
1, & k\in  [k_1,k_2], 
\\
0, & k\notin (k_3,k_4),
\end{array}
\right.
$$
as well as the condition $0\le |\theta(k)|\le 1$, $k\in \mathbb{R}.$ 
Now, we can rewrite $y_2$ as 
\begin{align*}
			y_2(x,t)&=\frac{1}{2\pi}\int_{\mathbb R \setminus [k_1, k_2]} e^{ikx-\omega(k)t}(1-\theta(k))\widehat{y}_0(k) dk
			\\
			&=\frac{1}{2\pi}\int_{(i\omega)(\mathbb R \setminus [k_1, k_2])}e^{i(i\omega)^{-1}(\tau)x+i\tau t}(1-\theta((i\omega)^{-1}(\tau)))\widehat{y}_0((i\omega)^{-1}(\tau))\;\frac{d\tau}{i\omega'((i\omega)^{-1}(\tau))}.
\end{align*}
Using the definition of the Sobolev norm, for each $x\in \mathbb{R}$ we have %
\begin{align}
\left\|y_2(x,\cdot)\right\|_{H_t^{\frac{s+1}{3}}(\mathbb{R})}^2
&=
\int_{\mathbb{R}}(1+\tau^2)^{\frac{s+1}{3}} |\widehat{y}_2(x,\tau)|^2 d\tau
\nonumber\\
&=\displaystyle\int_{(i\omega)(\mathbb R \setminus [k_1, k_2])}(1+\tau^2)^{\frac{s+1}{3}}\frac{|1-\theta((i\omega)^{-1}(\tau))|^2|\widehat{y}_0((i\omega)^{-1}(\tau))|^2}{|i\omega'((i\omega)^{-1}(\tau))|^2} d\tau
\nonumber\\
&\lesssim \int_{\mathbb R \setminus [k_1, k_2]}\frac{(1+(i\omega(k))^2)^{\frac{s+1}{3}}}{|i\omega'(k)|} |\widehat{y}_0(k)|^2 dk
\nonumber\\
&\lesssim \int_{\mathbb R}\left(1+k^2\right)^{s}|\widehat{y}_0(k)|^2 dk
=
\left\| y_0 \right\|_{H_x^s(\mathbb R)}^2,
\label{cauchyextra2}
\end{align}
where the last inequality follows  from the fact that $i\omega(k)=\mathcal{O}(k^3)$ and $\frac{1}{i\omega'(k)}=\mathcal{O}(k^{-2})$ as $|k|\rightarrow \infty$. Hence, \eqref{cauchyextra3} follows from \eqref{cauchyextra1} and \eqref{cauchyextra2}. Continuity in $x$ once again follows from the dominated convergence theorem.
\end{proof}

Notice that the conservation law \eqref{cuachyspaceest} allows us to control the $L^\infty_t((0, T); H_x^s(\mathbb{R}))$ norm of the solution to the homogeneous linear  Cauchy problem  by the $H^s_x(\mathbb{R})$ norm of the initial data. As we shall show below, this is also the case for the mixed Lebesgue norms $L^\mu_t((0, T); H^{s,r}_x(\mathbb{R}))$, where $H^{s, r}(\mathbb{R})$ is the usual Bessel potential space defined with norm 
\begin{equation}\label{bessel-def}
\left\|f\right\|_{H^{s,r}(\mathbb{R})}
:=
\left\| \mathcal{F}^{-1}\left\{ \left(1+k^2\right)^{\frac{s}{2}} \mathcal F\{f\}(k)\right\}\right\|_{L^{r}(\mathbb{R})}
\end{equation}
and $(\mu,r)$ is any higher-order Schrödinger admissible pair, i.e. any pair $(\mu,r)$ satisfying
	\begin{equation}\label{admissiblepair}
		\mu,r\geq 2, \quad \frac{3}{\mu}+\frac{1}{r}=\frac12.
	\end{equation}
More precisely, we have the following Strichartz estimate:
	\begin{thm}\label{homStr}
Let $s\in\mathbb{R}$ and suppose $(\mu,r)$ is higher-order Schrödinger admissible in the sense of \eqref{admissiblepair}. Then, the solution of the homogenerous linear Cauchy problem \eqref{cauchy} satisfies the Strichartz estimate
		\begin{equation}\label{stry}
			\left\|y\right\|_{L_{t}^{\mu}((0, T); H_{x}^{s,r}(\mathbb{R}))}\lesssim \left\|y_0\right\|_{H_{x}^{s}(\mathbb{R})}.
		\end{equation}
	\end{thm}
	\begin{proof}By the definition \eqref{bessel-def} of the $H^{s,r}$-norm, we have
		\begin{equation*}
			\left\|y\right\|_{L_{t}^{\mu}(\mathbb{R};H_{x}^{s,r}(\mathbb{R}))}=\left\|  \mathcal{F}^{-1}\left\{\left(1+k^2\right)^{\frac{s}{2}}\widehat{y}^{(x)}(k,\cdot)\right\} \right\|_{L_{t}^{\mu}((0, T); L_{x}^{r}(\mathbb{R}))}.
		\end{equation*} 
Recalling that $\widehat{y}^{(x)}(k,t)=e^{-\omega(k)t} \, \widehat{y}_0(k)$, we have
		\begin{equation*}
			\mathcal{F}^{-1}\left\{\left(1+k^2\right)^{\frac{s}{2}}\widehat{y}^{(x)}(k,t) \right\} =\frac{1}{2\pi}\int_{-\infty}^\infty e^{ikx-\omega(k)t} \left(1+k^2\right)^{\frac{s}{2}}\widehat{y}_0(k)dk=S[\varphi; 0](t),
		\end{equation*} where $\varphi(x) = \mathcal F^{-1}\left\{\left(1+k^2\right)^{\frac{s}{2}}\widehat{y}_0(k)\right\}(x)$.  So, it suffices to prove that
		\begin{equation}\label{c2}
			\|S[\varphi; 0]\|_{L_{t}^{\mu}((0, T); L_{x}^{r}(\mathbb{R}))}\lesssim \left\| \varphi \right\|_{L_x^2(\mathbb{R})}.
		\end{equation}
For this, we note that by the definition of the Fourier transform we can write
		\begin{equation*}
			S[\varphi; 0](x,t)=\frac{1}{2\pi}\int_{\mathbb{R}} I(x,y,t)\varphi(y)dy,
\end{equation*}  
where 
$$
I(x,y,t):=\int_{\mathbb{R}} e^{ik(x-y-t\delta)+i(t\beta k^3-t\alpha k^2)}dk.
$$
Then, invoking the following dispersive estimate from the proof of Lemma 4.2 in \cite{Car2003},
 $$
 |I(x,y,t)|\lesssim |\beta t|^{-\frac{1}{3}}, \quad t\neq 0,
 $$
where the inequality constant is independent of $x, y, t$, and proceeding along the lines of  the proof of Theorem 4.1 in \cite{Car2003} (see also relevant proof in \cite{Fam22}), we infer the desired estimate \eqref{c2}.
	\end{proof}

\subsection{\ttfamily\bfseries Nonhomogeneous linear Cauchy problem}\label{SecNonHomCauchy} 
We continue our linear analysis with the problem
\begin{equation}\label{nonhomogeneous1}
\begin{aligned}
&iz_t+i\beta z_{xxx}+\alpha z_{xx}+i\delta z_x = F,  \quad (x,t)\in\mathbb{R} \times (0,T),
\\
&z(x,0) = 0, \quad x\in \mathbb{R},
\end{aligned}
\end{equation}
where $F=E_0f\in L_t^2((0,T); H_x^s(\mathbb{R}))$ is a spatial extension of $f\in L_t^2((0,T); H_x^s(\mathbb{R}_+))$.
Thanks to Duhamel's principle, the solution of the nonhomogeneous problem \eqref{nonhomogeneous1}, denoted by $S[0;F]$, can be expressed as
\begin{equation}\label{duhamel}
\begin{aligned}
z(x,t)=S[0;F](x,t)&=-i\int_{0}^{t} S[F(\cdot,t');0](x,t-t')dt'
\\
&=-\frac{i}{2\pi}\int_{0}^{t}\int_{\mathbb{R}} e^{ikx-\omega(k)(t-t')}\widehat{F}(k,t') dkdt',
\end{aligned}
\end{equation}
	where, for each $t'\in[0, t]$, $S[F(\cdot,t');0]$ denotes the solution to the homogeneous problem \eqref{cauchy} with initial data $F(x, t')$.
%%
%\begin{equation}\label{duhcauchy}
%		&&iy_t^{t'}+i\beta y_{xxx}^{t'}+\alpha y_{xx}^{t'}+i\delta y_x^{t'} = 0,  \quad (x,t)\in\mathbb{R} \times \mathbb{R},\\\label{ducauchy2}
%		&&y^{t'}(x,0) = F(x,t'), \quad x\in \mathbb{R}.
%	\end{equation}
%
We then have the following result, whose proof is based on the approach that was used for the Korteweg-de Vries equation in  \cite{hmy2019-kdv}.  
\begin{thm}\label{Nonhomthm} The unique solution of \eqref{nonhomogeneous1} satisfies the space estimate
		\begin{equation}\label{nonhthm1}
			\sup_{t\in[0,T]}\left\| z(\cdot,t) \right\|_{H_x^{s}(\mathbb{R})}\leq \left\|F\right\|_{L_t^1((0,T); H_x^{s}(\mathbb{R}))}, \quad s\in \mathbb{R}.
\end{equation}
Moreover, if  $-1\leq s \leq 2$ with $s\neq \frac 12$ then
		the following time estimate holds
		\begin{equation}\label{nonhthm2}
			\sup_{x\in\mathbb{R}}\left\| z(x,\cdot) \right\|_{H_t^{\frac{s+1}{3}}(0,T)}\lesssim \max\{T^{\frac12}(1+T^{\frac12}), T^{\sigma}\} \|F\|_{L_t^2((0,T); H_x^{s}(\mathbb{R}))},
		\end{equation}
where
		\begin{equation}\label{stildedef}
			\sigma=\left\{
			\def\arraystretch{1.5}
			\begin{array}{ll}
				\frac{1-2s}{6}, &-1\leq s<\frac12, 
				\\
				\frac{2-s}{3}, &\frac12< s< 2,
				\\
				\frac 12, &s=2.
%				\\
%				\frac{7-2s}{6}, &2\leq s<\frac72.
			\end{array}
			\right.
		\end{equation}
	\end{thm}

\begin{rem}
For $2<s<\frac 72$, due to the fractional norm $\left\|\partial_{t}z(x,\cdot)\right\|_{m-1}$ (see definition \eqref{mnorm} below) the analogue of the time estimate \eqref{nonhthm2} turns out to be
$$
\sup_{x\in\mathbb{R}}\left\| z(x,\cdot) \right\|_{H_t^{\frac{s+1}{3}}(0,T)}\lesssim \max\{T^{\frac12}(1+T^{\frac12}), T^{\sigma}\} \|F\|_{L_t^2((0,T); H_x^{s}(\mathbb{R}))}
+
\sup_{x\in\mathbb R} \|F(x, \cdot)\|_{H_t^{\frac{s+1}{3}-1}(0, T)}.
$$
The appearance of the space $C(\mathbb R_x; H^{\frac{s+1}{3}-1}(0, T))$ via the relevant norm on the right-hand side has a direct impact on the analysis of the nonlinear problem, as it eventually requires one to establish an appropriate multilinear estimate for the term $\| |u|^p u(x, \cdot) \|_{H_t^{\frac{s+1}{3}-1}(0, T)}$ (note that the underlying range of $s$ implies $0<\frac{s+1}{3}-1<\frac 12$ and so the algebra property is not available). For this reason, a different approach (perhaps via energy estimates) might be preferable for showing well-posedness in this higher range of $s$. In any case, this task lies outside the scope of the present work, which instead focuses on solutions of lower smoothness and, in particular, towards the low regularity setting $0\leq s<\frac 12$.
\end{rem}

	\begin{proof}  
	In view of the Duhamel representation \eqref{duhamel}, the space estimate \eqref{nonhthm1} readily follows from the homogeneous conservation law \eqref{cuachyspaceest}. 
	
We proceed to the time estimate \eqref{nonhthm2}. Restricting $s\ge -1$ allows us to employ the physical space characterization of the Sobolev norm since then the exponent $\frac{s+1}{3}$ is non-negative. In particular, for $-1 \leq s < 2$, setting $m := \frac{s+1}{3}$ and observing that $0 \leq m < 1$, we have
\begin{equation}\label{z-norm-1}
			\left\| z(x,\cdot) \right\|_{H_t^{m}(0,T)} = \left\| z(x,\cdot) \right\|_{L_t^{2}(0,T)} + \left\| z(x,\cdot) \right\|_{m},
\end{equation}
		where the fractional part of the Sobolev norm is zero for $m=0$ and for $0<m<1$  is given by
		\begin{equation}\label{mnorm}
			\left\| z(x,\cdot) \right\|_{m}^{2}=2\int_{0}^{T}\int_{0}^{T-t}\frac{|z(x,t+l)-z(x,t)|^2}{l^{1+2m}}dldt.
		\end{equation}
For each $x\in\mathbb{R}$, employing Minkowski's integral inequality and subsequently using the homogeneous time estimates  \eqref{cauchyextra3} and \eqref{cauchyextra3*} for $\alpha^2+3\beta\delta\ge 0$ and $\alpha^2+3\beta\delta< 0$ respectively, along with the Cauchy-Schwarz inequality, we obtain
\begin{equation}\label{nonhpf7}
\begin{aligned}
			\left\| z(x,\cdot) \right\|_{L_t^2(0,T)}
%			&=&\bigg( \int_{0}^{T} \bigg| \int_{0}^{t}S[F(\cdot,t');0](x,t-t')dt' \bigg|^2 dt \bigg)^{\frac12}\\\label{nonhpf9}
&\leq \int_{0}^{T}\left\| S[F(\cdot,t');0](x,\cdot-t') \right\|_{L_t^2(0,T)}dt'
\\
&\lesssim (1+T^{\frac 12})\int_{0}^{T}\left\|F(\cdot,t')\right\|_{H_{x}^{-1}(\mathbb{R})}dt'
\\
			&\lesssim T^{\frac 12} (1+T^{\frac 12}) \, \|F\|_{L_t^2((0,T); H_x^s{(\mathbb{R})})}.
\end{aligned}
\end{equation}

For the fractional norm,  noting that
		\begin{align}\nonumber
			\left|z(x,t+l)-z(x,t)\right|^2
			&\leq  \bigg| \int_{0}^{t}S[F(\cdot,t');0](x,t+l-t')-S[F(\cdot,t');0](x,t-t')dt' \bigg|^2\\\nonumber
			&\quad
			+\bigg| \int_{t}^{t+l}S[F(\cdot,t');0](x,t+l-t')dt'\bigg|^2
		\end{align}
we have $\left\| z(x,\cdot) \right\|_{m}^{2} \lesssim I+J$, where 
\begin{align}
&I := \int_{0}^{T}\int_{0}^{T-t} \frac{1}{l^{1+2m}} \bigg(\int_{0}^{T} \Big|S[F(\cdot,t');0](x,t+l-t')-S[F(\cdot,t');0](x,t-t')\Big|dt' \bigg)^2dldt,
\\
&J := \int_{0}^{T}\int_{0}^{T-t}\frac{1}{l^{1+2m}}
\bigg| \int_{t}^{t+l}S[F(\cdot,t');0](x,t+l-t')dt'\bigg|^2 dldt.
\end{align}

For $I$, we proceed as follows.  First, we multiply the integrand by the characteristic function $\chi_{[0, T-t]}(l)$ so that $\chi_{[0, T-t]}(l)=1$ for $0\le l\le T-t$ and $\chi_{[0, T-t]}(l)=0$ otherwise. This allows us to  replace $T-t$ by $T$ in the upper limit of the  integral with respect to $l$. Then, we use Minkowski's inequality for the triple integral and, finally, we use the definition of $\chi_{[0, T-t]}(l)$ once again to switch $T$ by $T-t$ in the limit of the integral taken with respect to $l$. Performing these steps and employing the homogeneous time estimates \eqref{cauchyextra3} and \eqref{cauchyextra3*}, we find
\begin{align}\label{nonhpf4}
			I&\leq\bigg(\int_{0}^{T}\bigg(\int_{0}^{T} \int_{0}^{T-t}\frac{1}{l^{1+2m}}\Big|S[F(\cdot,t');0](x,t+l-t')-S[F(\cdot,t');0](x,t-t')\Big|^2dldt \bigg)^{\frac12}dt'\bigg)^2
\\
&\simeq\bigg(\int_{0}^{T}\left\|S[F(\cdot,t');0](x,\cdot-t')\right\|_{m}\;dt'\bigg)^2
\lesssim\bigg(\int_{0}^{T}\left\|F(\cdot,t')\right\|_{H_{x}^{s}(\mathbb{R})}\;dt'\bigg)^2
\lesssim T\int_{0}^{T}\|F(\cdot,t)\|_{H_{x}^{s}(\mathbb{R})}^2 dt.
\nonumber
\end{align}
%				
%An alternative to arrive at the above estimate is to apply Minkowski's inequality twice to double integrals, first for each $t \in [0, T]$ to the inner integrals with respect to $t'$ and $l$, and then to the resulting outer integrals with respect to $t'$ and $t$, whose limits will no longer depend on $t$.

In order to estimate $J$, we consider the cases $0<m<\frac 12$ and $\frac 12<m<1$ separately. 
The range $\frac 12<m<1$ corresponds to $\frac 12<s<2$ and hence we can employ  the Sobolev embedding theorem in $x$. In particular, substituting for $S[F(\cdot,t');0](x,t+l-t')$ via \eqref{c1} and then using the Sobolev embedding, the Fourier transform characterization of the Sobolev norm, and the fact that $\omega(k)$ is imaginary for $k\in\mathbb R$, we have
\begin{align*}
J&\leq\int_{0}^{T}\int_{0}^{T-t}\frac{1}{l^{1+2m}}\left\| \frac{1}{2\pi}\int_{\mathbb{R}}e^{ikx-\omega(k)(t+l)} \int_{t}^{t+l}e^{\omega(k)t'}\widehat{F}(k,t')dt' dk \right\|_{L_{x}^{\infty}(\mathbb R)}^2dldt
\nonumber\\
&\lesssim\int_{0}^{T}\int_{0}^{T-t}\frac{1}{l^{1+2m}}\left\| \frac{1}{2\pi}\int_{\mathbb{R}}e^{ikx-\omega(k)(t+l)} \int_{t}^{t+l}e^{\omega(k)t'}\widehat{F}(k,t')dt' dk \right\|_{H_{x}^{s}(\mathbb R)}^2dldt
\nonumber\\
&=\int_{0}^{T}\int_{0}^{T-t}\frac{1}{l^{1+2m}}\int_{\mathbb{R}}\left(1+k^2\right)^{s}\bigg| \int_{t}^{t+l}e^{\omega(k)t'}\widehat{F}(k,t')dt' \bigg|^2dkdldt.
\end{align*}
Thus, by Minkowski's integral inequality between the integrals with respect to $t'$ and $k$, Cauchy-Schwarz inequality in the $t'$-integral, and Fubini's theorem between the integrals with respect to $t$ and $t'$, 
\begin{equation}\label{nonhpf5}
\begin{aligned}
J &\lesssim \int_{0}^{T}\int_{0}^{T-t}\frac{1}{l^{1+2m}}\bigg( \int_{t}^{t+l}\left\|F(\cdot,t')\right\|_{H_{x}^{s}(\mathbb R)}dt' \bigg)^2dldt
\\
&\leq
\int_{0}^{T}\int_{0}^{T-t}\int_{t}^{t+l} \left\|F(\cdot,t')\right\|_{H_{x}^{s}(\mathbb R)}^2\; l^{-2m}\;dt'dldt
\\
&=
\int_{0}^{T} \left\|F(\cdot,t')\right\|_{H_{x}^{s}(\mathbb R)}^2\int_{0}^{T}l^{-2m}\int_{t'-l}^{t'} dtdldt'
\\
&\simeq\frac{T^{2-2m}}{2-2m}\int_{0}^{T} \left\|F(\cdot,t')\right\|_{H_{x}^{s}(\mathbb R)}^2dt'.
\end{aligned}
\end{equation}
%		In the above sequence of estimates, we use Sobolev embedding, Minkowski's integral inequality in $k$-$t'$, Cauchy-Schwarz inequality in $t'$ and Fubini's theorem in $t$-$t'$.  
%		
The range $0<m<\frac 12$ corresponds to $-1<s<\frac{1}{2}$, hence Sobolev embedding is no longer available. However, the fact that $m<\frac 12$ allows to proceed  via the Cauchy-Schwarz inequality in $t'$ as follows:
\begin{equation}\label{nonhpf6}
\begin{aligned}
			J&\lesssim \int_{0}^{T}\frac{1}{l^{2m}}\int_{0}^{T-l}\int_{t}^{t+l}|S[F(\cdot,t');0](x,t+l-t')|^2dt'dtdl
\\
			&= \int_{0}^{T}\frac{1}{l^{2m}}\int_{l}^{T}\int_{t-l}^{t}|S[F(\cdot,t');0](x,t-t')|^2dt'dtdl
\\
			&\lesssim \bigg(\int_{0}^{T}\frac{1}{l^{2m}}dl\bigg)\int_{0}^{T}\left\|S[F(\cdot,t');0](x,\cdot-t')\right\|_{L_t^2(t',T)}^2dt'
\\
			&\lesssim \frac{T^{1-2m}}{1-2m}\int_{0}^{T}\left\|F(\cdot,t')\right\|_{H_{x}^{-1}(\mathbb R)}^2dt'.
		\end{aligned}
\end{equation}
Note that the equality above is due to the change of variable $t \mapsto t-l$, and the inequality succeeding it follows by extending the range of the  integrals with respect to $t'$ and $t$ and then interchanging the resulting integrals. The final inequality is thanks to Theorem \ref{cauchylemma}.

Estimates \eqref{nonhpf7}, \eqref{nonhpf4}, \eqref{nonhpf5}, \eqref{nonhpf6}   combined with the Sobolev norm definition \eqref{z-norm-1} imply the desired time estimate \eqref{nonhthm2} in the range $-1\leq s < 2$ with $s\neq \frac 12$. 

Finally, we consider $2\leq s<\frac72$. As this range corresponds to $1\leq m < \frac 32$, the Sobolev norm \eqref{z-norm-1} must be modified to
%
%Let us set $m+1:=\frac{s+1}{3}\in[1,\frac32)$, i.e., $0\leq m<\frac12$. We have
		\begin{equation*}
			\left\| z(x,\cdot) \right\|_{H_{t}^{m}(0,T)}^{2}=\left\| z(x,\cdot) \right\|_{H_t^1(0,T)}^{2}+\left\|\partial_{t}z(x,\cdot)\right\|_{m-1}^2.
		\end{equation*}
Differentiating \eqref{duhamel} in $t$, we have
		\begin{equation}\label{dtz}
			\partial_{t}z(x,t)=-iS[F(\cdot,t);0](x,0)-i\int_{0}^{t}\partial_{t}\big[ S[F(\cdot,t');0](x,t-t') \big]dt'.
		\end{equation}
We begin by observing that $S[F(\cdot,t);0](x,0)
%=\frac{1}{2\pi}\int_{\mathbb{R}}e^{ikx}\widehat{F}(k,t)dk
=F(x,t)$. 
%		\begin{equation*}
%			S[F(\cdot,t);0](x,0)=\frac{1}{2\pi}\int_{\mathbb{R}}e^{ikx}\widehat{F}(k,t)dk=F(x,t).
%		\end{equation*}
Moreover, by using the Fourier transform property for derivatives, 
%namely for a polynomial $P$,
%		\begin{equation*}
%			\mathcal{F}[P(\partial/\partial_{x})F(x)]=P(ik)\widehat{F}(k).
%		\end{equation*}
%
we note that
		\begin{equation}\label{fourierderivative}
		\begin{aligned}
			\partial_{t}\big[ S[F(\cdot,t');0](x,t-t') \big]&=\partial_{t}\left[\frac{1}{2\pi}\int_{\mathbb{R}}e^{ikx-\omega(k)(t-t')}\widehat{F}(k,t')dk \right]\\
			&=\frac{1}{2\pi}\int_{\mathbb{R}}[-\omega(k)]e^{ikx-\omega(k)(t-t')}\widehat{F}(k,t')dk\\ 
			&=S\big[ (-\beta\partial_{x}^{3}+i\alpha\partial_{x}^{2}-\delta\partial_{x})F(\cdot,t');0 \big](x,t-t').
			\end{aligned}
		\end{equation}
Therefore, \eqref{dtz} can be rewritten as
		\begin{equation}\label{dtz2}
			\partial_{t}z(x,t)=-iF(x,t)-i\int_{0}^{t}S\big[ (-\beta\partial_{x}^{3}+i\alpha\partial_{x}^{2}-\delta\partial_{x})F(\cdot,t');0 \big](x,t-t')dt'
		\end{equation}
and so
		\begin{equation*}
			\left\|\partial_{t}z(x,\cdot)\right\|_{L_t^2(0,T)} \leq \left\|F(x,\cdot)\right\|_{L_t^2(0,T)}+\left\| \int_{0}^{t}S\big[ (-\beta\partial_{x}^{3}+i\alpha\partial_{x}^{2}-\delta\partial_{x})F(\cdot,t');0 \big](x,t-t')dt'\right\|_{L_t^2(0,T)}.
		\end{equation*}
The first term on the right-hand side can be handled as follows:
\begin{equation}\label{zt-1}
\left\|F(x,\cdot)\right\|_{L_t^2(0,T)}
\leq
\sup_{x\in\mathbb R} \left\|F(x,\cdot)\right\|_{L_t^2(0,T)}
\leq
\left\|F\right\|_{L_t^2((0,T); H_x^{\frac 12+}(\mathbb R))}.
\end{equation}
For the second term, extending the range of integration in $t'$ and then applying Minkowski's integral inequality in combination with Theorem \ref{cauchylemma}, we have 
		\begin{align}\label{zt-2}
			&\left\| \int_{0}^{t}S\big[ (-\beta\partial_{x}^{3}+i\alpha\partial_{x}^{2}-\delta\partial_{x})F(\cdot,t');0 \big](x,t-t')dt'\right\|_{L_t^2(0,T)}
			\nonumber\\
			&\lesssim (1+T^{\frac 12}) \int_{0}^{T}\left\| (-\beta\partial_{x}^{3}+i\alpha\partial_{x}^{2}-\delta\partial_{x})F(\cdot,t') \right\|_{H_{x}^{-1}(\mathbb{R})}dt'
			\nonumber\\
			&\lesssim (1+T^{\frac 12}) \left[\beta\int_{0}^{T}\left\|\partial_{x}^{3}F(\cdot,t') \right\|_{H_{x}^{-1}(\mathbb{R})}dt'+|\alpha|\int_{0}^{T}\left\|\partial_{x}^{2}F(\cdot,t') \right\|_{H_{x}^{-1}(\mathbb{R})}dt'+|\delta|\int_{0}^{T}\left\|\partial_{x}F(\cdot,t') \right\|_{H_{x}^{-1}(\mathbb{R})}dt'\right]
			\nonumber\\
			&\lesssim(1+T^{\frac 12})
			\left[\beta\int_{0}^{T}\left\|F(\cdot,t') \right\|_{H_{x}^{2}(\mathbb{R})}dt'+|\alpha|\int_{0}^{T}\left\|F(\cdot,t') \right\|_{H_{x}^{1}(\mathbb{R})}dt'+|\delta|\int_{0}^{T}\left\|F(\cdot,t') \right\|_{L_{x}^{2}(\mathbb{R})}dt'\right]
			\nonumber\\
			&\lesssim (1+T^{\frac 12}) \int_{0}^{T}\left\|F(\cdot,t') \right\|_{H_{x}^{2}(\mathbb{R})}dt'.
			\end{align}
Together, estimates \eqref{zt-1} and \eqref{zt-2} imply the bound
\begin{equation}
\left\|\partial_{t}z(x,\cdot)\right\|_{L_t^2(0,T)}
\lesssim
T^{\frac 12} (1+T^{\frac 12}) \left\|F\right\|_{L_t^2((0, T); H_{x}^{2}(\mathbb{R}))}
\end{equation}  
which corresponds to the desired estimate \eqref{nonhthm2} in the case $s=2$. 
\end{proof}

Regarding $L^\mu_tL^r_x$ Strichartz-type estimates for the nonhomogeneous linear Cauchy problem \eqref{nonhomogeneous1}, we have the following result which is a consequence of the homogeneous Strichartz estimates given in Theorem \ref{homStr}.
	\begin{thm}\label{NonHomStrThm}
		Let $s\in\mathbb{R}$ and suppose $(\mu,r)$ is higher-order Schrödinger admissible in the sense of \eqref{admissiblepair}. Then, the solution of the nonhomogeneous linear Cauchy problem \eqref{nonhomogeneous1} satisfies the Strichartz estimate
		\begin{equation}
			\left\|z\right\|_{L^\mu_t((0, T); H^{s,r}_x(\mathbb{R}))}\lesssim \|F\|_{L_t^1((0,T); H^s_x(\mathbb{R}))}.
		\end{equation}
	\end{thm}
	
	\begin{proof}Letting $H(x,t,t'):=\chi_{\{t'\le t\}}(t')S[F(\cdot,t');0](x,t-t')$, we  rewrite  $\left\|z\right\|_{L^\mu_t((0, T); H^{s,r}_x(\mathbb{R}))}$ as
		\begin{equation}
			\left\|z\right\|_{L^\mu_t((0, T); H^{s,r}_x(\mathbb{R}))} = \Big\|\int_0^TH(\cdot,\cdot,t')dt'\Big\|_{L^\mu_t((0, T); H^{s,r}_x(\mathbb{R}))}.
		\end{equation} 
		Therefore, in view of the homogeneous Strichartz estimate \eqref{stry}, we readily infer
		\begin{align*}
			\left\|z\right\|_{L^\mu_t((0, T); H^{s,r}_x(\mathbb{R}))} &\leq \int_0^T\left\|H(\cdot,\cdot,t')\right\|_{L^\mu_t((0, T); H^{s,r}_x(\mathbb{R}))}dt'\\
			&\leq \int_0^T\left\|S[F(\cdot,t');0](\cdot,\cdot-t')\right\|_{L^\mu_t((0,T); H^{s,r}_x(\mathbb{R}))}dt'
			\lesssim   \int_0^T\left\|F(\cdot,t')\right\|_{H^{s}_x(\mathbb{R})}dt'.
		\end{align*}
	\end{proof}

	\subsection{\ttfamily\bfseries Reduced initial-boundary value problem}\label{SecRedibvp} We consider the reduced initial-boundary value problem
	\begin{equation}\label{linear4}
\begin{aligned}
&iq_t+i\beta q_{xxx}+\alpha q_{xx}+i\delta q_x = 0,  \quad (x,t)\in\mathbb{R}_{+} \times (0,T'),\\ 
&q(x,0) = 0, \quad x\in \mathbb{R}_{+}, \\
&q(0,t) = g_0(t) := E_b[g-y(0,\cdot)-z(0,\cdot)](t), \quad t\in (0,T'),
\end{aligned}
	\end{equation}
	where $T'>T$, $y(0, t)$ and $z(0, t)$ are the solutions to the homogeneous and nonhomogeneous Cauchy problems \eqref{cauchy} and \eqref{nonhomogeneous1} evaluated at $x=0$, and $E_b:H_t^{(s+1)/3}(0,T)\rightarrow H_t^{(s+1)/3}(\mathbb{R})$ is a fixed bounded extension operator satisfying the additional property that $\supp g_0 \subset [0,T')$. The construction of such an extension is analogous to the one provided in detail in Section 3 of \cite{hm2021} in the context of the linear Schr\"odinger equation. In particular, we note that, for continuous Sobolev data, a compactly supported extension can be constructed thanks to the compatibility between the initial and boundary data at $(x,t)=(0,0)$ (see also discussion above Theorem \ref{ibvpthem01}). 
In this connection, observe that the traces $y(0,t)$ and $z(0,t)$ are well-defined and belong to $H_t^{(s+1)/3}(0,T)$ in view of Theorems \ref{cauchylemma} and \ref{Nonhomthm}.

	\subsubsection*{Solution formula} We obtain a formula to represent weak solutions of the reduced initial-boundary value problem \eqref{linear4} via Fokas's unified transform method. To this end, we first assume that $q$ is sufficiently smooth up to the boundary of $\mathbb{R}_+\times(0,T')$ and decays sufficiently fast as $x\rightarrow\infty$, uniformly in $[0,T']$.

The definition of the standard Fourier transform on $\mathbb R$ applied on the piecewise-defined function 
$$
F(x) = \left\{\begin{array}{ll} f(x), &x>0, \\ 0, &x<0,\end{array}\right.
\quad
f \in L^2(0, \infty),
$$
gives rise to the half-line Fourier transform pair
\begin{equation}\label{hl-ft-def}
\begin{aligned}
&\widehat f(k) = \int_0^\infty e^{-ikx} f(x) dx, \quad \text{Im}(k) \leq 0,
\\
&f(x) = \frac{1}{2\pi} \int_{\mathbb R} e^{ikx} \widehat f(k) dk, \quad x>0.
\end{aligned}
\end{equation}
Note that the above half-line Fourier transform makes sense for all $ \text{Im}(k) \leq 0$ and not just for $k\in\mathbb R$ as its whole-line counterpart.
Taking the half-line Fourier transform \eqref{hl-ft-def} of \eqref{linear4} and integrating over $(0,t)$, we obtain the following spectral identity known as the global relation:
	\begin{equation}\label{globalrelation}
		e^{\omega(k)t} \, \widehat{q}(k,t)=(-\beta k^2+\alpha k+\delta) \, \widetilde{g}_0(\omega(k),t)+(i\beta k-i\alpha) \, \widetilde{g}_1(\omega(k),t)+\beta \, \widetilde{g}_2(\omega(k),t),\quad \text{Im}k\leq 0,
	\end{equation}
	where $\omega$ is given by \eqref{wkdef} and the temporal transforms $\widetilde{g}_j(\omega(k),t)$ are defined by
		\begin{equation}\label{gj-def}
		\widetilde{g}_j(k,t)=\int_0^t e^{kt'}\partial_x^j q(0,t') dt', \quad k\in\mathbb{C},\quad j=0,1,2,
	\end{equation} 
Then, by the inversion formula in \eqref{hl-ft-def}, 
\begin{equation}\label{extra1}
		q(x,t)=\frac{1}{2\pi}\int_{-\infty}^{\infty} e^{ikx-\omega(k)t} \left[ (-\beta k^2+\alpha k+\delta) \, \widetilde{g}_0(\omega(k),t)+(i\beta k-i\alpha) \, \widetilde{g}_1(\omega(k),t)+\beta \, \widetilde{g}_2(\omega(k),t)\right] dk.
	\end{equation}

The transforms $\widetilde{g}_1$ and $\widetilde{g}_2$  involve the unknown boundary values $q_x(0, t)$ and $q_{xx}(0, t)$. In order to eliminate them from \eqref{extra1}, we proceed as follows. For $D:= \{ k\in\mathbb{C}: \text{Re}(\omega(k))<0 \}$, consider the region 
\begin{equation}\label{dpm-def}
D^+:=D\cap\left\{\text{Im}(k) > 0\right\}
= \Big\{\text{Im}(k) > 0: 3\Big(\text{Re}(k) -\frac{\alpha}{3\beta}\Big)^2-\text{Im}(k)^2-\frac{\alpha^2+3\beta\delta}{3\beta^2}<0 \Big\},
\end{equation}
which is depicted in Figure \ref{dplus123-f} for the various signs of the quantity $\alpha^2 + 3\beta\delta$. 
Then, thanks to analyticity (Cauchy's theorem) and exponential decay, it follows that (e.g. see Appendix A in \cite{hm2020} for a detailed explanation in the context of the linear Schr\"odinger equation) 
	\begin{equation}\label{deformationlemma}
		q(x,t)=\frac{1}{2\pi}\int_{\partial D^+} e^{ikx-\omega(k)t} \left[ (-\beta k^2+\alpha k+\delta) \, \widetilde{g}_0(\omega(k),t)+(i\beta k-i\alpha) \, \widetilde{g}_1(\omega(k),t)+\beta \, \widetilde{g}_2(\omega(k),t)\right] dk,
	\end{equation}
	where the contour $\partial D^+$ is positively oriented, i.e. it is traversed in the direction such that $D^+$ stays to the left of the contour, as shown in Figure \ref{dplus123-f}.  
%	%
%	\begin{figure}[h]
%		\centering
%		\subfloat[$\alpha^2+3\beta\delta>0$]{
%			\includegraphics[width=0.30\textwidth]{D-positive.png}
%			\label{fig:subfig1}}
%		\quad
%		\subfloat[$\alpha^2+3\beta\delta=0$]{
%			\includegraphics[width=0.30\textwidth]{D-zero.png}
%			\label{fig:subfig2}}
%		\quad
%		\subfloat[$\alpha^2+3\beta\delta<0$]{
%			\includegraphics[width=0.30\textwidth]{D-negative.png}
%			\label{fig:subfig3}}
%		\caption{The region $D^+$ with respect to sign of $\alpha^2+3\beta\delta$}
%		\label{fig:globfig}
%	\end{figure}
%	%

The fact that the integral \eqref{deformationlemma} is taken along the deformed contour $\partial D^+$ will allow us to eliminate the unknown transforms $\widetilde{g}_1$ and $\widetilde{g}_2$ from \eqref{deformationlemma} by employing two additional spectral identities emanating from the global relation \eqref{globalrelation} through suitable transformations that keep the spectral function $\omega(k)$ invariant. In particular, both of these identities are valid along $\partial D^+$ and so we will be able to use them simultaneously. It is important to emphasize that the two additional  identities are not valid along $\mathbb R$, which is the reason why the deformation from $\mathbb R$ to $\partial D^+$ that leads to \eqref{deformationlemma} is necessary.

In order to determine the symmetry transformations, we solve  the equation $\omega(\nu)=\omega(k)$ for $\nu = \nu(k)$. 
\\[2mm]
(i) If $\alpha^2+3\beta\delta>0$, then the two nontrivial symmetries are
\begin{equation}\label{symm1}
\nu_{\pm}(k) 
= 
-\frac{1}{2} \left(k-\frac{\alpha}{\beta}\right)\pm \frac{\sqrt 3 \, i}{2} \left[\left(k-\frac{\alpha}{3\beta}\right)^2 -\frac{4\left(\alpha^2+3\beta\delta\right)}{9\beta^2}\right]^{\frac 12}.
\end{equation}
The square root term in \eqref{symm1} is defined as follows. Denoting the two branch points by $b_\pm := \frac{1}{3\beta} \big(\alpha \pm 2 \sqrt{\alpha^2+3\beta\delta}\big)$, we write  $k-b_\pm = \left|k-b_\pm\right| e^{i\theta_\pm}$ with $-\pi < \theta_- \leq \pi$ and $0 \leq \theta_+ < 2\pi$, which correspond to branch cuts along $[b_+, \infty)$ for $(k-b_+)^{\frac 12}$ and along $(-\infty, b_-]$ for $(k-b_-)^{\frac 12}$. Then, we associate the square root in \eqref{symm1} with the singlevalued function
\begin{equation}\label{sqrt-term}
\left[\left(k-\frac{\alpha}{3\beta}\right)^2 - \frac{4\left(\alpha^2+3\beta\delta\right)}{9\beta^2}\right]^{\frac 12}
=
\sqrt{\left|k-b_+\right| \left|k-b_-\right|} \, e^{i(\theta_+ + \theta_-)/2},
\end{equation}
which is analytic for all $k \notin \mathcal B := (-\infty, b_-] \cup [b_+, \infty)$. In turn, this definition ensures that $\nu_{\pm}$ are analytic for all $k\in\mathbb C \setminus \mathcal B$. Importantly, as shown in Figure \ref{dplus123-f}, $\mathcal B \cap \overline{D^+} = \O$. 
\\[2mm]
(ii) $\alpha^2+3\beta\delta=0$. In this case, the symmetries are the two entire functions
\begin{equation}\label{symm2}
\nu_{\pm}(k) 
= 
-\frac{1}{2} \left(k-\frac{\alpha}{\beta}\right)\pm \frac{\sqrt 3 \, i}{2} \left(k-\frac{\alpha}{3\beta}\right),
\end{equation}
as shown in Figure \ref{dplus123-f}.
\\[2mm]
(iii) $\alpha^2+3\beta\delta<0$. In that case, the symmetries are again given by \eqref{symm1}; however, as the branch points $b_\pm$ are now complex conjugates along the line $\text{Re}(k) = \frac{\alpha}{3\beta}$, we write $k-b_\pm = \left|k-b_\pm\right| e^{i(\theta_\pm-\pi/2)}$ with $0 \leq \theta_\pm < 2\pi$ and corresponding branch cuts along the vertical half-lines from $b_\pm$ to $\frac{\alpha}{3\beta}-i\infty$, so that
\begin{equation}\label{sqrt-term2}
\left[\left(k-\frac{\alpha}{3\beta}\right)^2 - \frac{4\left(\alpha^2+3\beta\delta\right)}{9\beta^2}\right]^{\frac 12}
=
\sqrt{\left|k-b_+\right| \left|k-b_-\right|} \, e^{i(\theta_+ + \theta_--\pi)/2}
\end{equation}
is singlevalued and analytic for all $k\in\mathbb C \setminus \widetilde{\mathcal B}$, where $\widetilde{\mathcal B}$ is the finite vertical segment connecting $b_+$ and $b_-$, as shown in Figure \ref{dplus123-f}. Note that $\widetilde B \cap \overline{D^+} \neq \O$ as part of the branch cut $\widetilde B$ lies inside the region $D^+$. For this reason, \textit{before} employing the symmetries $\nu_\pm$ for the elimination of the unknown transforms $\widetilde g_2$ and $\widetilde g_1$ from \eqref{deformationlemma}, we use Cauchy's theorem to deform the contour $\partial D^+$ in \eqref{deformationlemma} to the modified contour $\partial \widetilde D^+$, which corresponds to the positively oriented boundary of the region $\widetilde D^+$ shown in Figure \ref{dplus23mod-f}. This way, the branch cut $\widetilde {\mathcal B}$ is avoided \textit{prior to} the use of the symmetries $\nu_\pm$, allowing us to take advantage of analyticity inside the region $\widetilde D^+$ later.

In view of the above discussion, we rewrite \eqref{deformationlemma} as 
	\begin{equation}\label{deformationGamma}
		q(x,t)=\frac{1}{2\pi}\int_{\Gamma} e^{ikx-\omega(k)t}\left[ \left(-\beta k^2+\alpha k+\delta\right) \widetilde{g}_0(\omega(k),t)+i\left(\beta k-\alpha\right) \widetilde{g}_1(\omega(k),t)+\beta \, \widetilde{g}_2(\omega(k),t)\right] dk,
	\end{equation}
	where the integration contour $\Gamma$ is given by
	\begin{equation}\label{Gamma-def}
		\Gamma =\left\{
		\begin{array}{ll}
			\partial D^+, &\alpha^2+3\beta\delta > 0, \\
			\partial \widetilde{D}^+, &\alpha^2+3\beta\delta \leq 0.
		\end{array}
		\right.
	\end{equation}

\begin{figure}[ht!]
		\centering
\includegraphics[width=5.9cm, height=2.75cm]{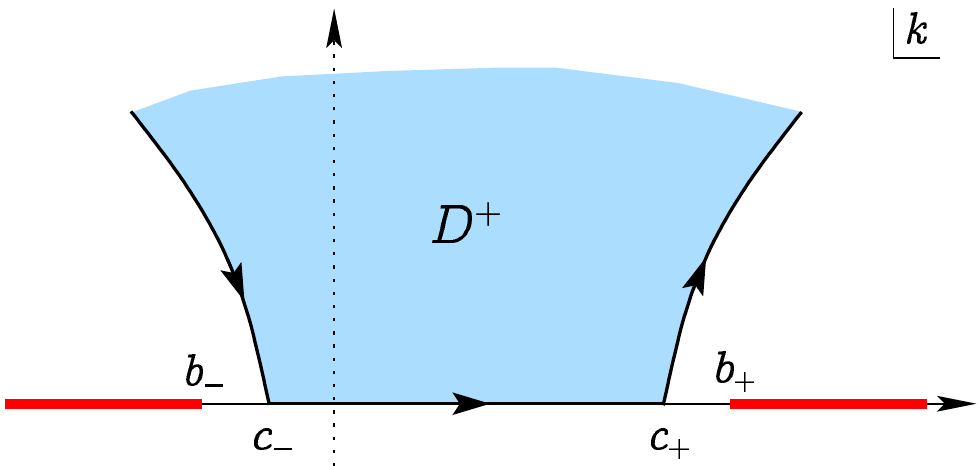}
\hspace*{3mm}
\includegraphics[width=4cm, height=2.75cm]{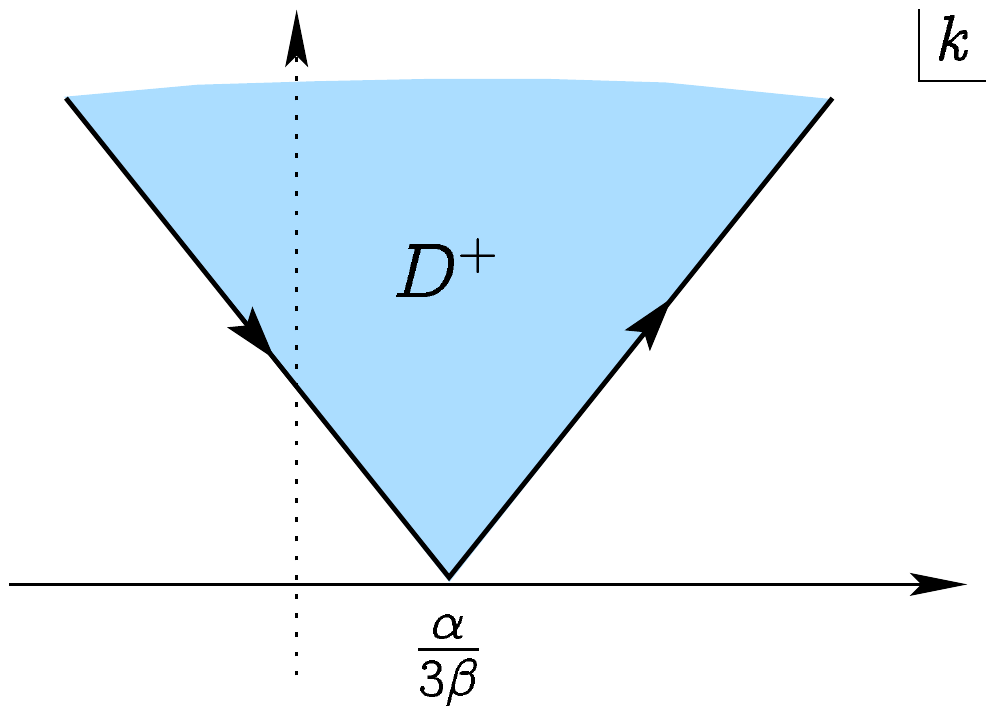}
\hspace*{3mm}
\includegraphics[width=4.2cm, height=2.75cm]{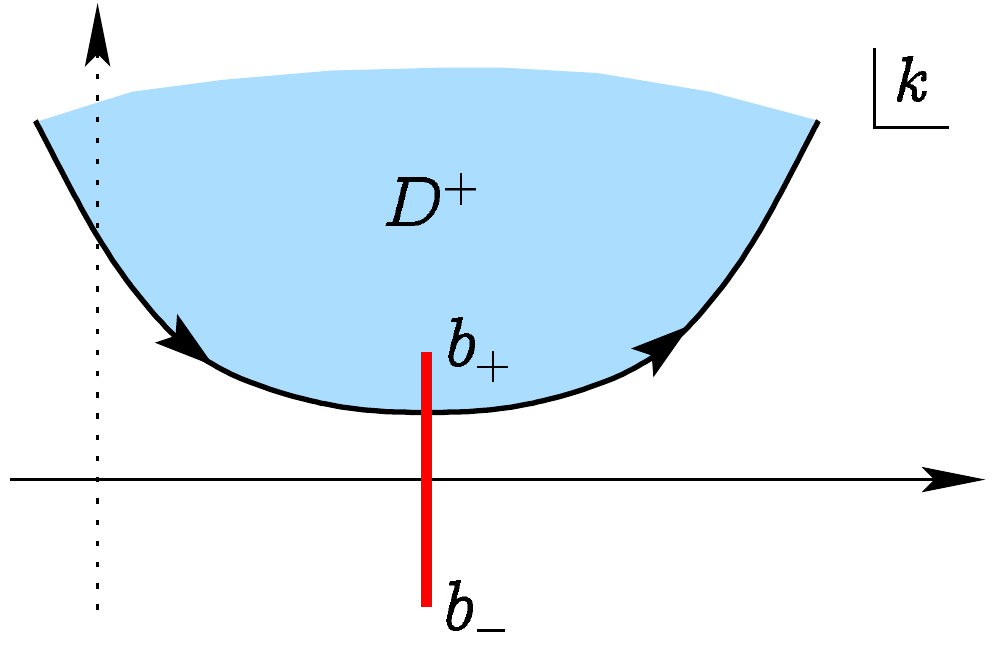}
		\caption{The region $D^+$ defined by \eqref{dpm-def} for $\alpha^2+3\beta\delta>0$ (left), $\alpha^2+3\beta\delta=0$ (center), and $\alpha^2+3\beta\delta<0$ (right). In the first case, the square root branch cut $\mathcal B = (-\infty, b_-] \cup [b_+, \infty)$ with branch points $b_\pm = \frac{1}{3\beta} \left(\alpha \pm 2 \sqrt{\alpha^2 + 3\beta\delta}\right)$ (shown in red) stays outside the region $D^+$, while in the second case there is no branching. On the other hand, in the third case the branch cut $\widetilde B$ (shown in red) is taken along the vertical line segment connecting $b_+$ to $b_-$ and so part of it lies in $D^+$, thus a local deformation around $\widetilde B$ is performed as shown in Figure \ref{dplus23mod-f} below.
}
\label{dplus123-f}
\end{figure}
	
\begin{figure}[ht!]
\centering
\includegraphics[width=4.1cm, height=3.1cm]{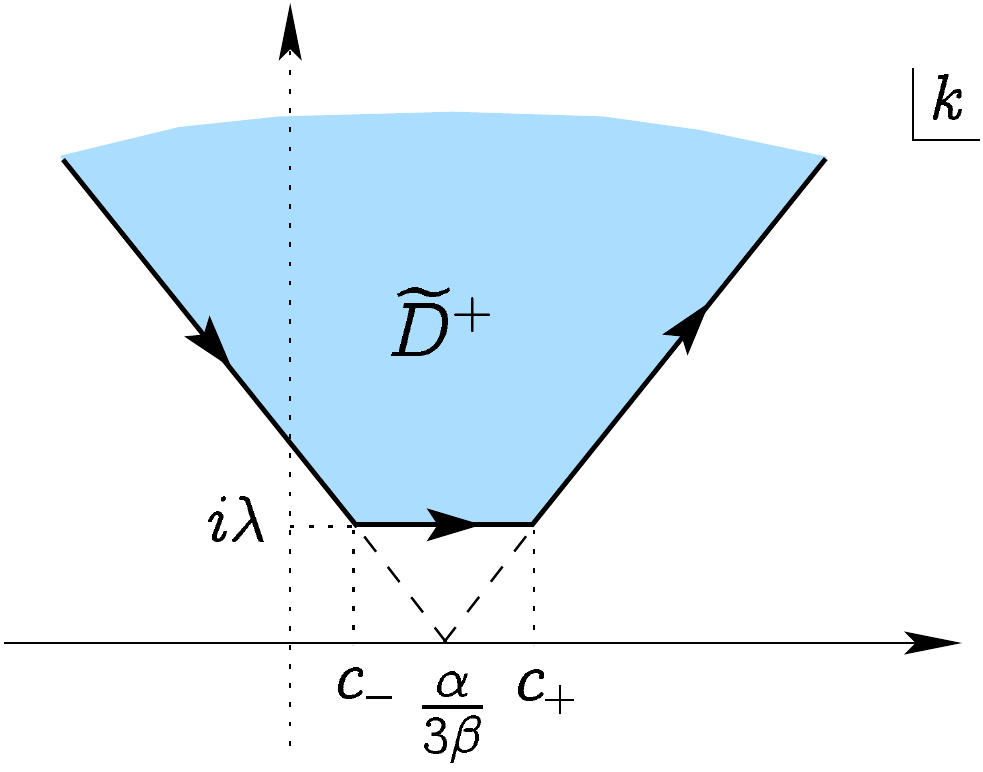}
\hspace*{6mm}
\includegraphics[width=4.3cm, height=3cm]{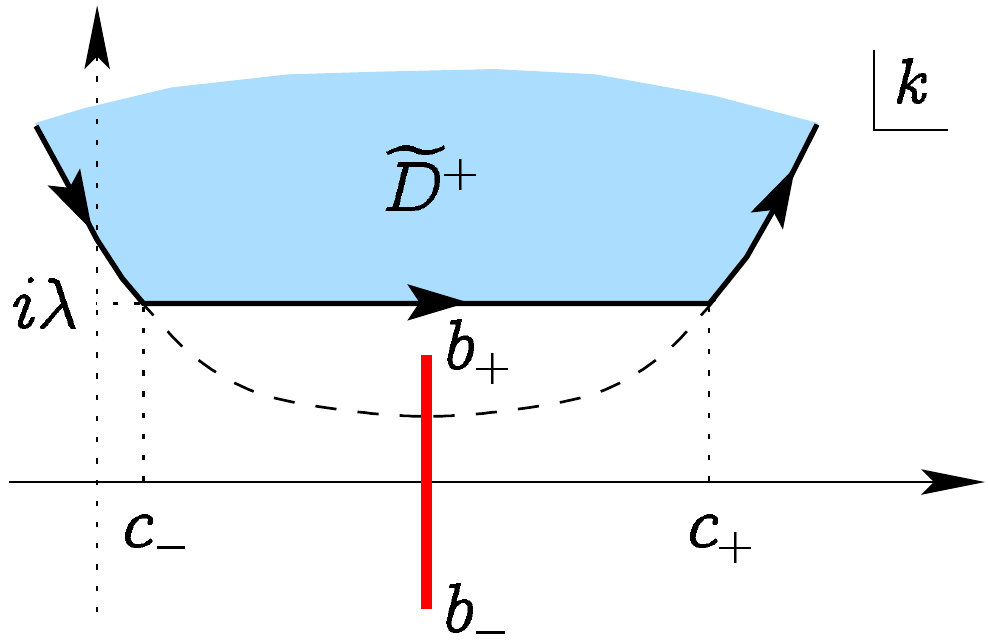}
\caption{Deformation of $\partial D^+$ to $\partial \widetilde D^+$ for $\alpha^2+3\beta\delta = 0$ (left) and $\alpha^2+3\beta\delta < 0$ (right). The left deformation is carried out in order to stay away from the point $\frac{\alpha}{3\beta}$, which is a zero of the quantity $\nu_-(k) - \nu_+(k)$. The right deformation is done in order to avoid crossing the branch cut $\widetilde B$ (shown in red). 
}
\label{dplus23mod-f}
\end{figure}

Replacing $k$ by $\nu_\pm(k)$ in the global relation \eqref{globalrelation} and using the fact that $\omega(\nu_\pm(k)) = \omega(k)$, we get the spectral identities
	\begin{equation}\label{global2}
	\begin{aligned}
		e^{\omega(k)t} \, \widehat{q}(\nu_\pm(k),t)&=\left(-\beta \nu_\pm^2(k)+\alpha \nu_\pm(k)+\delta\right) \widetilde{g}_0(\omega(k),t)+i\left(\beta \nu_\pm(k)-\alpha\right) \widetilde{g}_1(\omega(k),t)\\
		&\quad +\beta \, \widetilde{g}_2(\omega(k),t),\quad \text{Im}(\nu_\pm(k))\leq 0.
	\end{aligned}
	\end{equation}
We emphasize that the above identities are valid only for $k$ such that $\text{Im}(\nu_\pm(k))\leq 0$. Thus, in order to employ them for 
the elimination of the unknown boundary values from \eqref{deformationGamma}, we need to ensure that $\Gamma \subseteq \left\{\text{Im}(\nu_\pm(k)) \leq 0\right\}$. This is proved in the following lemma.
	\begin{lem}\label{nupm-l}
Let $\nu_\pm=\nu_\pm(k)$ be the nontrivial (i.e. $\nu_\pm \not\equiv k$) solutions of the equation $\omega(\nu)=\omega(k)$ as given by \eqref{symm1} or \eqref{symm2}, depending on the value of $\alpha^2+3\beta\delta$.	If $k\in \overline{D^+}$, then $\textnormal{Im} (\nu_\pm)\leq 0$.
	\end{lem}
	\begin{proof} For all $k=k_R+ik_I\in \overline{D^+}$ such that $\nu(k)\neq k$ satisfies $\omega(\nu)=\omega(k)$, we must have $\beta(\nu^2+k\nu+k^2)-\alpha(\nu+k)-\delta=0$. Writing  $\nu=\nu_R+i\nu_I$ and taking real and imaginary parts, this equation is equivalent to the system
%		\begin{align}\label{inv1}
%			\left(\nu_R+\frac{k_R}{2}-\frac{\alpha}{2\beta}\right) \left(\nu_I+\frac{k_I}{2}\right)&=\left(\frac{\alpha}{4\beta}-\frac{3k_R}{4}\right)k_I,\\
%			\label{inv2}
%			\left(\nu_R+\frac{k_R}{2}-\frac{\alpha}{2\beta}\right)^2-\left(\nu_I+\frac{k_I}{2}\right)^2&=-\frac34k_R^2+\frac34k_I^2+ \frac{\alpha}{2\beta}k_R+\frac{\alpha^2}{4\beta^2}+ \frac{\delta}{\beta}.
%		\end{align}
%%
\begin{equation}\label{inv1}
\widetilde \nu_R \widetilde \nu_I = c k_I,
\quad
\widetilde \nu_R^2-\widetilde \nu_I^2=d,	
\end{equation}
where $\widetilde \nu_R=\nu_R+\frac{k_R}{2}-\frac{\alpha}{2\beta}$, $\widetilde \nu_I=\nu_I+\frac{k_I}{2}$, $c=\frac{\alpha}{4\beta}-\frac{3k_R}{4}$ and $d=-\frac34k_R^2+\frac34k_I^2+ \frac{\alpha}{2\beta}k_R+\frac{\alpha^2}{4\beta^2}+ \frac{\delta}{\beta}$. 
If $\widetilde \nu_I=0$, then $\nu_I=-\frac{k_I}{2}\leq 0$ as $k\in\overline{D^+}$ and we are done. So let us assume  $\widetilde \nu_I\neq 0$.  Then, combining the two equations in \eqref{inv1} we obtain $\widetilde \nu^4_I+d\widetilde \nu_I^2-c^2k_I^2=0$, which can be solved for $\widetilde \nu_I^2$ to yield
$
\widetilde \nu_I^2 = \frac{-d \pm \sqrt{d^2+4c^2k_I^2}}{2} = -\frac d2 \pm \sqrt{\frac{d^2}{4} + c^2 k_I^2}.
$
Note that only the positive sign is acceptable since $\widetilde \nu_I \in \mathbb R \Rightarrow \widetilde \nu_I^2\geq 0$. That is, 
$
\widetilde \nu_I^2 = -\frac d2 + \sqrt{\frac{d^2}{4} + c^2 k_I^2}
$
implying
$
\widetilde \nu_I = \pm \sqrt{-\frac d2 + \sqrt{\frac{d^2}{4} + c^2 k_I^2}}.
$
In turn, from the first of equations \eqref{inv1} we get
$
\widetilde \nu_R = \pm \frac{ck_I}{\sqrt{-\frac d2 + \sqrt{\frac{d^2}{4} + c^2 k_I^2}}}
$
and so
\begin{equation}\label{nupm2}
\nu_R = \pm \frac{ck_I}{\sqrt{-\frac d2 + \sqrt{\frac{d^2}{4} + c^2 k_I^2}}} - \frac{k_R}{2} + \frac{\alpha}{2\beta},
\quad
\nu_I = -\frac{k_I}{2} \pm \sqrt{-\frac d2 + \sqrt{\frac{d^2}{4} + c^2 k_I^2}}.
\end{equation}
Observe that the radicand of the outer square root involved in the above  expressions is a non-negative number and hence that square root is a real (non-negative) number.
In addition, note that expressions \eqref{nupm2} are consistent with equations \eqref{symm1} and \eqref{symm2}; however, their dependence on $k_R$ and $k_I$ (as opposed to $k$) is not suitable for discussing the analyticity of the associated expressions for $\nu$, which is why \eqref{symm1} and \eqref{symm2} were used earlier for that purpose. On the other hand, \eqref{nupm2} are the forms convenient for proving Lemma \ref{nupm-l}.
%
%
%\begin{equation}\label{vfirst}
%			\nu_I=-\frac{k_I}{2}-\sqrt{-\frac{d}{2}+\sqrt{\frac{d^2}{4}+c^2k_I^2}}
%		\end{equation}
%or
%		\begin{equation}\label{vsecond}
%			\nu_I=-\frac{k_I}{2}+\sqrt{-\frac{d}{2}+\sqrt{\frac{d^2}{4}+c^2k_I^2}}
%		\end{equation}
%%

The case of the negative square root sign in \eqref{nupm2} is straightforward as then $\nu_I \leq 0$ for all $k_I\ge 0$ and, in particular, for $k\in \overline{D^+}$ as desired. On the other hand, the case of positive square root sign  in\eqref{nupm2} requires more work. More specifically, by definition \eqref{dpm-def}, for $k \in \overline{D^+}$  we have
\begin{equation}\label{dpm-def2}
3\left(k_R -\frac{\alpha}{3\beta}\right)^2-k_I^2-\frac{\alpha^2+3\beta\delta}{3\beta^2}\le 0, 
\end{equation}
which can be rearranged to
$-\frac{3}{4}k_R^2+\frac{1}{4}k_I^2+\frac{\alpha}{2\beta}k_R +\frac{\delta}{4\beta}\ge 0.$ 
For $k_I\neq 0$ (note that $k_I = 0$ implies $\nu_I =0$ and we are done), this is equivalent to
$ 
\frac{k_I^4}{16}+k_I^2\frac{d}{4}\ge c^2k_I^2
$ 
or, after completing the square,
		$\left(\frac{k_I^2}{4}+\frac{d}{2}\right)^2\ge \frac{d^2}{4}+c^2k_I^2.$ 
Hence, $\frac{k_I^2}{4}\ge -\frac{d}{2}+\sqrt{\frac{d^2}{4}+c^2k_I^2}$ or  $\frac{k_I^2}{4}\le -\frac{d}{2}-\sqrt{\frac{d^2}{4}+c^2k_I^2}$ and, as the second inequality is not possible because it would imply that $k_I^2 \leq 0$, taking the square root of the first inequality and using the fact that $k_I\ge 0$ for $k\in\overline{D^+}$, we obtain
$
0\ge -\frac{k_I}{2}+\sqrt{-\frac{d}{2}+\sqrt{\frac{d^2}{4}+c^2k_I^2}}=\nu_I
$
as desired.

The proof so far has been under the assumption that $\nu(k)\neq k$; however, although $\nu \not\equiv k$ by hypothesis,  there could still be points in $\overline{D^+}$ where $\nu(k) = k$ and hence this scenario must also be considered. In that case, recalling that $\nu_\pm$ satisfy 
		$\beta(\nu_\pm^2+k\nu_\pm+k^2)-\alpha(\nu_\pm+k)-\delta=0$, we infer that if $k\in\mathbb{C}$ is such that $\nu_\pm(k)=k$  then $3\beta k^2-2\alpha k - \delta=0$.
%
%$$
%k=\frac{\alpha}{3\beta}\pm\frac{\sqrt{\alpha^2+3\beta\delta}}{3\beta}.
%$$ 
%%
If $\alpha^2+3\beta\delta\ge 0$, then $k=k_\pm=\frac{\alpha}{3\beta}\pm\frac{\sqrt{\alpha^2+3\beta\delta}}{3\beta}\in\mathbb R$ i.e.  $k_I=\text{Im}(\nu_\pm)=0$ and we are done. 
If $\alpha^2+3\beta\delta< 0$, then $k=k_\pm=\frac{\alpha}{3\beta}\pm i\frac{\sqrt{-(\alpha^2+3\beta\delta)}}{3\beta}$. Note that $k_{-}\notin \overline{D^+}$ since $\text{Im}(k_-) < 0$. Also, $k_{+}\notin \overline{D^+}$ because if $\alpha^2+3\beta\delta< 0$ and $k\in \overline{D^+}$ then by \eqref{dpm-def2} we must have $k_I\ge \sqrt{-\frac{\alpha^2+3\beta\delta}{3\beta^2}} = \frac{\sqrt{-(\alpha^2+3\beta\delta)}}{\sqrt 3\beta}>\frac{\sqrt{-(\alpha^2+3\beta\delta)}}{3\beta} = \text{Im}(k_+)$. This completes the proof of Lemma~\ref{nupm-l}.
\end{proof}

	Thanks to Lemma \ref{nupm-l}, both of the identities \eqref{global2} are valid for $k\in \overline{D^+}$ and hence can be solved simultaneously as a system for the unknown transforms $\widetilde{g}_1(\omega(k),t)$ and $\widetilde{g}_2(\omega(k),t)$ to yield
\begin{align}\label{g1}
\widetilde{g}_1(\omega(k),t)
&=
\frac{e^{\omega(k)t}}{i\beta \left[\nu_+(k)-\nu_-(k)\right]}[\widehat{q}(\nu_+(k),t)-\widehat{q}(\nu_-(k),t)] +ik\widetilde{g}_0(\omega(k),t),
\\
\label{g2}
\widetilde{g}_2(\omega(k),t)
&=
\frac{e^{\omega(k)t}}{\beta^2 \left[\nu_-(k)-\nu_+(k)\right]} \left[\left(\beta \nu_-(k)-\alpha\right)\widehat{q}(\nu_+(k),t)-\left(\beta \nu_+(k)-\alpha\right)\widehat{q}(\nu_-(k),t)\right]
\\
&\quad
-k^2\widetilde{g}_0(\omega(k),t).
\nonumber
\end{align}
Substituting these expressions in the integral representation \eqref{deformationGamma}, we obtain
	\begin{equation}\label{vanish}
	\begin{aligned}
		q(x,t)&=\frac{1}{2\pi}\int_{\Gamma}e^{ikx-\omega(k)t} \left(-3\beta k^2+2\alpha k+\delta\right) \widetilde{g}_0(\omega(k),t)dk
		\\
		&\quad
		+\frac{1}{2\pi}\int_{\Gamma}e^{ikx}\left[\frac{\nu_-(k)-k}{\nu_-(k)-\nu_+(k)} \, \widehat{q}(\nu_+(k),t)-\frac{\nu_+(k)-k}{\nu_-(k)-\nu_+(k)} \, \widehat{q}(\nu_-(k),t)\right]dk. 
		\end{aligned}
	\end{equation}
	
Note that the definition \eqref{Gamma-def} of $\Gamma$ in conjunction with the choices of the contour $\partial \widetilde D^+$ shown in Figure~\ref{dplus23mod-f} ensure that $\nu_-(k)-\nu_+(k)$ stays away from zero. Indeed, for $\alpha^2 +3\beta\delta>0$ the solutions of $\nu_-(k)-\nu_+(k)=0$ occur at the branch points $b_\pm$, which lie on the real axis and outside segment $\big[\frac{1}{3\beta}(\alpha - \sqrt{\alpha^2+3\beta\delta}), \frac{1}{3\beta}(\alpha + \sqrt{\alpha^2+3\beta\delta})\big]$  forming the base of $\Gamma=\partial D^+$ (see left panel of Figure \ref{dplus123-f}). 
Moreover, for $\alpha^2 +3\beta\delta=0$ the quantity $\nu_-(k)-\nu_+(k)$ vanishes at $\frac{\alpha}{3\beta}$, which is bypassed by $\Gamma = \partial D^+$ as shown on the left panel of Figure \ref{dplus23mod-f}. Finally, for $\alpha^2 +3\beta\delta<0$ the roots of $\nu_-(k)-\nu_+(k)=0$ are again at the branch points $b_\pm$ and so they stay below the contour $\Gamma = \partial D^+$ depicted on the right panel of Figure \ref{dplus23mod-f}. 

Therefore, using analyticity (Cauchy's theorem) along with exponential decay as $|k| \to \infty$ inside $D^+$ or $\widetilde D^+$, as appropriate, we conclude that the second $k$-integral on the right-hand side of \eqref{vanish} is equal to zero. (To see the decay, note that $|e^{ikx-i\nu_\pm y}| = e^{-\text{Im}(k) x + \text{Im}(\nu_\pm) y}$ and use Lemma \ref{nupm-l} together with the fact that $x, y > 0$.)
Consequently, we deduce the solution formula
	\begin{equation}\label{representation0}
		q(x,t)=-\frac{i}{2\pi}\int_{\Gamma}e^{ikx-\omega(k)t} \omega'(k) \, \widetilde{g}_0(\omega(k),t)dk.
	\end{equation} 

In fact, noting that 
$
\big| e^{-\omega(k) (t-t')} \big|
=
e^{\text{Re}(\omega(k)) (t'-t)}
$
and recalling that, by definition \eqref{dpm-def}, $\text{Re}(\omega(k))<0$ inside $D^+$, we see that the exponential $e^{ikx-\omega(k)(t-t')}$ decays as $|k|\to\infty$ inside $D^+$ for all $x>0$,  $t'>t$. Thus, combining this decay with analyticity, in the second argument of the time transform $\widetilde g_0$ we can replace $t$ by any fixed $T'>t$ and thereby obtain the following equivalent version of the solution formula \eqref{representation0}, which is more convenient for the purpose of linear estimates as we will see below:
\begin{equation}\label{representation}
q(x,t) = -\frac{i}{2\pi}\int_{\Gamma}e^{ikx-\omega(k)t} \omega'(k) \, \widetilde{g}_0(\omega(k),T')dk.
\end{equation} 

	\subsubsection*{Compatibility between the data}  Recall that the initial and boundary data of the initial-boundary value problem \eqref{linear1} belong in the $L^2$-based Sobolev spaces $H_x^s(\mathbb{R}_+)$ and $H_t^{(s+1)/3}(0,T)$, respectively. Moreover, in view of the range of validity of Theorem \ref{Nonhomthm} for the nonhomogeneous Cauchy problem established earlier, as well as of Theorem \ref{ibvpthem01} for the reduced initial-boundary value problem proved below, we will restrict our attention to the range $0\leq s \leq 2$ with $s\neq \frac 12$.  
	
	For $\frac 12<s\leq 2$,  continuity becomes relevant to our analysis and it turns out that we need to impose a compatibility condition between the initial and the boundary data. More specifically, note that if $\frac 12 < s \leq 2$ then $\frac12<\frac{s+1}{3}\leq1$. Therefore, both of the traces $u_0(0)$ and $g(0)$ are well-defined. 
Furthermore, since $y(0,\cdot)$ and $z(0,\cdot)$ belong to $H_t^{(s+1)/3}(0,T)$ by Theorems \ref{cauchylemma} and  \ref{Nonhomthm}, the traces $y(0,0)$ and  $z(0,0)$ are well-defined and equal to $u_0(0)$ and $0$, respectively, due to continuity and the initial conditions in problems \eqref{cauchy} and \eqref{nonhomogeneous1}. Thus, using continuity at zero for the function $g_0 \in H_t^{(s+1)/3}(\mathbb R)$ defined in \eqref{linear4}, we have

$$
g_0(0) = \lim_{t \to 0^+} g_0(t) = \lim_{t \to 0^+} \left[g(t)-y(0, t)-z(0, t)\right] = g(0) - y(0, 0) - z(0, 0) = g(0) - u_0(0)
$$
which, upon imposing the (natural) compatibility condition
\begin{equation}\label{comp-cond}
u_0(0) = g(0), \quad \frac 12 < s \leq 2,
\end{equation}
implies that the boundary datum of the reduced problem \eqref{linear4} vanishes at $t=0$, i.e.
\begin{equation}\label{comp-cond-0}
g_0(0) = 0, \quad \frac 12 < s \leq 2.
\end{equation}
This feature will turn out to be convenient in the proof of Theorem \ref{ibvpthem01} which follows next.
%Using the initial conditions, we get $y(0,0)=u_0(0)$ and $z(0,0)=0$.  Therefore, under the above compatibility assumption, the boundary datum of reduced initial-boundary value problem satisfies
%	$$g_0(0)=g(0)-y(0,0)-z(0,0)=u_0(0)-u_0(0)=0, \quad \text{if } s\in\left(\frac12,\frac72\right).$$
	\subsubsection*{Sobolev-type estimates} 
	We  now establish the basic space estimate in the initial-boundary value problem setting. More precisely,  we  prove
\begin{thm}\label{ibvpthem01} 
Let $s\ge 0$. Then, the unique solution of the reduced initial-boundary value problem~\eqref{linear4} satisfies
\begin{equation}\label{ribvp-se}
\|q(\cdot,t)\|_{H_x^s(\mathbb{R}_+)}\le c \, \big(1+\sqrt{T'}e^{cT'}\big) \|g_0\|_{H_t^{\frac{s+1}{3}}(0,T')}
\end{equation}
uniformly for $t \in [0,T']$, where $c>0$ is a constant that only depends on $\alpha,\beta,\delta$ and $s$. 
\end{thm}

\begin{proof}
We employ the Fokas method solution formula \eqref{representation}. First, recalling the definition \eqref{dpm-def} of $D^+$ and the various scenarios depending on the sign of $\alpha^2+3\beta\delta$ that are shown in Figures \ref{dplus123-f} and \ref{dplus23mod-f}, we parametrize the integration contour in \eqref{representation} as $\Gamma=(-\gamma_1)\cup\gamma_2\cup\gamma_3$ with
		\begin{alignat}{2}
			\gamma_1(m)&=\frac{\alpha-\sqrt{3\beta^2m^2+\alpha^2+3\beta\delta}}{3\beta}+im,\quad &&\lambda\leq m<\infty,
			\nonumber\\
			\gamma_2(m)&=m+i\lambda, &&c_-< m< c_+,
			\label{gamma-par}\\
			\gamma_3(m)&=\frac{\alpha+\sqrt{3\beta^2m^2+\alpha^2+3\beta\delta}}{3\beta}+im, &&\lambda\leq m<\infty,
			\nonumber
		\end{alignat}
		where, as depicted in Figures \ref{dplus123-f} and \ref{dplus23mod-f},  $c_{\pm}=\frac{\alpha\pm\sqrt{3\beta^2\lambda^2+\alpha^2+3\beta\delta}}{3\beta}$ and $\lambda>0$ is a fixed non-negative real number such that
\begin{equation}\label{cursornumber}
		\left\{
		\begin{array}{ll}
			\lambda=0, &\alpha^2+3\beta\delta>0, \quad \text{(first panel in Figure \ref{dplus123-f})}
			 \\
			\lambda>\frac{2\sqrt{-(\alpha^2+3\beta\delta)}}{3\beta}, &\alpha^2+3\beta\delta\le 0. \quad \text{(Figure \ref{dplus23mod-f})}
		\end{array}
		\right.
\end{equation}
In view of the above parametrization, for any $j\in\mathbb N_0$ we have
\begin{align}
\partial_x^jq(x,t)
&=-\frac{1}{2\pi}\int_{\infty}^{\lambda}(i\gamma_1(m))^je^{i\gamma_1(m)x-(\omega(\gamma_1(m))t} \, \widetilde{g}_0(\omega(\gamma_1(m)),T')\frac{d[i\omega(\gamma_1(m))]}{dm}\; dm
\label{q1-def}
\\
&\quad
-\frac{1}{2\pi}\int_{c_-}^{c_+}(i\gamma_2(m))^je^{i\gamma_2(m)x-(\omega(\gamma_2(m))t} \, \widetilde{g}_0(\omega(\gamma_2(m)),T')\frac{d[i\omega(\gamma_2(m))]}{dm}\; dm
\label{q2-def}
\\
&\quad
-\frac{1}{2\pi}\int_{\lambda}^{\infty}(i\gamma_3(m))^je^{i\gamma_3(m)x-(\omega(\gamma_3(m))t} \, \widetilde{g}_0(\omega(\gamma_3(m)),T')\frac{d[i\omega(\gamma_3(m))]}{dm}\; dm
\label{q3-def}
\\
&=: q_1(x,t) + q_2(x,t) + q_3(x,t).
\nonumber
\end{align}

As the terms $q_1$ and $q_3$ are analogous, they can be handled  in a similar fashion and hence we only provide the details for the estimation of $q_1$ given by \eqref{q1-def}. Since
\begin{align}\nonumber
			\|q_1(\cdot,t)\|_{L_x^2(\mathbb{R}_+)}^{2}&=\frac{1}{4\pi^2}\int_{0}^{\infty}\bigg| \int_{\lambda}^{\infty}(i\gamma_1(m))^je^{i\gamma_1(m)x-(\omega(\gamma_1(m))t} \, \widetilde{g}_0(\omega(\gamma_1(m)),T')\frac{d[i\omega(\gamma_1(m))]}{dm} \, dm \bigg|^2 \,dx\\\nonumber
			&\lesssim\int_{0}^{\infty}\left( \int_{0}^{\infty}e^{-mx}|\gamma_1(m)|^j\left|\widetilde{g}_0(\omega(\gamma_1(m)),T')\right|\left|\frac{d[i\omega(\gamma_1(m))]}{dm} \right|\, \chi_{(\lambda,\infty)}(m)dm \right)^2 \,dx,
\end{align}
by the boundedness of the Laplace transform in $L^2(\mathbb R_+)$ (e.g. see Lemma 3.2 in \cite{fhm2017}) we have
\begin{equation}\label{lowneed}
\|q_1(\cdot,t)\|_{L_x^2(\mathbb{R}_+)}^{2}
\lesssim
\int_{\lambda}^{\infty}|\gamma_1(m)|^{2j}|\widetilde{g}_0(\omega(\gamma_1(m)),T')|^2\left|\frac{d[i\omega(\gamma_1(m))]}{dm} \right|^2dm.
\end{equation}

Let $\tau(m)=i\omega(\gamma_1(m)),\, m\in[\lambda,\infty)$. Note that $\tau(m) \in \mathbb R$ since $\gamma_1(m)\in \partial D^+$ and $\text{Re}(\omega(k))=0$ for $k\in \partial D^+$ and, more precisely, $\text{Range}(\tau)= [i\omega(c_-),\infty)$. Furthermore, since $\tau'(m)\neq 0$ on $(\lambda,\infty)$ and $\tau\rightarrow \infty$ as $m\rightarrow \infty$, it follows that $\tau:[\lambda,\infty)\rightarrow [i\omega(c_-),\infty)$ is monotone increasing and so $\tau'(m) > 0$.  Then, \eqref{lowneed} becomes
\begin{equation}\label{lowneed2}
\begin{aligned}
\|q_1(\cdot,t)\|_{L_x^2(\mathbb{R}_+)}^{2}
&\lesssim
\int_{\lambda}^{\infty}|\gamma_1(m)|^{2j}|\widetilde{g}_0(-i\tau(m),T')|^2 \left[\tau'(m)\right]^2dm
\\
&=
\int_{\lambda}^{\infty}|\gamma_1(m)|^{2j}|\widehat{g}_0(\tau(m))|^2 \left[\tau'(m)\right]^2dm
\end{aligned}
\end{equation}
after observing that the time transform \eqref{gj-def} of $g_0$ at $T'$ is in fact the Fourier transform of $g_0$ thanks to the fact that $g_0$ has compact support inside $(0,T')$, namely
\begin{equation}\label{ft-comp}
\widetilde{g}_0(-i\tau(m),T')=\widehat{g}_0(\tau(m)).
\end{equation} 
Next, we have the following auxiliary result.
\begin{lem}\label{suplem}
There is a constant $c>0$ depending only on $\alpha,\beta,\delta$ such that
$$
\sup_{m\in [\lambda,\infty)}\frac{|\gamma_1(m)|^{2j} \tau'(m)}{[1+\tau^2(m)]^{\frac{j+1}{3}}}\le c<\infty.
$$
\end{lem}

We prove 	Lemma \ref{suplem} after the end of the current proof. Employing it in combination with \eqref{lowneed2}, we obtain
\begin{equation}
\begin{aligned}
\|q_1(\cdot,t)\|_{L_x^2(\mathbb{R}_+)}^{2}
&\lesssim
\int_{\lambda}^{\infty}[1+\tau^2(m)]^{\frac{j+1}{3}}|\widehat{g}_0(\tau(m))|^2 \tau'(m) dm
\\
&=
\int_{i\omega(c_-)}^\infty(1+\tau^2)^{\frac{j+1}{3}}|\widehat{g}_0(\tau)|^2d\tau = \|g_0\|_{H_t^{\frac{j+1}{3}}(\mathbb{R})}^2
\end{aligned}
\end{equation}
uniformly for $t \in [0, T']$, completing the estimation of $q_1$. 

\vskip 3mm

We proceed to the estimation of $q_2$ given by \eqref{q2-def}. 
\\[2mm]
\textit{Case 1: $\alpha^2+3\beta\delta>0$}. Then, $\lambda=0$ and by the definition of $\gamma_2$ we can rewrite $q_2$ as
		\begin{equation}\label{q2rewrite}
			q_2(x,t)=-\frac{i}{2\pi}\int_{c_-}^{c_+} (im)^j e^{imx-\omega(m)t} \, \widetilde{g}_0(\omega(m),T') \, \omega'(m) dm.
		\end{equation}
so that $q_2(\cdot,t)$ can be regarded as the inverse spatial Fourier transform of the function
		\begin{equation}\label{Q2mtdef}
		Q_2(m,t)=\left\{\def\arraystretch{1.5}\begin{array}{ll}
				0, &m\notin (c_-,c_+),\\
				-i (im)^j e^{-\omega(m)t} \, \widetilde{g}_0(\omega(m),T')\, \omega'(m), &m\in (c_-,c_+).
		\end{array}\right.
		\end{equation}
Note that $|e^{\omega(m)\rho}|=1$ for $m\in (c_-,c_+)$, $\rho\in \mathbb{R}$. Hence, using the definition of the $t$-transform \eqref{gj-def} and the Cauchy-Schwarz inequality, we have 
\begin{equation}\label{g0sqrtT}
|\widetilde{g}_0(\omega(m),T')|\le \sqrt{T'} \, \|g_0\|_{L_t^2(0,T')}.
\end{equation}  
implying via Plancherel's theorem that
\begin{equation}
\begin{aligned}
\|q_2(\cdot,t)\|_{L_x^2(\mathbb{R}_+)}^2&=\int_{-\infty}^{\infty}|Q_2(m,t)|^2dm
\leq T'\|g_0\|_{L_t^2(0,T')}^2\int_{c_-}^{c_+}|m|^{2j}|\omega'(m)|^2dm
\\
&=cT'\|g_0\|_{L_t^2(0,T')}^2
\lesssim
T' \|g_0\|_{H_t^{\frac{j+1}{3}}(0,T')}^2
\end{aligned}
\end{equation}
with the various constants depending on $\alpha,\beta,\delta,$ and $j$.
\\[2mm]
\textit{Case 2: $\alpha^2+3\beta\delta\leq 0$}. 
Then,  $\lambda>\frac{2\sqrt{-(\alpha^2+3\beta\delta)}}{3\beta}>0$ and, by the definition of $\gamma_2$, 
\begin{equation}\label{q2midest}
			|q_2(x,t)|\le \frac{1}{2\pi}e^{-\lambda x}\int_{c_-}^{c_+}|m^2+\lambda^2|^{\frac{j}{2}}\left|e^{-(\omega(m+i\lambda)t} \, \widetilde{g}_0(\omega(m+i\lambda),T') \, \omega'(m+i\lambda)\right| dm.
\end{equation}
		Recall that for $k\in D^+$, we have Re$(\omega(k))<0$, which implies  $|e^{\omega(m+i\lambda)\rho}|\le 1$ for $m\in (c_-,c_+)$, $\rho\in [0,T']$. Therefore, similarly to \eqref{g0sqrtT}, 
		\begin{equation}\label{q2auxest}|\widetilde{g}_0(\omega(m+i\lambda),T')|\le \sqrt{T'} \, \|g_0\|_{L_t^2(0,T')}.\end{equation}
Combining \eqref{q2midest} and \eqref{q2auxest}, we deduce
$$
|q_2(x,t)|\le c \sqrt{T'}e^{cT'}\|g_0\|_{L_t^2(0,T')}e^{-\lambda x}.
$$ 
Taking the square of the above inequality, integrating with respect to $x \in (0,\infty)$ (for this step, recall that $\lambda>0$), and then taking square roots, we obtain
		$$\|q_2(\cdot,t)\|_{L^2_x(\mathbb{R}_+)}\le \frac{c}{\sqrt{2\lambda}} \sqrt{T'}e^{cT'}\|g_0\|_{L_t^2(0,T')}\lesssim  \sqrt{T'}e^{cT'}\|g_0\|_{H_t^{\frac{j+1}{3}}(0,T')},$$ where the constant of the last inequality depends only on $\alpha,\beta,\delta$ and $j$.
		
The desired estimate \eqref{ribvp-se} has been established for $s\in \mathbb{N}_0$. The proof for $s\ge 0$ follows by interpolation, e.g. see Theorem 5.1 in  \cite{lm1972}. Make a remark about the possibility of using the fractional norm along the lines of the nonhomogeneous Cauchy problem and \cite{fhm2017}.
\end{proof}

\begin{proof}[Proof of Lemma \ref{suplem}]
First, we make a few observations. From the definition \eqref{wkdef} of $\omega$ and the triangle inequality,  
$$
%			\omega(k):=-i\beta k^3+i \alpha k^2+i\delta k
%			\quad
|\omega(k)| 
\geq 
\beta |k|^3 - | \alpha k^2+\delta k|
\geq
\beta |k|^3 - \left( |\alpha| |k|^2 + |\delta| |k| \right).
$$
%		
%			$$\frac{|i\omega(k)|}{\beta}\ge |k|^3-\left(\frac{|\alpha|}{\beta}|k|^2+\frac{|\delta|}{\beta}|k|\right), \quad k\in \partial D^+.$$ 
%		
In addition, for $|k|\ge \frac{|\alpha|+\sqrt{\alpha^2+2\beta |\delta|}}{\beta}$ we have $|\alpha| |k|^2+ |\delta| |k|\le\frac{1}{2} \beta |k|^3$ and so, noting also that $\text{Re}(\omega(k)) = 0$ along $\partial D^+$,  
%
%			Therefore, for $k\in \partial D^+$ and $|k|$ large enough, we obtain
\begin{equation}\label{iwkest}
|i\omega(k)|\ge \frac{\beta}{2}|k|^3
\
\Rightarrow 
\ 
\frac{1}{1+[i\omega(k)]^2}
\leq \frac{1}{1+\frac{\beta^2}{4} |k|^6}
\simeq \frac{1}{1+|k|^6}.
\end{equation}
Observe further that $|\gamma_1(m)|\ge m$ thus $|\gamma_1(m)|$ can be made as large as we wish by taking $m\in [\lambda,\infty)$ large enough. Therefore, using \eqref{iwkest}, for large enough $m$ we  have   
			$$\frac{1}{1+\tau^2(m)}=\frac{1}{1+[i\omega(\gamma_1(m)]^2}\lesssim \frac{1}{1+|\gamma_1(m)|^6}.$$
			On the other hand, for $|k|\geq1$ we have $|k|^2\geq |k|$ and so by the triangle inequality
			$$|\omega'(k)|\le 3\beta|k|^2+2|\alpha\|k|+|\delta|\le (3\beta+2|\alpha|+|\delta|)(1+|k|^2).$$
			From the definition of $\gamma_1$, there exist non-negative constants $c_1, c_2$ depending on $\alpha,\beta,\delta$ such that
			$$|\gamma_1(m)|\le c_1m \text{ and } |\gamma_1'(m)|\le c_2, \quad m\in [\lambda, \infty).$$ Hence, there are some constants $c_3>0$, $M\ge \lambda$ depending on $\alpha,\beta,\delta$ such that 
$$\frac{|\gamma_1(m)|^{2j}|\left|\omega'(\gamma_1(m)) \gamma_1'(m)\right|}{[1+\tau^2(m)]^{\frac{j+1}{3}}}\le c_3\frac{(|\gamma_1(m)|^2)^{j}(1+|\gamma_1(m)|^2)}{(1+|\gamma_1(m)|^2)^{j+1}}\le c_3, 
\quad m>M.
$$ 
However, by continuity of the function on the left-hand side on the compact interval $[\lambda,M]$, there is also some constant $c_4>0$ depending on $\alpha,\beta,\delta$ such that	
$$
\frac{|\gamma_1(m)|^{2j}|\left|\omega'(\gamma_1(m)) \gamma_1'(m)\right|}{[1+\tau^2(m)]^{\frac{j+1}{3}}}\le c_4, \quad m\in [\lambda,M].
$$  
Combining the last two inequalities yields the desired estimate with $c = \max\{c_3,c_4\}<\infty$.
\end{proof}

\subsubsection*{Strichartz-type estimates}
It turns out convenient to reparametrize the contour of integration in the solution formula \eqref{representation} of the reduced initial-boundary value problem \eqref{linear4} as  
$\Gamma=\Gamma_1\cup\Gamma_2\cup\Gamma_3$ with
\begin{alignat}{2}
		\Gamma_1(m)&=m+i\sqrt{3\Big(m-\frac{\alpha}{3\beta}\Big)^2-\frac{\alpha^2+3\beta\delta}{3\beta^2}},\quad &&-\infty < m\leq c_-,
		\nonumber\\
		\Gamma_2(m)&=m+i\lambda, &&c_-< m< c_+,
		\label{Gamma-par}
		\\
		\Gamma_3(m)&=m+i\sqrt{3\Big(m-\frac{\alpha}{3\beta}\Big)^2-\frac{\alpha^2+3\beta\delta}{3\beta^2}}, &&c_+\leq m<\infty,
		\nonumber
	\end{alignat} 
where, as before, $c_{\pm}=\frac{\alpha\pm\sqrt{3\beta^2\lambda^2+\alpha^2+3\beta\delta}}{3\beta}$ and $\lambda>0$ satisfies \eqref{cursornumber}.  With this parametrization, formula \eqref{representation} can be expressed as the sum
$$
		q(x,t)=-\frac{i}{2\pi}\sum_{j=1}^{3}\int_{\Gamma_j}e^{ikx-\omega(k)t} \, \widetilde{g}_0(\omega(k),T')\omega'(k)dk
		=: \sum_{j=1}^3 q_j(x,t).
$$

We first consider $q_1$, which after recalling also \eqref{ft-comp} takes the form
\begin{equation}\label{q1-form}
\begin{aligned}
		q_1(x,t)&=-\frac{i}{2\pi}\int_{-\infty}^{c_-}e^{i\Gamma_1(m)x-\omega(\Gamma_1(m))t} \, \widehat{g}_0(i\omega(\Gamma_1(m)))\, \omega'(\Gamma_1(m))\Gamma'_1(m)dm\\
		&=\frac{1}{2\pi}\int_{-\infty}^{c_-}e^{i\Gamma_1(m)x-\omega(\Gamma_1(m))t}\bigg(\int_{-\infty}^{\infty} e^{-imy}\Psi_1(y)dy \bigg)dm
\end{aligned}
\end{equation} 
where $\Psi_1$ is the inverse Fourier transform of 
$$
		\widehat{\Psi}_1(m):=\left\{
		\begin{array}{ll}
			-i\widehat{g}_0(i\omega(\Gamma_1(m)))\, \omega'(\Gamma_1(m))\Gamma'_1(m), & m\leq c_-, \\
			0, & m> c_-.
		\end{array}
		\right.
$$
Then, introducing the kernel 
	\begin{equation}\label{kernel}
		\mathcal{K}(y;x,t)=\int_{-\infty}^{c_-}e^{i\phi(m;x,y,t)}p(m;x)dm
	\end{equation}
with amplitude 
	\begin{equation}
		p(m;x)=e^{-x\sqrt{3\big(m-\frac{\alpha}{3\beta}\big)^2-\frac{\alpha^2+3\beta\delta}{3\beta^2}}}
	\end{equation}
and phase
\begin{equation}\label{phase-def}
\begin{aligned}
\phi(m;x,y,t)
&=m(x-y)+i\omega(\Gamma_1(m))t
\\
&=m(x-y)+t\left[-8\beta m^3+8\alpha m^2+2\left(\delta-\frac{\alpha^2}{\beta}\right)m-\frac{\alpha\delta}{\beta}\right],
\end{aligned}
\end{equation}
we can rearrange \eqref{q1-form} in the form 
\begin{equation}\label{q1-form2}
q_1(x,t)=[K_1(t)\Psi_1](x):=\frac{1}{2\pi}\int_{-\infty}^{\infty}\mathcal{K}(y;x,t)\Psi_1(y)dy.
\end{equation}
This writing provides the starting point for proving the following central estimate of Strichartz type.

%$$
%\left\|f\right\|_{H^{s,r}(\mathbb{R})}
%:=
%\left\| \mathcal{F}^{-1}\left\{ \left(1+k^2\right)^{\frac{s}{2}} \mathcal %F\{f\}(k)\right\}\right\|_{L^{r}(\mathbb{R})}
%$$
%
\begin{thm}\label{BdrStrThm}
Let $s\ge 0$ and $(\mu,r)$ be higher-order Schr\"odinger admissible in the sense of \eqref{admissiblepair}. Then,
\begin{equation}\label{stric8}
\|q\|_{L_t^{\mu}((0,T');H_x^{s,r}(\mathbb{R}_+))} 
\lesssim 
\big(1+(T')^{\frac{1}{\mu}+\frac{1}{2}}\big) \|g_0\|_{H_t^{\frac{s+1}{3}}(0,T')}
\end{equation} 
where $H_x^{s,r}(\mathbb{R}_+)$ is the restriction on $\mathbb{R}_+$ of the Bessel potential space $H_x^{s,r}(\mathbb{R})$ defined by \eqref{bessel-def} and the inequality constant depends only $r, s$.
\end{thm}

	\begin{proof} We will use a standard duality argument. Let $\eta\in C_c([0,T'];\mathcal{D}(\mathbb{R}_+))$ be an arbitrary function. Then,
\begin{equation}\label{dualityarg}
\begin{aligned}
2\pi \, \bigg|\int_{0}^{T'}
\left\langle K_1(t)\Psi_1,\eta(\cdot,t) \right\rangle_{L_x^2(\mathbb{R}_+)}dt\bigg| 
&=
\left| \int_{0}^{T'} 
\int_{0}^{\infty}\left(\int_{-\infty}^{\infty} \mathcal{K}(y;x,t)\Psi_1(y)dy\right) \overline{\eta(x,t)}dx dt\right|
\\
&= 
\int_{-\infty}^{\infty}\Psi_1(y)\;\overline{\int_{0}^{T'} \int_{0}^{\infty}\overline{\mathcal{K}(y;x,t)}\;\eta(x,t)dxdt} \, dy
\\
&\leq 
\|\Psi_1\|_{L^2(\mathbb{R})}\left\|  \int_{0}^{T'}\int_{0}^{\infty}\overline{\mathcal{K}(y;x,t)}\;\eta(x,t)dxdt \right\|_{L_y^2(\mathbb{R})}.
\end{aligned}
\end{equation}
Set $K_2(y):=\displaystyle \int_0^{T'}\int_0^{\infty} \overline{\mathcal{K}(y;x,t)}\eta(x,t)dxdt$. By the definition of the $L^2$-norm, we have
\begin{align*}
\| K_2 \|_{L^2(\mathbb{R})}^2
&=
\int_{-\infty}^{\infty} \bigg( \int_0^{T'}\int_0^{T'}\int_0^{\infty}\int_0^{\infty}\overline{\mathcal{K}(y;x,t)}\eta(x,t)\mathcal{K}(y;x',t')\overline{\eta(x',t')}dxdx'dtdt' \bigg)dy
\\
&=
\int_0^{T'}\int_0^{T'}\int_0^{\infty}\int_0^{\infty}\eta(x,t)\overline{\eta(x',t')}K_3(x,x';t,t')dxdx'dtdt'
\\
&=
\int_0^{T'}\int_0^{\infty}\eta(x,t)\bigg(\int_0^{T'}\int_0^{\infty}\overline{\eta(x',t')}K_3(x,x';t,t')dx'dt'\bigg)dxdt
\end{align*}
where $K_3(x,x';t,t'):=\displaystyle \int_{-\infty}^{\infty} \overline{\mathcal{K}(y;x,t)}\mathcal{K}(y;x',t')dy$.
Then, by Hölder's inequality in $(x,t)$ and then Minkowski's integral inequality between $x$ and $t'$ we deduce
\begin{equation}\label{k2midest}
\begin{aligned}
\| K_2 \|_{L^2(\mathbb{R})}^2
&\leq
\|\eta\|_{L_t^{\mu'}((0,T');L_x^{r'}(\mathbb{R}_+))}\left\| \int_0^{T'}\int_0^{\infty}\overline{\eta(x',t')}K_3(x,x';t,t')dx'dt'\right\|_{L_t^{\mu}((0,T');L_x^{r}(\mathbb{R}_+))}
\\
&\leq
\|\eta\|_{L_t^{\mu'}((0,T');L_x^{r'}(\mathbb{R}_+))}
\left\| \int_0^{T'}
\left\|
\int_0^{\infty}\overline{\eta(x',t')}K_3(x,x';t,t')dx'
\right\|_{L_x^{r}(\mathbb{R}_+)}
dt'\right\|_{L_t^{\mu}(0,T')}.
\end{aligned}
\end{equation}

We begin with the estimation of the interior $L_x^r(\mathbb R_+)$-norm.
Using the definition \eqref{kernel} of $\mathcal K$, we rewrite $K_3$ in the form of an oscillatory integral:
\begin{equation*}%\label{K3rewrite}
\begin{aligned}
K_3(x,x';t,t')&=\int_{-\infty}^{\infty}\bigg( \int_{-\infty}^{c_-}e^{-i\phi(m;x,y,t)}p(m;x)dm \bigg) \bigg( \int_{-\infty}^{c_-}e^{i\phi(m';x',y,t')}p(m';x')dm' \bigg)dy
\\
&=\int_{-\infty}^{c_-}p(m;x) \int_{-\infty}^{\infty}e^{-i\phi(m;x,y,t)}\left(\int_{-\infty}^{c_-}e^{i\phi(m';x',y,t')}p(m';x')dm'\right)dy dm.
\end{aligned}
\end{equation*}
Recalling the definition \eqref{phase-def} of the phase function $\phi$ and introducing the function
\begin{equation*}%\label{Qdeffourier}
Q(m';x',t'):=
\left\{
\begin{array}{ll}
e^{im'x'-\omega(\Gamma_1(m'))t'}p(m';x'), &m'\in (-\infty, c_-],
\\
0, &m'\in (c_-,\infty),
\end{array}
\right.
\end{equation*}
we have via the Fourier inversion theorem
\begin{align*}
&\quad
\int_{-\infty}^{\infty}e^{-i\phi(m;x,y,t)}\left(\int_{-\infty}^{c_-}e^{i\phi(m';x',y,t')}p(m';x')dm'\right)dy
\\
&=
2\pi e^{-imx+\omega(\Gamma_1(m))t} \cdot \frac{1}{2\pi}\int_{-\infty}^{\infty}e^{imy}\left(\int_{-\infty}^{\infty}e^{-im'y}Q(m';x',t')dm'\right)dy
\\
&=2\pi e^{-imx+\omega(\Gamma_1(m))t} \, Q(m;x',t').
%\\
%&=2\pi e^{-im(x-x')+\omega(\Gamma_1(m))(t-t')}p(m;x')=2\pi e^{-i\phi(m;x,x',t-t')}p(m;x'), \quad m\in (-\infty,c_-).
\end{align*}
Thus, for $m\in (-\infty,c_-]$ we deduce
$$
\int_{-\infty}^{\infty}e^{-i\phi(m;x,y,t)}\left(\int_{-\infty}^{c_-}e^{i\phi(m';x',y,t')}p(m';x')dm'\right)dy
%=2\pi e^{-im(x-x')+\omega(\Gamma_1(m))(t-t')}p(m;x')
=2\pi e^{-i\phi(m;x,x',t-t')}p(m;x')
$$
and, consequently,
\begin{equation*}
K_3(x,x';t,t')=2\pi\int_{-\infty}^{c_-}e^{-i\phi(m;x,x',t-t')}p(m;x+x')dm.
\end{equation*}

Next, we employ the following fundamental result.
\begin{lem}\label{vander2}
Let $K(x,y,z,t)=\displaystyle \int_{-\infty}^{c_-}e^{i\phi(m;x,y,t)}p(m;z)dm,$ where $x,z\in\mathbb{R}_+$ and $y, t \in \mathbb{R}$.  Then, 
		\begin{equation}
			|K(x,y,z,t)|\lesssim |t|^{-\frac 13}, \quad t\neq 0,
		\end{equation}
		where the constant of the inequality is independent of $x,y,z,t$.
\end{lem}
The proof of Lemma \ref{vander2} relies on the classical van der Corput lemma and is provided after the end of the current proof.
Observe that Lemma \ref{vander2} with $y, z, t$ replaced respectively by $x', x+x', t-t'$ yields 
$$
|K_3(x,x',t,t')|\lesssim |t-t'|^{-\frac 13}, \quad t\neq t',
$$ 
with inequality constant independent of $x$, $x'$, $t$ and $t'$.
This dispersive estimate implies 
		\begin{equation}\label{stric1}
			\left\|\int_0^{\infty}\overline{\eta(x',t')}K_3(x,x';t,t')dx'\right\|_{L_x^{\infty}(\mathbb{R}_+)}\lesssim |t-t'|^{-\frac 13}\|\eta(t')\|_{L_x^1(\mathbb{R}_+)}.
		\end{equation}
		On the other hand, we also have
		\begin{equation}\label{stric2}
			\left\|\int_0^{\infty}\overline{\eta(x',t')}K_3(x,x';t,t')dx'\right\|_{L_x^{2}(\mathbb{R}_+)}\lesssim \|\eta(t')\|_{L_x^2(\mathbb{R}_+)}.
		\end{equation}
Indeed, we have
\begin{align}
&\quad
\left\|\int_0^{\infty}\overline{\eta(x',t')}K_3(x,x';t,t')dx'\right\|_{L_x^{2}(\mathbb{R}_+)}^2
=\int_{0}^{\infty}\bigg| \int_{0}^{\infty}\overline{\eta(x',t')}K_3(x,x';t,t')dx' \bigg|^2 \,dx
\nonumber\\
&= (2\pi)^2\int_{0}^{\infty}\bigg| \int_{0}^{\infty}\overline{\eta(x',t')}\left(\int_{-\infty}^{c_-}e^{-i\phi(m;x,x',t-t')}p(m;x+x')dm\right)dx' \bigg|^2dx
\nonumber\\
&\le 
%(2\pi)^2\int_{0}^{\infty}\Bigg(\int_{0}^{\infty}|\eta(x',t')|\bigg(\int_{-\infty}^{c_-}e^{-(x+x')\sqrt{3\big(m-\frac{\alpha}{3\beta}\big)^2-\frac{\alpha^2+3\beta\delta}{3\beta^2}}}dm\bigg)dx'\Bigg)^2dx
%\nonumber\\
%&=
(2\pi)^2\int_{0}^{\infty}\left(\int_{-\infty}^{c_-}e^{-xs(m)}\left(\int_{0}^{\infty}e^{-x's(m)}|\eta(x',t')|dx'\right)dm\right)^2dx,
\nonumber
%\label{K3eta}
\end{align} where
$s(m)=\sqrt{3\big(m-\frac{\alpha}{3\beta}\big)^2-\frac{\alpha^2+3\beta\delta}{3\beta^2}}$.
The claimed estimate \eqref{stric2} then directly follows  by invoking the following lemma, which provides a generalization of the $L^2(\mathbb R_+)$-boundedness of the Laplace transform given in \cite{fhm2017} and is established after the end of the current proof.
\begin{lem}\label{modLTbound}The estimates in $(i)$ and $(ii)$ below hold true for $f\in L_m^2(-\infty,c_-)$ and $f\in L_x^2(\mathds{R}_+)$, respectively.
\begin{align*}
&\textnormal{(i)}\quad \left\|\int_{-\infty}^{c_-}e^{-xs(m)}f(m)dm\right\|_{L_x^2(\mathbb{R}_+)}\lesssim \|f\|_{L_m^2(-\infty,c_-)},\\
&\textnormal{(ii)} \quad \left\|\int_{0}^{\infty}e^{-xs(m)}f(x)dx\right\|_{L_m^2(-\infty,c_-)}\lesssim \|f\|_{L_x^2(\mathbb{R}_+)}.
\end{align*}
\end{lem}

Now, (\ref{stric1}) and (\ref{stric2}) together with Riesz-Thorin interpolation theorem yield for any $r\ge 2$ that
		\begin{equation}\label{stric3}
			\left\|\int_0^{\infty}\overline{\eta(x',t')}K_3(x,x';t,t')dx'\right\|_{L_x^{r}(\mathbb{R}_+)}\lesssim |t-t'|^{-\frac2\mu}\|\eta(t')\|_{L_x^{r'}(\mathbb{R}_+)},
		\end{equation}
where  $\frac{1}{r'}=1-\frac{1}{r}$ and we have also used \eqref{admissiblepair}. Hence,  for any $\eta\in L_t^{\mu'}((0,T');L_x^{r'}(\mathbb{R}_+))$  we obtain
\begin{equation*}
 \int_0^{T'} \left\|\int_0^{\infty}\overline{\eta(x',t')}K_3(x,x';t,t')dx'\right\|_{L_x^r(\mathbb{R})} dt'
\lesssim 
%\int_0^{T'}|t-t'|^{-(\frac13-\frac{2}{3r})}\|\eta(t')\|_{L_x^{r'}(\mathbb{R})}dt'
%\\
%			&\cong
\int_0^{T'}|t-t'|^{-\frac{2}{\mu}}\|\eta(t')\|_{L_x^{r'}(\mathbb{R})}dt'.
\end{equation*}
Handling the right-hand side via Hardy-Littlewood-Sobolev fractional integration (e.g. see  Theorem 1 on page 119 of \cite{s1970}) and combining the resulting inequality with \eqref{k2midest}, we infer
%
%Riesz potential inequalities		
%%
%\begin{thm}[Hardy-Littlewood-Sobolev Fractional Integration Theorem -- \cite{s1970}, p. 119, Thm. 1]
%\label{hls-t}
%%
%Let $0<\gamma<1$, $1\leqslant r<q<\infty$, $\frac 1q = \frac 1r -  \gamma$. If $f\in L^r(\mathbb R)$, then the integral 
%%
%\begin{equation}
%\left(I_\gamma f\right)(t) \doteq   \int_{t'\in\mathbb R} \left|t-t'\right|^{-1+\gamma} f(t') dt'
%\end{equation}
%%
%converges absolutely for almost every $t$. If, in addition, $1<r$, then 
%%
%\begin{equation}
%\left\|I_\gamma f\right\|_{L^q(\mathbb R_t)} \leqslant c_{r, q, \gamma} \left\|f\right\|_{L^r(\mathbb R_t)}.
%\end{equation}
%%
%\end{thm}
%%
%\begin{equation}\label{k3etaest}
%			\left\| \int_0^{T'}\int_0^{\infty}\overline{\eta(x',t')}K_3(x,x';t,t')dx'dt'\right\|_{L_t^{\mu}((0,T');L_x^{r}(\mathbb{R}_+))}\lesssim \|\eta\|_{L_t^{\mu'}((0,T');L_x^{r'}(\mathbb{R}_+))}.
%		\end{equation}
%
$$
\| K_2 \|_{L^2(\mathbb{R})}\lesssim \|\eta\|_{L_t^{\mu'}((0,T');L_x^{r'}(\mathbb{R}_+))}
$$ 
which can be combined with \eqref{dualityarg} to yield
		\begin{equation}\label{stric5}
			\|q_1\|_{L_t^{\mu}((0,T');L_x^{r}(\mathbb{R}_+))}\lesssim \|\Psi_1\|_{L^2(\mathbb{R})}.
		\end{equation}

Differentiating the expression \eqref{q1-form} $j$ times in $x$ and repeating the above arguments,  for any $j\in \mathbb{N}_0$ we conclude that
		\begin{equation}\label{jstric5}
			\|\partial_x^jq_1\|_{L_t^{\mu}((0,T');L_x^{r}(\mathbb{R}_+))}\lesssim \|\partial_x^j\Psi_1\|_{L^2(\mathbb{R})}\lesssim \|g_0\|_{H_t^{\frac{j+1}{3}}(\mathbb{R})}.
		\end{equation} 
Observe that the left-hand side of estimate \eqref{jstric5} is simply the  $L_t^\mu((0, T'); W^{j, r}(\mathbb R_+))$-norm of $q_1$.
In this connection, note that, according to a classical result by Calder\'on \cite{c1961}, for any $j\in\mathbb N_0$, $1<r<\infty$ the Sobolev space $W^{j, r}(\mathbb R)$ and the Bessel potential space $H^{j, r}(\mathbb R)$ coincide (i.e. they are equal as sets). 
Indeed, it is not hard to show that $W^{j, r}(\mathbb R) \subseteq H^{j, r}(\mathbb R)$. 
On the other hand, showing that $H^{j, r}(\mathbb R) \subseteq W^{j, r}(\mathbb R)$ is more involved; see page 22 of \textsection 2.3 in \cite{g2014m}, where the result is proved with the help of the Mikhlin-H\"ormander theorem for $L^p$ multipliers (Theorem 6.2.7 in \cite{g2014c}). Thus, for any $j\in\mathbb N_0$, $1<r<\infty$ we have $\left\| \cdot \right\|_{H^{j, r}(\mathbb R)} \simeq \left\| \cdot \right\|_{W^{j, r}(\mathbb R)}$ and so
\begin{equation}\label{eq-norms}
\left\| q_1 \right\|_{W^{j, r}( \mathbb R_+)}
:=
\inf_{\substack{\tilde q_1 \in W^{j, r}(\mathbb R) \\ \tilde q_1|_{\mathbb R_+} = q_1}} \left\| \tilde q_1 \right\|_{W^{j, r}(\mathbb R)}
\simeq
\inf_{\substack{\tilde q_1 \in H^{j, r}(\mathbb R) \\ \tilde q_1|_{\mathbb R_+} = q_1}} \left\| \tilde q_1 \right\|_{H^{j, r}(\mathbb R)}
=:
\left\| q_1 \right\|_{H^{s, r}(\mathbb R_+)}.
\end{equation}
Observing that the left-hand side of estimate \eqref{jstric5} is simply the $W^{j, r}(\mathbb R_+)$-norm of $q_1$, in view of \eqref{eq-norms} we see that \eqref{jstric5} is in fact equivalent to
\begin{equation}\label{jstric52}
\|q_1\|_{L_t^{\mu}((0,T');H_x^{j, r}(\mathbb{R}_+))}\lesssim \|g_0\|_{H_t^{\frac{j+1}{3}}(\mathbb{R})}.
\end{equation} 
Finally, by interpolation (e.g. see Theorem 5.1 in  \cite{lm1972}) we deduce
\begin{equation}\label{stric7}
			\|q_1\|_{L_t^{\mu}((0,T');H_x^{s,r}(\mathbb{R}_+))}\lesssim \|g_0\|_{H_t^{\frac{s+1}{3}}(\mathbb{R})}, \quad s\geq 0,
		\end{equation} 
completing the estimation of $q_1$.

		In order to estimate $\|q_2\|_{L_t^{\mu}((0,T');H_x^{s,r}(\mathbb{R}_+))}$, we use \eqref{q2rewrite}-\eqref{g0sqrtT} (note the difference in notation, as $q_2$ in those expressions now corresponds to $\partial_x^j q_2$) as the portions $\gamma_2$ and $\Gamma_2$ of the two parametrizations \eqref{gamma-par} and \eqref{Gamma-par} coincide.
In particular, 
\begin{align*}
&\quad
\|\partial_x^jq_2(\cdot,t)\|_{L_x^r(\mathbb{R}_+)}=\left(\int_{-\infty}^{\infty}|Q_2(m,t)|^rdm\right)^{\frac{1}{r}}
\\
&\le\bigg((T')^{\frac{r}{2}}\|g_0\|_{L_t^2(0,T')}^r\int_{c_-}^{c_+}m^{jr}|\omega'(m)|^rdm\bigg)^{\frac{1}{r}}=c_{j,r}\sqrt{T'}\|g_0\|_{L_t^2(0,T')}.
\end{align*} 
Therefore, for any $j\in\mathbb N_0$ we find
\begin{equation*}
\|\partial_x^jq_2\|_{L_t^{\mu}((0,T');L_x^{r}(\mathbb{R}_+))}
%=\left(\int_0^{T'}\|\partial_x^jq_2(\cdot,t)\|_{L_x^r(\mathbb{R}_+)}^\mu dt\right)^{\frac{1}{\mu}}
=c_{j,r}(T')^{\frac{1}{\mu}+\frac{1}{2}}\|g_0\|_{L_t^2(0,T')}\\
\le
c_{j,r}(T')^{\frac{1}{\mu}+\frac{1}{2}}\|g_0\|_{H_t^{\frac{j+1}{3}}(0,T')},
\end{equation*}
and, using again the equivalence of the Bessel potential and Sobolev norms \eqref{eq-norms} along with interpolation, we conclude that
		\begin{equation}\label{stricq2}
			\|q_2\|_{L_t^{\mu}((0,T');H_x^{s,r}(\mathbb{R}_+))}\lesssim (T')^{\frac{1}{\mu}+\frac{1}{2}}\|g_0\|_{H_t^{\frac{s+1}{3}}(0,T')}, \quad s\geq 0.
\end{equation} 

As the estimation of $q_3$ is similar to that of $q_1$, the proof of Theorem \ref{BdrStrThm} is complete.
\end{proof}

\begin{proof}[Proof of Lemma \ref{vander2}]
By the Fundamental Theorem of Calculus, $K$ can be rewritten as
$$
K(x,y,z,t) = - \int_{-\infty}^{c_-} \frac{d I}{d m}(m;x,y,t) \, p(m;z) dm,
$$
where $I(m;x,y,t):= \displaystyle \int_{m}^{c_-}e^{i\phi(\xi;x,y,t)}d\xi$. 
Integrating by parts using the fact that $I(c_-;x,y,t)=0$ and $p(m;z)\rightarrow 0$ as $m\rightarrow -\infty$, and noting also that $\displaystyle \frac{dp}{dm}(m;z)>0$, we get
$$
|K(x,y,z,t)|\leq\int_{-\infty}^{c_-} \left|I(m;x,y,t)\right| \frac{dp}{dm}(m;z) \, dm.
$$
According to van der Corput's lemma (e.g. see page 370 in \cite{stein1986}), if $\eta(\xi)$ is a real function such that $|\eta^{(j)}(\xi)| \geq \alpha$ on $[a, b]$, with the additional condition that $\eta'(\xi)$ is monotone if $j=1$, then 
$$
\bigg|\int_a^b e^{i\eta(\xi)} d\xi\bigg| \leq c_j \alpha^{-\frac 1j}.
$$
Noting that $|\phi^{(3)}(\xi; x, y, t)| = 48\beta|t|$, we can employ this classical result for $I$ with $\eta(\xi) = \phi(\xi; x, y, t)$ to infer that $|I(m;x,y,t)|\lesssim |t|^{-\frac 13}$, $t\neq 0$, where the constant of inequality is independent of $m, x, y, t$.  In turn, for any $t\neq 0$ and $z>0$ we obtain
$$
|K(x,y,z,t)|
\lesssim |t|^{-\frac 13}\int_{-\infty}^{c_-}\frac{dp}{dm}(m;z) dm
%\lesssim |t|^{-\frac13}p(m;z)|_{m=-\infty}^{m=c_-}
=
|t|^{-\frac13}e^{-\lambda z}
\leq |t|^{-\frac 13}, 
$$
which is the desired estimate.
\end{proof}
	
\begin{proof}[Proof of Lemma \ref{modLTbound}]
			First, we prove part (i).
			By definition of $s(m)$, we have
			$$\frac{ds(m)}{dm}=\frac{3(m-\frac{\alpha}{\beta})}{s(m)}=-\sqrt{3} \, \frac{\sqrt{s^2(m)+c_{\alpha,\beta,\delta}}}{s(m)}$$ with $c_{\alpha,\beta,\delta}=\frac{\alpha^2+3\beta\delta}{3\beta^2}$. Therefore, upon change of variable $s=s(m)$, we get
			$$\int_{-\infty}^{c_-}e^{-xs(m)}f(m)dm=\frac{1}{\sqrt{3}}\int_\lambda^\infty e^{-xs}f(m(s))\frac{s}{\sqrt{s^2+c_{\alpha,\beta,\delta}}}ds=\int_0^\infty e^{-xs} f_\lambda(s)ds,$$ where
			$$f_\lambda(s):=\begin{cases}
				\frac{1}{\sqrt{3}}f(m(s))\frac{s}{\sqrt{s^2+c_{\alpha,\beta,\delta}}} & s\in (\lambda,\infty),\\
				0, &s\notin (\lambda,\infty).
			\end{cases}$$
			Using the $L^2$-boundedness of the Laplace transform (see Lemma 3.2 in \cite{fhm2017}), we get
			$$\left\|\int_0^\infty e^{-xs} f_\lambda(s)ds\right\|_{L_x^2(\mathbb{R}_+)}\lesssim \|f_\lambda\|_{L^2_s(\mathbb{R}_+)}.
$$
Finally, note that
\begin{align*}
				\|f_\lambda\|_{L^2_s(\mathbb{R}_+)}^2&=\frac{1}{3}\int_\lambda^\infty|f(m(s))|^2\frac{s^2}{{s^2+c_{\alpha,\beta,\delta}}}ds\\
				&=\frac{1}{\sqrt{3}}\int_{-\infty}^{c_-}\left|f(m)\right|^2\frac{s(m)}{\sqrt{s^2(m)+c_{\alpha,\beta,\delta}}}dm\lesssim \int_{-\infty}^{c_-}\left|f(m)\right|^2dm.
\end{align*}

Next, we establish part (ii). Setting $F(m):=\displaystyle \int_{0}^{\infty}e^{-xs(m)}f(x)dx$ and using the Cauchy-Schwarz inequality along the lines of the proof of Lemma 3.2 in \cite{fhm2017}, we have
			\begin{align*}
				|F(m)|^2&=\left|\int_{0}^{\infty}e^{-xs(m)}f(x)dx\right|^2=\left|\int_{0}^{\infty}e^{-\frac{xs(m)}2}f(x)x^{\frac 14} \cdot x^{-\frac 14}e^{-\frac{xs(m)}2}dx\right|^2\\
				&\le \left(\int_{0}^{\infty}e^{-xs(m)}|f(x)|^2x^{\frac 12}dx\right)\left(\int_{0}^{\infty}e^{-xs(m)}x^{-\frac 12}dx\right).
			\end{align*}
Then, since the second integral on the right-hand side is equal to 
$\displaystyle \frac{1}{{\sqrt{s(m)}}}\int_{0}^{\infty}e^{-u}u^{-\frac 12}{du}=\frac{\sqrt{\pi}}{\sqrt{s(m)}}$, 
			$$\int_{-\infty}^{c_-}|F(m)|^2dm=\sqrt{\pi}\int_{-\infty}^{c_-}\left(\int_{0}^{\infty}\frac{1}{\sqrt{s(m)}} \, e^{-xs(m)}|f(x)|^2x^{\frac 12}dx\right)dm.$$
Finally, noting that
			\begin{align*}\int_{-\infty}^{c_-}\frac{1}{\sqrt{s(m)}}\, x^{\frac 12}e^{-xs(m)}dm&=\frac{1}{\sqrt{3}}\int_\lambda^\infty e^{-xs}s^{-\frac 12}x^{\frac 12}\frac{s}{\sqrt{s^2+c_{\alpha,\beta,\delta}}} \, ds \lesssim \int_\lambda^\infty e^{-xs}s^{-\frac 12}x^{\frac 12}ds\le c\sqrt{\pi},
\end{align*}
we arrive at the desired estimate 
$$
\int_{-\infty}^{c_-}|F(m)|^2dm\lesssim \int_0^\infty |f(x)|^2dx
$$ 
completing the proof of the lemma.
\end{proof}

\section{Nonlinear analysis}
\label{nonlinear-s}

The various linear estimates established in Section \ref{linear-s} will now be combined with a contraction mapping argument in order to establish local well-posedness in the sense of Hadamard for the nonlinear initial-boundary value problem \eqref{nonlinear1}. In view of these linear results, the solution space will change as we transition from the setting of high regularity ($\frac 12<s\leq 2$) to the one of low regularity ($0\leq s < \frac 12$). More specifically, in the former case  well-posedness will be established in the space $C([0, T]; H_x^s(\mathbb R_+))$ for a appropriate choice of $T>0$ (see Theorem \ref{HighRegThm}), while in the latter case that space will be refined by intersecting it with the Strichartz-inspired space $L_t^\mu((0, T); H_x^{s, r}(\mathbb R_+))$ for an admissible choice of exponents $(\mu, r)$ in terms of the nonlinearity order $p$ and the Sobolev exponent $s$ according to \eqref{admissiblepair} (see Theorem \ref{LowRegThm}). 
	
\subsection{\ttfamily\bfseries Linear reunification}\label{LinSolFormula}
The nonlinear analysis will be performed by using a solution operator $u\mapsto \Phi u$ associated with the original forced linear initial-boundary value problem \eqref{linear1}. To this end, thanks to the superposition principle we reunify the solution representation formulae corresponding to (i) the homogeneous Cauchy problem \eqref{cauchy}, (ii) the nonhomogeneous Cauchy problem \eqref{nonhomogeneous1}, and (iii)  the reduced initial-boundary value problem \eqref{linear4}. More precisely, given $u$, we formally define the map
	\begin{equation}\label{solformula}
	\begin{aligned}
		\Phi u&:=y|_{Q_T}+z^u|_{Q_T}+q^u|_{(0,T)}\\
		&\equiv 
		S[E_0u_0;0]\big|_{Q_T} + 
		S[0;f(Eu)]\big|_{Q_T}
		-\frac{i}{2\pi}\int_{\Gamma}e^{ikx-\omega(k)t}\omega'(k) \, \widetilde{g}_0^u(\omega(k),T')dk\Big|_{(0,T)},
	\end{aligned}
	\end{equation} 
where $Q_T=\mathbb{R}_+\times(0,T)$ for some $T>0$ to be determined and 
\begin{equation}\label{gu-def}
g_0^u(t):=E_b\big\{g(\cdot)-S[E_0u_0;0](0, \cdot)-S[0;f(Eu)](0, \cdot)\big\}(t)
\end{equation}
with the temporal transform $\widetilde{g}_0^u(\omega(k),T')$ defined according to \eqref{gj-def}. The extension operators $E_0$ and $E_b$ were defined below problems \eqref{cauchy} and \eqref{linear4} respectively; importantly, $E_0$ satisfies inequality \eqref{e0-def} and $E_b$ induces compact support on $g_0$, namely $\text{supp} g_0 \subset [0, T')$, $T'>T$. Moreover, the operator $E$ is a similar bounded fixed extension operator. In particular, for $s>\frac 12$ we take $E=E_0$ while for $0\leq s < \frac12$ we take $E$ from $H_x^s(\mathbb{R}_+)\cap H^{s,r}_x(\mathbb{R}_+)$ into $H_x^s(\mathbb{R})\cap H_x^{s,r}(\mathbb{R})$ for a certain $r>2$ to be specified later.

In view of \eqref{solformula}, we define the solutions of the nonlinear problem \eqref{nonlinear1} as the fixed points of the operator $\Phi$. 
Thus, our goal will be to prove the existence of a unique such fixed point in a suitable function space.
Throughout our analysis, we assume $u_0\in H_x^s(\mathbb{R}_+)$ and $g\in H_{t,\textnormal{loc}}^{\frac{s+1}{3}}(\mathbb{R}_+)$ with $s\in[0, 2]\setminus \{\frac{1}{2}\}$ and the compatibility conditions \eqref{comp-cond} in place as necessary.  We first treat the high regularity case $\frac12<s\leq 2$ in which we are able to employ the algebra property of $H_x^s(\mathbb{R}_+)$, and then move on to the low regularity case $0\leq s<\frac12$ in which we address the lack of the algebra property by refining our solution space motivated by the linear Strichartz estimates.
	\subsection{\ttfamily\bfseries High regularity solutions: Proof of Theorem \ref{HighRegThm}}
In the high regularity setting, we suppose that $\frac 12 < s \leq 2$ and $p>0$ with the additional assumptions \eqref{sprelation1} as necessary.
Our goal is to establish local well-posedness in the space $X_{T}:=C([0,T];H_x^s(\mathbb{R}_+))$ for some $T>0$ to be determined. We consider $X_T$ as a metric space with the metric
	$$d_{X_T}(u_1,u_2):=\|u_1-u_2\|_{X_T}, \quad u_1,u_2\in X_T.$$ Note that any closed ball in $X_T$ is a complete subspace. 
\\[2mm]
\noindent	
\textit{Showing that $\Phi$ is into.}
The conservation law \eqref{cuachyspaceest} in Theorem \ref{cauchylemma} and the boundedness \eqref{e0-def} of the spatial extension operator $E_0$ imply
\begin{equation}\label{HRCest}
			\|y|_{Q_T}\|_{X_T}\le\|S[E_0u_0;0]\|_{C([0,T];H_x^s(\mathbb{R}))} =\|E_0u_0\|_{H_x^s(\mathbb{R})}\lesssim \|u_0\|_{H_x^s(\mathbb{R}_+)},
\end{equation}
which takes care of the first term in \eqref{solformula}.
For the second term in \eqref{solformula}, let $u\in X_T$ and combine the nonhomogeneous estimate \eqref{nonhthm1} in Theorem \ref{Nonhomthm} with the algebra property in $H_x^s(\mathbb R)$ to yield
\begin{equation}\label{HRNonhomCEst}
\begin{aligned}
\|z^u|_{Q_T}\|_{X_T}
&\leq
\big\|S[0;f(E_0u)]\big\|_{C([0,T];H_x^s(\mathbb{R}))}
\lesssim \int_0^{T}\|f(E_0u(\cdot,t))\|_{H_x^s(\mathbb{R})}dt
\\ 
&\lesssim
\int_0^{T}\|E_0u(\cdot,t)\|_{H_x^s(\mathbb{R})}^{p+1}dt
\lesssim\int_0^{T}\|u(\cdot,t)\|_{H_x^s(\mathbb{R}_+)}^{p+1}dt \lesssim T\|u\|_{X_{T}}^{p+1}.
\end{aligned}
\end{equation}
Regarding the third term in \eqref{solformula}, using estimate \eqref{ribvp-se} in Theorem \ref{ibvpthem01} and the boundedness of the temporal extension operator $E_b$ (see Section 3 of \cite{hm2021} for more details), we get (say with $T'=2T$)
\begin{equation}\label{HRqest}
\begin{aligned}
\|q^u|_{(0,T)}\|_{X_T} \le \|q\|_{X_{T'}} &\lesssim \big(1+\sqrt{T'}e^{cT'}\big)\|g_0^u\|_{H_t^{\frac{s+1}{3}}(0,T')}
\\
&\lesssim
\big(1+\sqrt{T'}e^{cT'}\big)\|g_0^u\|_{H_t^{\frac{s+1}{3}}(0,T)}
\lesssim \big(1+\sqrt{T}e^{cT}\big)\|g_0^u\|_{H_t^{\frac{s+1}{3}}(0,T)}.
\end{aligned}
\end{equation}
By using the definition of $g_0$ in \eqref{linear4} and temporal trace estimates \eqref{cauchyextra3}, \eqref{cauchyextra3*} and \eqref{nonhthm2}, we obtain
\begin{equation}\label{HRbvpEst}
\begin{aligned}
\|g_0^u\|_{H_t^{\frac{s+1}{3}}(0,T)} 
&\lesssim \|g\|_{H_t^{\frac{s+1}{3}}(0,T)} + (1+T^{\frac12}) \|u_0\|_{H_x^s(\mathbb{R}_+)}
\\
&\quad +\max\{T^{\frac12}(1+T^{\frac12}), T^{\sigma}\} \, \|f(E_0u)\|_{L_t^2((0,T); H_x^{s}(\mathbb{R}))},
\end{aligned}
\end{equation} 
with $\sigma$ given by \eqref{stildedef}.
By using the definition of the solution space $X_T$ and the boundedness \eqref{e0-def}  of the spatial extension operator $E_0$, we have 
\begin{equation}\label{fL2Hsest}\|f(E_0u)\|_{L_t^2((0,T); H_x^{s}(\mathbb{R}))}\lesssim T^{\frac12}\|u\|_{X_{T}}^{p+1}.\end{equation}
Using the definition \eqref{solformula} of $\Phi$ and combining estimates \eqref{HRCest}-\eqref{fL2Hsest}, we deduce
		\begin{equation}\label{phiu-est}
			\|\Phi(u)\|_{X_{T}}\le c_0\left( c_1(T)\|u_0\|_{H_x^s(\mathbb{R}_+)}+c_2(T)\|g\|_{H_t^{\frac{s+1}{3}}(0,T)}+c_3(T)\|u\|_{X_T}^{p+1}\right),
		\end{equation}
		where the positive constants $c_1, c_2, c_3$  are given by  $c_1(T)=(1+\sqrt{T}e^{cT})(1+T^{\frac12})$, $c_2(T)=(1+\sqrt{T}e^{cT})$, $c_3(T)=T+(1+\sqrt{T}e^{cT})T^{\frac12}\max\{T^{\frac12}(1+T^{\frac12}), T^{\sigma}\}$ and $c_0$ is a non-negative constant independent of $T$ and only depending on fixed parameters such as $\alpha,\beta,\delta$ and $s$.  

In view of estimate \eqref{phiu-est}, we set $R(T):=2A(T)$ with
$$
A(T):=c_0\Big( c_1(T)\|u_0\|_{H_x^s(\mathbb{R}_+)}+c_2(T)\|g\|_{H_t^{\frac{s+1}{3}}(0,T)}\Big)
$$  
and choose $T$ small enough so that $A(T)+c_0c_3(T)R(T)^{p+1}\le R(T)$ or, equivalently,  $c_0c_3(T)R^{p}(T)\le \frac12$. We note that such a choice is possible because $c_3(T)\rightarrow 0^+$ and $R(T)$ remains bounded as $T\rightarrow 0^+$. Then, for that choice of $T$, the map $\Phi$ takes the closed ball $\overline{B_{R(T)}(0)}\subset X_{T}$ into itself.
It remains to show that $\Phi$ is a contraction on $\overline{B_{R(T)}(0)}$. 
\\[2mm]
\noindent	
\textit{Showing that $\Phi$ is a contraction.}
Let $u_1,u_2\in \overline{B_{R(T)}(0)}$. Then,
\begin{equation}\label{diff-1}
\begin{aligned}
\|\Phi(u_1)-\Phi(u_2)\|_{X_{T}}
&=
\left\|z^{u_1}|_{Q_T}-z^{u_2}|_{Q_T}\right\|_{X_T}+\left\|q^{u_1}|_{(0,T)}-q^{u_2}|_{(0,T)}\right\|_{X_T}
\\
&\lesssim 
\big\|S[0;f(E_0u_1)-f(E_0u_2)]\big\|_{C([0,T];H_x^s(\mathbb{R}))}
\\
&\quad
+\big(1+\sqrt{T}e^{cT}\big) \left\|g_0^{u_1}-g_0^{u_2}\right\|_{H_t^{\frac{s+1}{3}}(0,T)}.
\end{aligned}
\end{equation}
We then recall the following difference estimate (e.g. see \cite{BO2016}).
\begin{lem}\label{nonlinearitylemma} 
Let $s>\frac{1}{2}$, $p>0$ satisfy \eqref{sprelation1}  and $\varphi, \varphi_1,\varphi_2\in H^s(\mathbb{R})$. Then, 
%
%\begin{equation}\label{algprop}
%\left\| |\varphi|^p\varphi	\right\|_{H^s(\mathbb{R})}\lesssim \left\| \varphi \right\|_{H^s(\mathbb{R})}^{p+1}
%\end{equation} 
%and
$$
\left\| |\varphi_1|^p\varphi_1-|\varphi_2|^p\varphi_2\right\|_{H^s(\mathbb{R})}\lesssim \left(\|\varphi_1\|_{H^s(\mathbb{R})}^p+\|\varphi_2\|_{H^s(\mathbb{R})}^p\right) \|\varphi_1-\varphi_2\|_{H^s(\mathbb{R})}.
$$
\end{lem}
Employing Lemma \ref{nonlinearitylemma} and the arguments used earlier in \eqref{HRNonhomCEst}, we deduce
\begin{equation}\label{s-diff}
\big\|S[0;f(E_0u_1)-f(E_0u_2)]\big\|_{C([0,T];H_x^s(\mathbb{R}))}\lesssim 
T\big(\|u_1\|_{X_{T}}^p+\|u_2\|_{X_{T}}^p\big)\|u_1-u_2\|_{X_{T}}.
\end{equation}
Moreover, for the difference of boundary data we have, similarly to \eqref{HRbvpEst}, 
\begin{equation}\label{g-diff}
\begin{aligned}
\|g_0^{u_1}-g_0^{u_2}\|_{H_t^{\frac{s+1}{3}}(0,T)}
&\lesssim
\max\{T^{\frac12}(1+T^{\frac12}), T^{\sigma}\} \left\|f(E_0u_1)-f(E_0u_2)\right\|_{L_t^2((0,T); H_x^{s}(\mathbb{R}))}
\\
&\lesssim
 \max\{T^{\frac12}(1+T^{\frac12}), T^{\sigma}\}T^{\frac12}\big(\|u_1\|_{X_{T}}^p+\|u_2\|_{X_{T}}^p\big)\|u_1-u_2\|_{X_{T}},
\end{aligned}
\end{equation}
where $\sigma$ is given by \eqref{stildedef}.  Combining \eqref{s-diff} and \eqref{g-diff} with \eqref{diff-1}, we obtain 
\begin{equation}\label{Phidifest}
\|\Phi(u_1)-\Phi(u_2)\|_{X_{T}}
\lesssim 
c_3(T)\big(\|u_1\|_{X_{T}}^p+\|u_2\|_{X_{T}}^p\big)\|u_1-u_2\|_{X_{T}}
\lesssim c_3(T)R^p(T)\|u_1-u_2\|_{X_{T}}.
\end{equation}
		Note that $c_3(T)\rightarrow 0^+$ and $R(T)$ remains bounded as $T\rightarrow 0^+$. Therefore, for sufficiently small $T>0$ the map $\Phi$ is a contraction on $\overline{B_{R(T)}(0)}$, and hence $\Phi$ has a unique fixed point in $\overline{B_{R(T)}(0)}$ which, as noted earlier, amounts to local existence of a unique solution to the HNLS initial-boundary value problem \eqref{nonlinear1} on $\overline{B_{R(T)}(0)}$.  
\\[2mm]
\noindent		
\textit{Extending uniqueness to $X_T$.} To prove uniqueness over the entire space $X_T$ and not just the closed ball $\overline{B_{R(T)}(0)}$, we suppose that $u_1,u_2\in X_T$ are two solutions associated with the  same pair of initial and boundary data $(u_0, g)$. At first, we consider the case of $u_1,u_2$ being sufficiently smooth and, along with their derivatives, decaying sufficiently fast as $x\rightarrow \infty$.  This allows us to proceed via energy estimates. In particular, we note that the difference $w:=u_1-u_2$  solves the following problem:
\begin{equation}\label{difnonlinear1}
\begin{aligned}
&iw_t+i\beta w_{xxx}+\alpha w_{xx}+i\delta w_x = f(u_1)-f(u_2),  \quad (x,t)\in\mathbb{R}_{+} \times (0,T),
\\
&w(x,0) = 0, \quad x\in \mathbb{R}_{+},
\\
&w(0,t) = 0, \quad t\in (0,T).
\end{aligned}
\end{equation}
Multiplying the main equation by $\overline{w}$, integrating in $x$, taking imaginary parts, and using Lemma \ref{nonlinearitylemma}  and the embedding $H^s_x(\mathbb{R}_+)\hookrightarrow L^\infty_x(\mathbb{R}_+)$, which is valid for $s>\frac 12$, we find
\begin{align*}
\frac{1}{2}\frac{d}{dt}\|w(t)\|_{L^2_x(\mathbb{R}_+)}^2
&=
-\frac{\beta}{2}|w_x(0,t)|^2+\text{Im}\int_0^\infty \big[f(u_1(x,t))-f(u_2(x,t))\big]\bar{w}(x,t)dx
\\
&\lesssim 
\int_0^\infty \big(|u_1(x,t)|^p+|u_2(x,t)|^p\big) |{w}(x,t)|^2dx
%\\
%&\lesssim \big( \|u_1(t)\|_{L^\infty_x(\mathbb{R}_+)}^p+\|u_2(t)\|_{L^\infty_x(\mathbb{R}_+)}^p\big) \|w(t)\|_{L^2_x(\mathbb{R}_+)}^2
\\
&\lesssim \big( \|u_1(t)\|_{H_x^s(\mathbb{R}_+)}^p+\|u_2(t)\|_{H_x^s(\mathbb{R}_+)}^p\big) \|w(t)\|_{L^2_x(\mathbb{R}_+)}^2
%\\
%&\lesssim
%\big(\|u_1\|_{L^\infty_t((0, T); L^\infty_x(\mathbb{R}_+))}^p+\|u_2(t)\|_{L^\infty_t((0, T); L^\infty_x(\mathbb{R}_+))}^p\big) 
%\|w(t)\|_{L^2_x(\mathbb{R}_+)}^2
%\\
%&\lesssim 
%\big(\|u_1\|_{L^\infty_t((0, T); H^s_x(\mathbb{R}_+))}^p+\|u_2(t)\|_{L^\infty_t((0, T); H^s_x(\mathbb{R}_+))}^p\big) 
%\|w(t)\|_{L^2_x(\mathbb{R}_+)}^2
\\
&\lesssim \big(\|u_1\|_{X_T}^p+\|u_2(t)\|_{X_T}^p\big) 
\|w(t)\|_{L^2_x(\mathbb{R}_+)}^2.
\end{align*} 
Setting $y(t):=\|w(t)\|_{L^2_x(\mathbb{R}_+)}$, the above energy estimate is satisfied provided that $y'(t)-cy(t)\le 0$, $t\in (0,T)$ for some non-negative constant $c$.
Solving this differential inequality alongside the condition $y(0)=\|w(0)\|_{L^2_x(\mathbb{R}_+)}=0$ (note that $w(x,0)\equiv 0$), we obtain $y \equiv 0$ i.e. $w = u_1-u_2\equiv 0$. 
The case of rough $u_1, u_2$ can be treated via mollification along the lines of the arguments used in the proof of Proposition 1.4 in \cite{h2005}.
%		
		%Then,
		%\begin{equation*}
		%	u_1-u_2&=&(z^{u_1}-z^{u_2})|_{Q_T}+(q^{u_1}-q^{u_2})|_{(0,T)}\\&=&\left.S[0;f(E_0u_1)-f(E_0u_2)]\right|_{Q_T}\left.-\frac{i}{2\pi}\int_{\Gamma}e^{ikx-\omega(k)t}\omega'(k) \, \widetilde{g}_0^{u_1,u_2}(\omega(k),T')dk\right|_{(0,T)},
		%\end{equation*} where
		%$g_0^{u_1,u_2}:=E_b(-S[0;f(E_0u_1)-f(E_0u_2)]_{x=0})$. We first note
		%\begin{equation*}
		%	\|(z^{u_1}(t)-z^{u_2}(t))|_{Q_T}\|_{H_x^s(\mathbb{R}_+)}&\le&\big\|S[0;f(E_0u_1)-f(E_0u_2)](t)\big\|_{H_x^s(\mathbb{R})}\\
		%	&\le& \int_{0}^{t} \|S[f(E_0u_1(\cdot,t')-f(E_0u_2(\cdot,t');0](\cdot,t-t')\|_{H_x^s(\mathbb{R})}dt'\\
		%	 &=& \int_{0}^{t} \|f(E_0u_1(\cdot,t')-f(E_0u_2(\cdot,t')\|_{H_x^s(\mathbb{R})}dt'\\
		%	 &\lesssim&\int_{0}^{t}\big(\|u_1(t')\|_{H_x^s(\mathbb{R}_+)}^p+\|u_2(t')\|_{H_x^s(\mathbb{R}_+)}^p\big)\|u_1(t')-u_2(t')\|_{H_x^s(\mathbb{R}_+)}dt'.
		%\end{equation*}
		%Secondly, we have
		%\begin{equation*}
		%	\|q^{u_1}(t)|_{(0,T)}-q^{u_2}(t)|_{(0,T)}\|_{H^s_x(\mathbb{R}_+)}\le.
		%\end{equation*}
		%It follows that
		%\begin{equation*}
		%	\|u_1(t)-u_2(t)\|_{H^s_x(\mathbb{R}_+)}&\lesssim& \left(1+\sqrt{T}e^{cT}\right)\|g_0^{u_1,u_2 }\|_{H_t^{\frac{s+1}{3}}(0,T)}\\
		%	&&+\big(\|u_1\|_{X_T}^p+\|u_2\|_{X_T}^p\big)\int_{0}^{t}\|u_1(t')-u_2(t')\|_{H_x^s(\mathbb{R}_+)}dt'
\\[2mm]
\noindent		
\textit{Continuous dependence on the data.} 
For $(u_0,g)\in H_x^s(\mathbb{R}_+)\times H_{t,\textnormal{loc}}^{\frac{s+1}{3}}(\mathbb{R}_+)$, let 
$$
T_{\text{max}}:=\sup\,\{\,T>0\,|\,\text{there is a solution associated to the data } (u_0,g) \text{ on } [0,T]\}.
$$ Then, either $T_{\text{max}}=\infty$ or else $T_{\text{max}}<\infty$ and there is no solution $u \in X_{T_{\text{max}}}$ since otherwise the lifespan of $u$ could be extended beyond $T_{\text{max}}$ by starting with initial datum equal to $u(T_{\text{max}})$. 
Therefore, we may let $u\in C([0,T_{\text{max}});H^s_x(\mathbb{R}_+))$ be the maximal solution associated to the data $(u_0,g)$; then, for $T<T_{\text{max}}$, in particular,   $u|_{[0,T]}$ is the unique solution in $X_T$ established above.
		
Let $T<T_{\text{max}}$ be small enough that $\Phi$ is a contraction on $\overline{B_{R(T)}(0)}$ for any solution associated with data $(v_0,h)\in H_x^s(\mathbb{R}_+)\times H_{t,\textnormal{loc}}^{\frac{s+1}{3}}(\mathbb{R}_+)$ and satisfying
$$
\|v_0\|_{H^s_x(\mathbb{R}_+)}+ \|h\|_{H_t^{\frac{s+1}{3}}(0,T)}
\leq 
2\, \Big(\|u_0\|_{H^s_x(\mathbb{R}_+)}+ \|g\|_{H_t^{\frac{s+1}{3}}(0,T)}\Big).
$$
If follows that if $\delta>0$ is small enough, for $(v_0,h)$ satisfying
$$
\|v_0-u_0\|_{H^s_x(\mathbb{R}_+)}+ \|g-h\|_{H_t^{\frac{s+1}{3}}(0,T)}<\delta
$$
the associated solution $v$ belongs  to $\overline{B_{R(T)}(0)}$.  Therefore,  $u$ and $v$ are both fixed points of $\Phi$ on $\overline{B_{R(T)}(0)}$ associated with the pairs of data $(u_0,g)$ and $(v_0,h)$, respectively. Then, the corresponding nonlinear estimates from the contraction argument imply
$$
\|u-v\|_{X_T}=\|\Phi u-\Phi v\|_{X_T}
\lesssim 
c(T)\Big( \|u_0-v_0\|_{H_x^s(\mathbb{R}_+)}+\|g-h\|_{H_t^{\frac{s+1}{3}}(0,T)}\Big)
\lesssim \delta c(T).
$$
which amounts to continuity of the data-to-solution map. The proof of Theorem \ref{HighRegThm} for  well-posedness in the high regularity setting is complete.

\subsection{\ttfamily\bfseries Low regularity solutions: Proof of Theorem \ref{LowRegThm}}
In this setting, we work under the assumptions \eqref{assmponsandp}.
The lack of the algebra property brings in the need for the various Strichartz estimates established in Section \ref{linear-s} and hence motivates the solution space 
$$
Y_T:=C([0,T];H^s_x(\mathbb{R}_+))\cap L^\mu_t((0, T); H_x^{s,r}(\mathbb{R}_+)).
$$ 
It is convenient to also consider the associated space on the whole spatial line, namely
$$\widetilde Y_T:=C([0,T];H^s_x(\mathbb{R}))\cap L^\mu_t((0, T); H_x^{s,r}(\mathbb{R})).$$ 
The following lemma will serve as the low regularity analogue of the algebra property and Lemma \ref{nonlinearitylemma}.
\begin{lem}\label{nonlinest}
Let $(s,p)$, $(\mu, r)$ satisfy \eqref{assmponsandp}  and suppose $\varphi,\varphi_1,\varphi_2\in L_{t}^{\mu}((0, T); H_{x}^{s,r}(\mathbb{R}))$.
Then, 
\begin{align}\label{nonlinest1}
\big\| |\varphi|^p\varphi \big\|_{L_t^1((0, T); H_x^{s}(\mathbb{R}))}&\lesssim T^{\frac{\mu-p-1}{\mu}}\left\| \varphi \right\|_{L_{t}^{\mu}((0, T); H_{x}^{s,r}(\mathbb{R}))}^{p+1},
\\
\label{nonlinest2}
\big\| |\varphi_1|^p\varphi_1-|\varphi_2|^p\varphi_2 \big\|_{L_t^1((0, T); H_x^{s}(\mathbb{R}))}
&\lesssim T^{\frac{\mu-p-1}{\mu}}\Big(\|\varphi_1\|_{L_{t}^{\mu}((0, T); H_{x}^{s,r}(\mathbb{R}))}^p
+
\|\varphi_2\|_{L_{t}^{\mu}((0, T); H_{x}^{s,r}(\mathbb{R}))}^p  \Big)
\nonumber\\
&\quad
\cdot \|\varphi_1-\varphi_2\|_{L_{t}^{\mu}((0, T); H_{x}^{s,r}(\mathbb{R}))}.
\end{align}
\end{lem}

Lemma \ref{nonlinest} is proved after the end of the current proof and corresponds to the one-dimensional analogue of inequality (6.17) for the two-dimensional  NLS equation proved in \cite{hm2020}. Note, importantly, that the admissibility conditions \eqref{assmponsandp} are different than those in \cite{hm2020} due to the third-order dispersion of the HNLS equation. Thus, the proof of Lemma \ref{nonlinest} does not follow from \cite{hm2020}.
Now, we are ready to prove Theorem \ref{LowRegThm} for low regularity solutions.
\\[2mm]
\noindent
\textit{Existence.}
First, we consider the \emph{subcritical case} $p\neq \frac{6}{1-2s}$ so that $\frac{\mu-p-1}{\mu}>0$.  We work again with the solution operator \eqref{solformula}, which was obtained via linear reunification.
Theorems \ref{cauchylemma} and \ref{homStr} imply
\begin{equation}\label{nonlowregest1}
\|y|_{Q_T}\|_{Y_T}\le \|y\|_{\widetilde Y_T}\lesssim \|E_0u_0\|_{H^s_x(\mathbb{R})}\lesssim \|u_0\|_{H^s_x(\mathbb{R}_+)},
\end{equation}
while Theorems \ref{Nonhomthm} and \ref{NonHomStrThm} along with inequality \eqref{nonlinest1} and the same argument that was used in \eqref{HRNonhomCEst} yield
\begin{equation}\label{nonlowregest2}
\|z^u|_{Q_T}\|_{Y_T}\le \|z^u\|_{\widetilde Y_T}\lesssim \big(T+T^{\frac{\mu-p-1}{\mu}}\big)\|u\|_{Y_T}^{p+1}.
\end{equation}
Furthermore, Theorem \ref{BdrStrThm} (with say $T'=2T$) and the same arguments that led to \eqref{HRqest} imply
\begin{equation}\label{nonlowregest3}
\|q^u|_{(0,T)}\|_{Y_T} \lesssim \big(1+\sqrt{T}e^{cT}+T^{\frac{1}{\mu}+\frac{1}{2}}\big)\|g_0^u\|_{H_t^{\frac{s+1}{3}}(0,T)}.
\end{equation}
Combining \eqref{nonlowregest1}-\eqref{nonlowregest3} and proceeding along the lines of the arguments that resulted in \eqref{HRbvpEst} and \eqref{fL2Hsest}, we obtain
$$
\|\Phi(u)\|_{Y_T}\le c_0\Big( c_1(T)\|u_0\|_{H_x^s(\mathbb{R}_+)}+c_2(T)\|g\|_{H_t^{\frac{s+1}{3}}(0,T)}+c_3(T)\|u\|_{Y_T}^{p+1}\Big),
$$
where  the positive constants  $c_1, c_2, c_3$  are given by $c_1(T)=(1+\sqrt{T}e^{cT}+T^{\frac{1}{\mu}+\frac{1}{2}})(1+T^{\frac12})$, $c_2(T)=(1+\sqrt{T}e^{cT}+T^{\frac{1}{\mu}+\frac{1}{2}})$, $c_3(T)=(T+T^{\frac{\mu-p-1}{\mu}})+(1+\sqrt{T}e^{cT}+T^{\frac{1}{\mu}+\frac{1}{2}})T^{\frac12}\max\{T^{\frac12}(1+T^{\frac12}), T^{\sigma}\}$  and $c_0$ is a non-negative constant independent of $T$ and only depending on fixed parameters such as $\alpha,\beta,\delta$ and $s$. 

For the contraction, given $u_1,u_2\in Y_T$ we employ inequality \ref{nonlinest2} together with the same arguments that led to \eqref{Phidifest} to infer		
\begin{equation}\label{Phidifestlow}
\|\Phi(u_1)-\Phi(u_2)\|_{Y_T}\lesssim c_3(T)\big(\|u_1\|_{Y_T}^p+\|u_2\|_{Y_T}^p\big)\|u_1-u_2\|_{Y_T}.
\end{equation}
This estimate implies the existence of a fixed point in $Y_T$ for sufficiently small $T>0$ via the same arguments that were used in the proof of Theorem \ref{HighRegThm}.
		
Next, we consider the \textit{critical case} $p=\frac{6}{1-2s}$. The difference here compared to the subcritical case is that the limit $c_3(T)\rightarrow 0^+$ as $T\rightarrow 0^+$ is no longer true; however, $\Phi$ is still a contraction provided that the data (and, correspondingly, the radius of the closed ball that depends on the size of the data) are chosen sufficiently small. 
\\[2mm]
\noindent
\textit{Uniqueness.}
We adapt the method used for the Cauchy problem in the proof of Proposition 4.2 of \cite{Caz90} to the framework of initial-boundary value problems. 

First, consider the subcritical case $p\neq \frac{6}{1-2s}$. Let $u_1=\Phi(u_1), u_2=\Phi(u_2)\in Y_T$ be two solutions associated with the same pair of initial and boundary data.  Suppose to the contrary that there is $t\in [0,T]$ for which $u_1(t)\neq u_2(t)$, and let
$$t_{\text{inf}}:=\inf\left\{t\in [0,T]\,|\, u_1(t)\neq u_2(t)\right\}.$$
Taking $t_n<t_{\text{inf}}$ such that $t_n\rightarrow t_{\text{inf}}^-$ as $n\rightarrow \infty$, we see that  $u_1(t_n)=u_2(t_n)$ by definition of $t_{\text{inf}}$. Thus,  in view of the fact that $u_1,u_2$ are both continuous from $[0,T]$ into $H_x^s(\mathbb{R}_+)$,  taking the limit $n\rightarrow \infty$ we deduce that 
$u_1(t_{\text{inf}})=u_2(t_{\text{inf}})=:\varphi \in H_x^s(\mathbb{R}_+)$
makes sense. Set $U_1(t)=u_1(t+t_{\text{inf}})$ and $U_2(t)=u_2(t+t_{\text{inf}})$. Then, $U_1$ and $U_2$ are both solutions on the temporal interval $[0,T-t_{\text{inf}}]$ that satisfy the same initial and boundary conditions, namely
$$U_1(0)=U_2(0)=\varphi,\quad U_1|_{x=0}=U_2|_{x=0}=g(\cdot+t_{\text{inf}}) =: g_{\text{inf}}.
$$
Since $U_1$ and $U_2$ are continuous in $t$, by the definition of $t_{\text{inf}}$ there is a $\delta>0$ such that $U_1\neq U_2$ for $t\in (0,\delta)$.  Let $t=t_{\text{inf}}+\eps$ with $\eps\in (0,\delta)$ fixed and to be specified below. We have 
\begin{equation}\label{u1u2-eps}
\begin{aligned}
\|U_1-U_2\|_{L_t^\mu((0, \eps);H_x^{s,r}(\mathbb{R}_+))}
&\lesssim c_{\text{inf}}(\eps)\big(\|U_1\|_{L_t^\mu((0, \eps);H_x^{s,r}(\mathbb{R}_+))}^p+\|U_2\|_{L_t^\mu((0, \eps);H_x^{s,r}(\mathbb{R}_+))}^p\big)
\\
&\quad
\cdot \|U_1-U_2\|_{L_t^\mu((0, \eps);H_x^{s,r}(\mathbb{R}_+))}, 
\end{aligned}
\end{equation} 
where
$c_{\text{inf}}(\eps):=\eps^{\frac{\mu-p-1}{\mu}}+\eps^{\frac{1}{\mu}+\frac{1}{2}}\eps^{\frac12}\max\{\eps^{\frac12}(1+\eps^{\frac12}), \eps^{\sigma}\}.$ Let $\eps \in (0, \delta)$ be small enough so that 
\begin{equation}\label{UniqEst}
c_{\text{inf}}(\eps)\big(\|U_1\|_{L_t^\mu((0, \eps);H_x^{s,r}(\mathbb{R}_+))}^p+\|U_2\|_{L_t^\mu((0, \eps);H_x^{s,r}(\mathbb{R}_+))}^p\big)<1,
\end{equation} 
which is possible because $c_{\text{inf}}(\eps)\rightarrow 0^+$ as $\eps\rightarrow 0^+$. Then, \eqref{u1u2-eps} implies that $U_1=U_2$ on $(0,\eps) \subset (0, \delta)$, leading to a contradiction. Hence, uniqueness follows.
		
		In the critical case $p=\frac{6}{1-2s}$, although the limit $c_{\text{inf}}(\eps)\rightarrow 0^+$ as $\eps\rightarrow 0^+$ is no longer true, the uniqueness argument remains valid as \eqref{UniqEst} still holds due to the fact that, due to the dominated convergence theorem, the norms $\|U_1\|_{L_t^\mu((0, \eps);H_x^{s,r}(\mathbb{R}_+))}$ and $\|U_2\|_{L_t^\mu((0, \eps);H_x^{s,r}(\mathbb{R}_+))}$ can be made arbitrarily small  by taking $\eps$ small enough. 
		
Finally, the continuous dependence of the unique solution in $Y_T$ on the initial and boundary data can be proved as in the high regularity setting, thereby completing the proof of Theorem \ref{LowRegThm}.

\begin{proof}[Proof of Lemma \ref{nonlinest}]
By H\"{o}lder's inequality,
$$
\big\| |\varphi|^p\varphi \big\|_{L_t^1((0,T); H^{s}(\mathbb{R}))} \leq T^{\frac{\mu-p-1}{\mu}} \bigg( \int_{0}^{T}  \big\| |\varphi(t)|^p\varphi(t) \big\|_{H_x^{s}(\mathbb{R})}^{\frac{\mu}{p+1}} dt \bigg)^{\frac{p+1}{\mu}}.
$$
On the other hand,
$\left\| \varphi \right\|_{L_{t}^{\mu}(0,T_0;H_{x}^{s,r}(\mathbb{R}))}^{p+1}=\bigg( \int_{0}^{T}  \|\varphi(t)\|_{H_{x}^{s,r}(\mathbb{R})}^{\mu} dt \bigg)^{\frac{p+1}{\mu}}$. 
Hence, in order to establish \eqref{nonlinest1}, it suffices to prove that 
		\begin{equation}\label{nonlinearityext1}
			\|D^\theta\big( |\varphi(t)|^{p}\varphi(t) \big)\|_{L_x^2(\mathbb{R})} \lesssim \|\varphi(t)\|_{H_{x}^{\theta,r}(\mathbb{R})}^{p+1}, \quad t\in(0,T),
		\end{equation}
for $\theta=0$ and $\theta=s$.  
To this end, we set $F(z):=|z|^pz$, $z\in\mathbb{C}$. If $s\neq 0$, by using the chain rule for fractional derivatives (e.g. see Proposition 3.1 in \cite{Christ91}), we have
		\begin{equation}\label{nonlinearityext2}
			\|D^s F(\varphi(t))\|_{L_x^2(\mathbb{R})}\lesssim \|F'(\varphi(t))\|_{L_x^{\frac{\mu}{3}}(\mathbb{R})}\|D^s \varphi(t)\|_{L_x^r(\mathbb{R})}
		\end{equation} with $\frac12=\frac{3}{\mu}+\frac{1}{r}$. Noting that $|F'(\varphi(t))|\le (p+1)|\varphi(t)|^p$, we further find
		\begin{equation}\label{nonlinearityext3}
			\|F'(\varphi(t))\|_{L_x^{\frac{\mu}{3}}(\mathbb{R})}\lesssim\|\varphi(t)\|^{p}_{L_x^{\frac{p\mu}{3}}(\mathbb{R})},
		\end{equation}
while for $\frac{3}{p\mu}=\frac{1}{r}-s$ we also have the embedding
		\begin{equation}\label{nonlinearityext4}
			\|\varphi(t)\|_{L_x^{\frac{p\mu}{3}}(\mathbb{R})}\lesssim \|\varphi(t)\|_{H^{s,r}_{x}(\mathbb{R})}.
		\end{equation}
		Combining \eqref{nonlinearityext3} and \eqref{nonlinearityext4} with \eqref{nonlinearityext2}, we obtain (\ref{nonlinearityext1}) for $\theta=s\neq 0$. Notice that $r=2(p+1)$ for $s=0$. Therefore, $$\| |\varphi(t)|^p\varphi(t)\|_{L^2_x(\mathbb{R})}=\|\varphi(t)\|_{L_x^{2(p+1)}(\mathbb{R})}^{p+1}=\|\varphi(t)\|_{L_x^{r}(\mathbb{R})}^{p+1},$$ which corresponds to  \eqref{nonlinearityext1} for $\theta= 0$.

		Regarding inequality \eqref{nonlinest2} for the differences, we first consider the case $s=0$ which implies $r=2(p+1)$. Using the standard pointwise difference estimate for the power-type nonlinearity and then applying Hölder's inequality in $x$, we get
		\begin{align*}
			\big\| |\varphi_1|^p\varphi_1-|\varphi_2|^p\varphi_2 \big\|_{L_t^1((0, T); L_x^{2}(\mathbb{R}))}
			&\lesssim \int_0^T\left(\int_{-\infty}^\infty(|\varphi_1(x,t)|^p+|\varphi_2(x,t)|^p)^2|\varphi_1(x,t)-\varphi_2(x,t)|^2dx\right)^{\frac12}dt
\\
			&\lesssim\int_0^T\left(\|\varphi_1(t)\|_{L_x^{r}(\mathbb{R})}^p+\|\varphi_2(t)\|_{L_x^{r}(\mathbb{R})}^p\right)\|\varphi_1(t)-\varphi_2(t)\|_{L_x^{r}(\mathbb{R})}dt
		\end{align*} 
and the desired estimate \eqref{nonlinest2} for $s=0$ follows via Hölder's inequality in $t$.
		
Next, let us consider the case $s\neq 0$, in which $r=\frac{2(p+1)}{1+2sp}$.  First, observe that for $z_1,z_2\in \mathbb{C}$ and $\xi(\rho)=(1-\rho)z_2+\rho z_1$, $\rho\in [0,1]$, we have $\xi(0)=z_1$, $\xi(1)=z_2$, $\xi'(\rho)=z_1-z_2$,. Moreover,
		\begin{align*}
			|z_2|^pz_2\, -\, |z_1|^pz_1&=\int_{0}^1\frac{d}{d\rho}\left(|\xi(\rho)|^p\xi(\rho)\right)d\rho\\
			&=\frac{(p+2)}{2}(z_1-z_2)\int_0^1|\xi(\rho)|^pd\rho+\frac{p}{2}(\bar{z}_1-\bar{z}_2)\int_0^1|\xi(\rho)|^{p-2}\xi^2(\rho)d\rho.
		\end{align*}
		%Therefore,
		%$$\left\left\| z_2|^pz_2\, -\,|z_1|^pz_1\right|\lesssim \,(|z_1|^p+|z_2|^p)|z_1-z_2|.$$
Combining this writing with the fractional product rule (see Proposition 3.3 in \cite{Christ91}), we find
		\begin{align*}
\|D^sF(\varphi_1(t))-D^sF(\varphi_2(t))\|_{L_x^2(\mathbb{R})}
			&\lesssim\|D^s(\varphi_1(t)-\varphi_2(t))\|_{L_x^{r}(\mathbb{R})}\sup_{\rho\in[0,1]}\left\| |w(t)|^p\right\|_{L_x^{\frac{2r}{r-2}}(\mathbb{R})}\\
\quad
&+\,\|\varphi_1(t)-\varphi_2(t)\|_{L^{\frac{p\mu}{3}}_x(\mathbb{R})}\Big(\sup_{\rho\in[0,1]}\left\{\left\|D^s\left(G(w(t))\right)\right\|_{L_x^{r_1}(\mathbb{R})}\right\}\Big)
		\end{align*} where $\frac{1}{r_1}=\frac{1}{2}-\frac{3}{p\mu}$,
		$w(t)=(1-\rho)\varphi_2(t)+\rho\varphi_1(t)$ and $G(z)=F'(z)=\frac{p+2}{2}|z|^p+\frac{p}{2}|z|^{p-2}z^2$, $z\in \mathbb{C}$. 
		
Observing that  $|G'(w(t))|\le p(p+1)|w(t)|^{p-1}$ for $p>1$, we use the fractional chain rule (Proposition~3.1 in \cite{Christ91}) to infer	that, for $p>1$,
	\begin{align*}
			\left\|D^s\left(G(w(t))\right)\right\|_{L_x^{r_1}(\mathbb{R})}&\lesssim \| |w(t)|^{p-1}\|_{L_x^{r_2}(\mathbb{R})}\|D^sw(t)\|_{L^r_x(\mathbb{R})}\\
			&\lesssim \|w(t)\|_{L_x^{\frac{\mu p}{3}}(\mathbb{R})}^{p-1}\|D^sw(t)\|_{L^r_x(\mathbb{R})}\lesssim\|w(t)\|_{H_x^{s,r}(\mathbb{R})}^p,
		\end{align*} 
where $\frac{1}{r_2}=\frac{1}{r_1}-\frac{1}{r}$. In the above, the second inequality is due to the fact that, in view of \eqref{assmponsandp}, $r_2 = \frac{2(p+1)}{(p-1)(1-2s)}=\frac{\mu p}{3(p-1)}$, and the third inequality follows from the embedding \eqref{nonlinearityext4}. 
Furthermore, notice that $\frac{2r}{r-2}=\frac{\mu}{3}$ and so, using once again the embedding \eqref{nonlinearityext4}, 
		$$\left\| |w(t)|^p\right\|_{L_x^{\frac{2r}{r-2}}(\mathbb{R})}=\left\|w(t)\right\|_{L_x^{\frac{\mu p}{3}}(\mathbb{R})}^p\lesssim \|w(t)\|_{H_x^{s,r}(\mathbb{R})}^p.$$

Combining the last three estimates, we deduce
$$
			\|D^sF(\varphi_1(t))-D^sF(\varphi_2(t))\|_{L_x^2(\mathbb{R})}\lesssim \left(\|\varphi_1(t)\|_{H_x^{s,r}(\mathbb{R})}^p+\|\varphi_2(t)\|_{H_x^{s,r}(\mathbb{R})}^p\right)\|\varphi_1(t)-\varphi_2(t)\|_{H_x^{s,r}(\mathbb{R})}.
$$
Then, integrating over $(0,T)$, applying Hölder's inequality in $t$, and combining the resulting estimate with the case of $s=0$, we obtain \eqref{nonlinest2} for $s\neq 0$ and $p>1$.
		
Finally, for $p=1$ we note that $\frac12=\frac{1+2s}{4}+\frac{1-2s}{4}=\frac{1}{r}+\frac{1-2s}{4}=\frac{1}{r}+\frac{3}{\mu}$. Therefore,
		\begin{align*}
			\|D^sF(\varphi_1(t))-D^sF(\varphi_2(t))\|_{L_x^2(\mathbb{R})} &\lesssim \|D^s(\varphi_1(t)-\varphi_2(t))\|_{L_x^{r}(\mathbb{R})}\sup_{\rho\in[0,1]}\left\|w(t)\right\|_{L_x^{\frac{\mu}{3}}(\mathbb{R})}\\
			&\quad + \|\varphi_1(t)-\varphi_2(t)\|_{L^{\frac{\mu}{3}}_x(\mathbb{R})}\Big(\sup_{\rho\in[0,1]}\left\{\left\|D^s\left(G(w(t))\right)\right\|_{L_x^{r}(\mathbb{R})}\right\}\Big)\\
			&\lesssim \left(\|\varphi_1(t)\|_{H_x^{s,r}(\mathbb{R})}^p+\|\varphi_2(t)\|_{H_x^{s,r}(\mathbb{R})}^p\right)\|\varphi_1(t)-\varphi_2(t)\|_{H_x^{s,r}(\mathbb{R})}
		\end{align*}
with the last step thanks to the embedding \eqref{nonlinearityext4}. 
\end{proof}

	\bibliographystyle{amsplain}

\providecommand{\bysame}{\leavevmode\hbox to3em{\hrulefill}\thinspace}
\providecommand{\MR}{\relax\ifhmode\unskip\space\fi MR }
\providecommand{\MRhref}[2]{%
  \href{http://www.ams.org/mathscinet-getitem?mr=#1}{#2}
}
\providecommand{\href}[2]{#2}
\begin{thebibliography}{10}

\bibitem{agrawal}
Govind~P. Agrawaal, \emph{Nonlinear fiber optics}, 5th ed., Academic Press,
  2013.

\bibitem{Aud19}
Corentin Audiard, \emph{Global {S}trichartz estimates for the {S}chr\"{o}dinger
  equation with non zero boundary conditions and applications}, Ann. Inst.
  Fourier (Grenoble) \textbf{69} (2019), no.~1, 31--80. \MR{3973445}

\bibitem{BOF2020}
A.~Batal, A.~S. Fokas, and T.~\"{O}zsar\i, \emph{Fokas method for linear
  boundary value problems involving mixed spatial derivatives}, Proc. A.
  \textbf{476} (2020), no.~2239, 20200076, 15. \MR{4133772}

\bibitem{BO2016}
Ahmet Batal and T\"{u}rker \"{O}zsar\i, \emph{Nonlinear {S}chr\"{o}dinger
  equations on the half-line with nonlinear boundary conditions}, Electron. J.
  Differential Equations (2016), Paper No. 222, 20. \MR{3547411}

\bibitem{batal2020}
Ahmet Batal, T\"{u}rker \"{O}zsar\i, and Kemal~Cem Yilmaz, \emph{Stabilization
  of higher order {S}chr\"{o}dinger equations on a finite interval: {P}art
  {I}}, Evol. Equ. Control Theory \textbf{10} (2021), no.~4, 861--919.
  \MR{4338836}

\bibitem{Bis07}
Eleni Bisognin, Vanilde Bisognin, and Octavio~Paulo Vera~Villagr\'{a}n,
  \emph{Stabilization of solutions to higher-order nonlinear {S}chr\"{o}dinger
  equation with localized damping}, Electron. J. Differential Equations (2007),
  No. 06, 18. \MR{2278420}

\bibitem{bsz2002}
Jerry~L. Bona, S.~M. Sun, and Bing-Yu Zhang, \emph{A non-homogeneous
  boundary-value problem for the {K}orteweg-de {V}ries equation in a quarter
  plane}, Trans. Amer. Math. Soc. \textbf{354} (2002), no.~2, 427--490.
  \MR{1862556}

\bibitem{Bona18}
Jerry~L. Bona, Shu-Ming Sun, and Bing-Yu Zhang, \emph{Nonhomogeneous
  boundary-value problems for one-dimensional nonlinear {S}chr\"{o}dinger
  equations}, J. Math. Pures Appl. (9) \textbf{109} (2018), 1--66. \MR{3734975}

\bibitem{c1961}
A.-P. Calder\'{o}n, \emph{Lebesgue spaces of differentiable functions and
  distributions}, Proc. {S}ympos. {P}ure {M}ath., {V}ol. {IV}, American
  Mathematical Society, Providence, R.I., 1961, pp.~33--49. \MR{0143037}

\bibitem{Cap20}
Roberto de~A. Capistrano-Filho, M\'{a}rcio Cavalcante, and Fernando~A. Gallego,
  \emph{Lower regularity solutions of the biharmonic {S}chr\"{o}dinger equation
  in a quarter plane}, Pacific J. Math. \textbf{309} (2020), no.~1, 35--70.
  \MR{4202004}

\bibitem{Car2003}
X.~Carvajal and F.~Linares, \emph{A higher-order nonlinear {S}chr\"{o}dinger
  equation with variable coefficients}, Differential Integral Equations
  \textbf{16} (2003), no.~9, 1111--1130. \MR{1989544}

\bibitem{Carvajal04}
Xavier Carvajal, \emph{Local well-posedness for a higher order nonlinear
  {S}chr\"{o}dinger equation in {S}obolev spaces of negative indices},
  Electron. J. Differential Equations (2004), No. 13, 10. \MR{2036197}

\bibitem{Carvajal06}
\bysame, \emph{Sharp global well-posedness for a higher order {S}chr\"{o}dinger
  equation}, J. Fourier Anal. Appl. \textbf{12} (2006), no.~1, 53--70.
  \MR{2215677}

\bibitem{Caz90}
Thierry Cazenave and Fred~B. Weissler, \emph{The {C}auchy problem for the
  critical nonlinear {S}chr\"{o}dinger equation in {$H^s$}}, Nonlinear Anal.
  \textbf{14} (1990), no.~10, 807--836. \MR{1055532}

\bibitem{Ceballos05}
Juan~Carlos Ceballos~V., Ricardo Pavez~F., and Octavio~Paulo
  Vera~Villagr\'{a}n, \emph{Exact boundary controllability for higher order
  nonlinear {S}chr\"{o}dinger equations with constant coefficients}, Electron.
  J. Differential Equations (2005), No. 122, 31. \MR{2174554}

\bibitem{cm2022}
Andreas Chatziafratis and Dionyssios Mantzavinos, \emph{Boundary behavior for
  the heat equation on the half-line}, Math. Methods Appl. Sci. \textbf{45}
  (2022), no.~12, 7364--7393. \MR{4456042}

\bibitem{chen2018}
Mo~Chen, \emph{Stabilization of the higher order nonlinear {S}chr\"{o}dinger
  equation with constant coefficients}, Proc. Indian Acad. Sci. Math. Sci.
  \textbf{128} (2018), no.~3, Art. 39, 15. \MR{3814780}

\bibitem{Christ91}
F.~M. Christ and M.~I. Weinstein, \emph{Dispersion of small amplitude solutions
  of the generalized {K}orteweg-de {V}ries equation}, J. Funct. Anal.
  \textbf{100} (1991), no.~1, 87--109. \MR{1124294}

\bibitem{ck2002}
J.~E. Colliander and C.~E. Kenig, \emph{The generalized {K}orteweg-de {V}ries
  equation on the half line}, Comm. Partial Differential Equations \textbf{27}
  (2002), no.~11-12, 2187--2266. \MR{1944029}

\bibitem{ct2017}
E.~Compaan and N.~Tzirakis, \emph{Well-posedness and nonlinear smoothing for
  the ``good'' {B}oussinesq equation on the half-line}, J. Differential
  Equations \textbf{262} (2017), no.~12, 5824--5859. \MR{3624540}

\bibitem{et2016}
M.~B. Erdo\u{g}an and N.~Tzirakis, \emph{Regularity properties of the cubic
  nonlinear {S}chr\"{o}dinger equation on the half line}, J. Funct. Anal.
  \textbf{271} (2016), no.~9, 2539--2568. \MR{3545224}

\bibitem{Fam22}
Andrei~V. Faminskii, \emph{The higher order nonlinear {S}chr\"{o}dinger
  equation with quadratic nonlinearity on the real axis}, Adv. Differential
  Equations \textbf{28} (2023), no.~5-6, 413--466. \MR{4555001}

\bibitem{Fokas97}
A.~S. Fokas, \emph{A unified transform method for solving linear and certain
  nonlinear {PDE}s}, Proc. Roy. Soc. London Ser. A \textbf{453} (1997),
  no.~1962, 1411--1443. \MR{1469927}

\bibitem{Fokasbook}
Athanassios~S. Fokas, \emph{A unified approach to boundary value problems},
  CBMS-NSF Regional Conference Series in Applied Mathematics, vol.~78, Society
  for Industrial and Applied Mathematics (SIAM), Philadelphia, PA, 2008.
  \MR{2451953}

\bibitem{fhm2016}
Athanassios~S. Fokas, A.~Alexandrou Himonas, and Dionyssios Mantzavinos,
  \emph{The {K}orteweg--de {V}ries equation on the half-line}, Nonlinearity
  \textbf{29} (2016), no.~2, 489--527. \MR{3461607}

\bibitem{fhm2017}
\bysame, \emph{The nonlinear {S}chr\"{o}dinger equation on the half-line},
  Trans. Amer. Math. Soc. \textbf{369} (2017), no.~1, 681--709. \MR{3557790}

\bibitem{g2014c}
Loukas Grafakos, \emph{Classical {F}ourier analysis}, third ed., Graduate Texts
  in Mathematics, vol. 249, Springer, New York, 2014. \MR{3243734}

\bibitem{g2014m}
\bysame, \emph{Modern {F}ourier analysis}, third ed., Graduate Texts in
  Mathematics, vol. 250, Springer, New York, 2014. \MR{3243741}

\bibitem{gw2021}
Boling Guo and Jun Wu, \emph{Well-posedness of the initial-boundary value
  problem for the {H}irota equation on the half line}, J. Math. Anal. Appl.
  \textbf{504} (2021), no.~2, Paper No. 125571, 25. \MR{4302672}

\bibitem{Hay21D}
Nakao Hayashi, Elena~I. Kaikina, and Takayoshi Ogawa, \emph{Inhomogeneous
  {D}irichlet boundary value problem for nonlinear {S}chr\"{o}dinger equations
  in the upper half-space}, Partial Differ. Equ. Appl. \textbf{2} (2021),
  no.~6, Paper No. 69, 24. \MR{4338033}

\bibitem{Hay21N}
\bysame, \emph{Inhomogeneous {N}eumann-boundary value problem for nonlinear
  {S}chr\"{o}dinger equations in the upper half-space}, Differential Integral
  Equations \textbf{34} (2021), no.~11-12, 641--674. \MR{4335244}

\bibitem{hy2022b}
A.~Alexanddrou Himonas and Fangchi Yan, \emph{A higher dispersion {K}d{V}
  equation on the half-line}, J. Differential Equations \textbf{333} (2022),
  55--102. \MR{4441362}

\bibitem{hmy2021}
A.~Alexandrou Himonas, Carlos Madrid, and Fangchi Yan, \emph{The {N}eumann and
  {R}obin problems for the {K}orteweg--de {V}ries equation on the half-line},
  J. Math. Phys. \textbf{62} (2021), no.~11, Paper No. 111503, 24. \MR{4334475}

\bibitem{hm2015}
A.~Alexandrou Himonas and Dionyssios Mantzavinos, \emph{The ``good''
  {B}oussinesq equation on the half-line}, J. Differential Equations
  \textbf{258} (2015), no.~9, 3107--3160. \MR{3317631}

\bibitem{hm2015l}
\bysame, \emph{On the initial-boundary value problem for the linearized
  {B}oussinesq equation}, Stud. Appl. Math. \textbf{134} (2015), no.~1,
  62--100. \MR{3298877}

\bibitem{hm2020}
\bysame, \emph{Well-posedness of the nonlinear {S}chr\"{o}dinger equation on
  the half-plane}, Nonlinearity \textbf{33} (2020), no.~10, 5567--5609.
  \MR{4151418}

\bibitem{hm2021}
\bysame, \emph{The nonlinear {S}chr\"{o}dinger equation on the half-line with a
  {R}obin boundary condition}, Anal. Math. Phys. \textbf{11} (2021), no.~4,
  Paper No. 157, 25. \MR{4303633}

\bibitem{hm2022}
\bysame, \emph{The {R}obin and {N}eumann problems for the nonlinear
  {S}chr\"{o}dinger equation on the half-plane}, Proc. A. \textbf{478} (2022),
  no.~2265, Paper No. 279, 20. \MR{4492208}

\bibitem{hmy2019-rd}
A.~Alexandrou Himonas, Dionyssios Mantzavinos, and Fangchi Yan,
  \emph{Initial-boundary value problems for a reaction-diffusion equation}, J.
  Math. Phys. \textbf{60} (2019), no.~8, 081509, 19. \MR{3996714}

\bibitem{hmy2019-kdv}
\bysame, \emph{The {K}orteweg--de {V}ries equation on an interval}, J. Math.
  Phys. \textbf{60} (2019), no.~5, 051507, 26. \MR{3947621}

\bibitem{hmy2019-nls}
\bysame, \emph{The nonlinear {S}chr\"{o}dinger equation on the half-line with
  {N}eumann boundary conditions}, Appl. Numer. Math. \textbf{141} (2019),
  2--18. \MR{3944685}

\bibitem{hy2022a}
A.~Alexandrou Himonas and Fangchi Yan, \emph{The {K}orteweg--de {V}ries
  equation on the half-line with {R}obin and {N}eumann data in low regularity
  spaces}, Nonlinear Anal. \textbf{222} (2022), Paper No. 113008, 31.
  \MR{4432354}

\bibitem{h2005}
Justin Holmer, \emph{The initial-boundary-value problem for the 1{D} nonlinear
  {S}chr\"{o}dinger equation on the half-line}, Differential Integral Equations
  \textbf{18} (2005), no.~6, 647--668. \MR{2136703}

\bibitem{h2006}
\bysame, \emph{The initial-boundary value problem for the {K}orteweg-de {V}ries
  equation}, Comm. Partial Differential Equations \textbf{31} (2006), no.~7-9,
  1151--1190. \MR{2254610}

\bibitem{h2020}
Lin Huang, \emph{The initial-boundary-value problems for the {H}irota equation
  on the half-line}, Chinese Ann. Math. Ser. A \textbf{41} (2020), no.~1,
  117--132. \MR{4264317}

\bibitem{KalOz2020}
Konstantinos Kalimeris and T\"{u}rker \"{O}zsar\i, \emph{An elementary proof of
  the lack of null controllability for the heat equation on the half line},
  Appl. Math. Lett. \textbf{104} (2020), 106241, 6. \MR{4056215}

\bibitem{KalOz2023}
Konstantinos Kalimeris, T\"{u}rker \"{O}zsar\i, and Nicholas Dikaios,
  \emph{Numerical computation of {N}eumann controls for the heat equation on a
  finite interval}, IEEE Trans. Automat. Control (forthcoming).

\bibitem{staffilani2}
Carlos~E. Kenig and Gigliola Staffilani, \emph{Local well-posedness for higher
  order nonlinear dispersive systems}, J. Fourier Anal. Appl. \textbf{3}
  (1997), no.~4, 417--433. \MR{1468372}

\bibitem{KO22}
Bilge Köksal and T\"{u}rker \"{O}zsar{\i}, \emph{The interior-boundary
  strichartz estimate for the schrödinger equation on the half line
  revisited}, Turkish J. Math. \textbf{46} (2022), no.~8, 3323–3351.

\bibitem{koda85}
Yuji Kodama, \emph{Optical solitons in a monomode fiber}, vol.~39, 1985,
  Transport and propagation in nonlinear systems (Los Alamos, N.M., 1984),
  pp.~597--614. \MR{807002}

\bibitem{koda87}
Yuji Kodama and Akira Hasegawa, \emph{Nonlinear pulse propagation in a monomode
  dielectric guide}, IEEE Journal of Quantum Electronics \textbf{23} (1987),
  no.~5, 510--524.

\bibitem{Laurey97}
Corinne Laurey, \emph{The {C}auchy problem for a third order nonlinear
  {S}chr\"{o}dinger equation}, Nonlinear Anal. \textbf{29} (1997), no.~2,
  121--158. \MR{1446222}

\bibitem{lm1972}
J.-L. Lions and E.~Magenes, \emph{Non-homogeneous boundary value problems and
  applications. {V}ol. {I}}, Die Grundlehren der mathematischen Wissenschaften,
  Band 181, Springer-Verlag, New York-Heidelberg, 1972, Translated from the
  French by P. Kenneth. \MR{0350177}

\bibitem{Caval19}
Mauricio A. Sepulveda~C. Marcelo M.~Cavalcanti, Wellington J.~Correa and
  Rodrigo~Vejar Asem, \emph{Finite difference scheme for a high order nonlinear
  {S}chrödinger equation with localized damping}, Stud. Univ. Babes-Bolyai
  Math. \textbf{64} (2019), no.~2, 161--172.

\bibitem{OAK22}
T\"{u}rker \"{O}zsar\i, K\i v\i lc\i~m Alkan, and Konstantinos Kalimeris,
  \emph{Dispersion estimates for the boundary integral operator associated with
  the fourth order {S}chr\"{o}dinger equation posed on the half line}, Math.
  Inequal. Appl. \textbf{25} (2022), no.~2, 551--571. \MR{4428587}

\bibitem{KalOz2023-2}
T\"{u}rker \"{O}zsar\i and Konstantinos Kalimeris, \emph{Existence of
  unattainable states for {S}chrödinger type flows on the half-line},
  unpublished (forthcoming).

\bibitem{OY2022}
T\"{u}rker \"{O}zsar{\i} and Kemal~Cem Y\i~lmaz, \emph{Stabilization of higher
  order {S}chr\"{o}dinger equations on a finite interval: {P}art {II}}, Evol.
  Equ. Control Theory \textbf{11} (2022), no.~4, 1087--1148. \MR{4455286}

\bibitem{OY19}
T\"{u}rker \"{O}zsar{\i} and Nermin Yolcu, \emph{The initial-boundary value
  problem for the biharmonic {S}chr\"{o}dinger equation on the half-line},
  Commun. Pure Appl. Anal. \textbf{18} (2019), no.~6, 3285--3316. \MR{3985385}

\bibitem{Ran18}
Yu~Ran, Shu-Ming Sun, and Bing-Yu Zhang, \emph{Nonhomogeneous boundary value
  problems of nonlinear {S}chr\"{o}dinger equations in a half plane}, SIAM J.
  Math. Anal. \textbf{50} (2018), no.~3, 2773--2806. \MR{3809532}

\bibitem{staffilani1}
Gigliola Staffilani, \emph{On the generalized {K}orteweg-de {V}ries-type
  equations}, Differential Integral Equations \textbf{10} (1997), no.~4,
  777--796. \MR{1741772}

\bibitem{stein1986}
E.~M. Stein, \emph{Oscillatory integrals in {F}ourier analysis}, Beijing
  lectures in harmonic analysis ({B}eijing, 1984), Ann. of Math. Stud., vol.
  112, Princeton Univ. Press, Princeton, NJ, 1986, pp.~307--355. \MR{864375}

\bibitem{s1970}
Elias~M. Stein, \emph{Singular integrals and differentiability properties of
  functions}, Princeton Mathematical Series, No. 30, Princeton University
  Press, Princeton, N.J., 1970. \MR{0290095}

\bibitem{s1977}
Robert Strichartz, \emph{Restrictions of {F}ourier transforms to quadratic
  surfaces and decay of solutions of wave equations}, Duke Math. J. \textbf{44}
  (1977), 705--714.

\bibitem{Taka00}
Hideo Takaoka, \emph{Well-posedness for the higher order nonlinear
  {S}chr\"{o}dinger equation}, Adv. Math. Sci. Appl. \textbf{10} (2000), no.~1,
  149--171. \MR{1769176}

\bibitem{gw2023}
Jun Wu and Boling Guo, \emph{Initial-boundary value problem for the {H}irota
  equation posed on a finite interval}, J. Math. Anal. Appl. \textbf{526}
  (2023), no.~2, Paper No. 127330. \MR{4582106}

\bibitem{zs1972}
V.~E. Zakharov and A.~B. Shabat, \emph{Exact theory of two-dimensional
  self-focusing and one-dimensional self-modulation of waves in nonlinear
  media}, \v{Z}. \`Eksper. Teoret. Fiz. \textbf{61} (1971), no.~1, 118--134.
  \MR{0406174}

\end{thebibliography}

\end{document}